\documentclass[reqno]{amsart}
\pdfoutput=1
\usepackage{amssymb}
\usepackage[dvips]{epsfig}
\usepackage{graphicx}
\usepackage{color}
\usepackage{amsmath}
\usepackage{amssymb}
\usepackage{amsfonts}
\usepackage{amsthm}
\usepackage{bbm}
\usepackage{tikz}
\usepackage[colorlinks=true,linkcolor=blue,citecolor=brown]{hyperref}

\newtheorem{thm}{Theorem}[section]
\newtheorem{lem}[thm]{Lemma}

\theoremstyle{remark}
\newtheorem{rem}{\bf Remark}[section]
\theoremstyle{definition}
\newtheorem{defn}[thm]{Definition}

\numberwithin{equation}{section}

\begin{document}
\title[The Existence of Full Dimensional KAM tori for d-dimensional NLS]{The Existence of full dimensional tori for d-dimensional Nonlinear Schr$\ddot{\mbox{O}}$dinger equation}

\author[H. Cong]{Hongzi Cong}
\address[H. Cong]{School of Mathematical Sciences, Dalian University of Technology, Dalian 116024, China}
\email{conghongzi@dlut.edu.cn}
\author[X.Wu]{Xiaoqing Wu}
\address[X. Wu]{School of Mathematical Sciences, Dalian University of Technology, Dalian 116024, China}
\email{wuxiaoqing0216@163.com}
\author[Y.Wu]{Yuan Wu}
\address[Y. Wu]{School of Mathematics and Statistics,
Huazhong University of Science and Technology, Wuhan 430070, China} \email{wuyuan@hust.edu.cn}

\date{\today}



\begin{abstract}
In this paper, we prove the existence of full dimensional tori for $d$-dimensional nonlinear Schr$\ddot{\mbox{o}}$dinger equation  with periodic boundary conditions
\begin{equation*}\label{L1}
\sqrt{-1}u_{t}+\Delta u+V*u\pm\epsilon |u|^2u=0,\hspace{12pt}x\in\mathbb{T}^d,\quad d\geq 1,
\end{equation*}
where $V*$ is the convolution potential. Here the radius of the invariant torus satisfies a slower decay, i.e.
\begin{equation*}\label{031601}
I_{\textbf n}\sim e^{-r\ln^{\sigma}\left\|\textbf n\right\|},\qquad \mbox{as}\ \left\|\textbf
n\right\|\rightarrow\infty,
\end{equation*}for any $\sigma>2$ and $r\geq 1$.
This result confirms a conjecture by Bourgain [J. Funct. Anal. 229 (2005), no. 1, 62-94].

\end{abstract}

\maketitle

\section{Introduction and main results}

It is an elementary problem to understand qualitative certain aspects of the long time behavior of solutions of Hamiltonian partial differential equations (PDEs), such as Nekhoroshev stability under an $\varepsilon$-perturbation, possible growth of higher Sobolev norm of classical solutions for $t\rightarrow \infty$ and  the existence and abundance of quasi-periodic and almost periodic motion in phase space.

An open problem was raised by Kuksin (see Problem 7.1 in \cite{Kuksin2004}):
 \vskip8pt
 {\it Can the full dimensional KAM tori be expected with a suitable decay for Hamiltonian partial differential equations, for example, $$I_n\sim e^{-C\ln |n|}$$ with some $ C>0$ as $|n|\rightarrow+\infty$? }

 In 2005, a pioneering work by Bourgain \cite{Bourgain2005JFA} was to construct the full dimensional tori, which are the support of almost periodic solutions, for 1-dimensional nonlinear Schr\"{o}dinger
equation (NLS)
\begin{equation*}
\sqrt{-1}u_t-u_{xx}+V*u+\epsilon|u|^4u=0,\hspace{12pt}x\in\mathbb{T},
\end{equation*}
where the radius of the invariant torus satisfies
\begin{equation}\label{031701}
I_n\sim e^{-\sqrt{|n|}},\qquad \mbox{as}\ |n|\rightarrow \infty.
\end{equation}
At the same time Bourgain pointed out that: \textbf{we do not know at this time how to prove a 2D-analogue of Theorem 1, considering for instance the cubic NLS.}

Our basic motivation is to prove Bourgain's conjecture that the existence of full dimensional tori for $d$-dimensional NLS
\begin{equation}\label{061510}
\sqrt{-1}u_{t}+\Delta u+V*u\pm\epsilon |u|^2u=0,\hspace{12pt}x\in\mathbb{T}^d,\quad d\geq 1.
\end{equation}
Here $V*$ is the convolution potential defined by
\begin{equation*}
\widehat{V*u}(\textbf n)=V_{\textbf n}\widehat u(\textbf n)
\end{equation*}
and $\widehat u(\textbf n)$ is the $\textbf n$-th Fourier coefficient of $u$. Furthermore, using some techniques in \cite{cong2021}, we can obtain a better result, i.e. the radius of the full dimensional tori satisfies a much slower decay than (\ref{031701}), i.e.
 \begin{equation*}\label{031601}
I_{\textbf n}\sim e^{-r \ln^{\sigma}\left\|\textbf n\right\|},\qquad \mbox{as}\ \left\|\textbf n\right\|\rightarrow\infty,
\end{equation*}with any $\sigma>2,r\geq 1$ for equation (\ref{061510}), which is closer to the conjecture by Kuksin.

To state our result, we will introduce a stronger Diophantine condition than the one given in \cite{Bourgain2005JFA} firstly.
For any vector $0\neq \ell \in\mathbb{Z}^{\mathbb{Z}^d}$ with
$$|\ell|:=\sum_{\textbf n\in\mathbb{Z}^d}|\ell_{\textbf n}|<\infty,$$ define the system by
\begin{equation}\label{040603}
\textbf n(\ell):=\left\{\textbf n\in\mathbb{Z}^d:\ \mbox{which is repeated $|\ell_{\textbf n}|$ times}\right\}.
\end{equation}
Also we can write
 \begin{equation*}\textbf n(\ell)=\left(\textbf n_{i}^{*}\left(\ell\right)\right)_{1\leq i\leq |\ell|},\qquad \textbf{n}_i^{*}\left(\ell\right)\in\mathbb{Z}^d,
\end{equation*}
which satisfies
\begin{equation*}
\left\|\textbf n_1^{*}\left(\ell\right)\right\|\geq \dots\geq \left\|\textbf n_{|\ell|}^{*}\left(\ell\right)\right\|,
\end{equation*}
where $\left\|\cdot\right\|$ denotes the the usual Euclidean norm by
\begin{equation*}
\left\|\textbf n\right\|:=\sqrt{\sum_{i=1}^d|n_i|^2},\qquad \textbf{n}=(n_i)_{1\leq i\leq d}\in\mathbb{Z}^d.
\end{equation*}
Let
\begin{equation}\label{050701}
\Pi:=\left\{\omega=\left(\omega_{\textbf n}\right)_{\textbf n\in\mathbb{Z}^d}:\omega_{\textbf n}\in\left[0,{\langle \textbf n\rangle}^{-1}\right]\right\},
\end{equation}
where
\begin{equation*}
\langle \textbf n\rangle =\max\left\{1,\left\|\textbf n\right\|\right\}.
\end{equation*}
Then we say a vector $\omega\in \Pi$ is strong Diophantine, if  there exists a real number $0<\gamma<1$ such that both of the following inequalities hold true:\\
(1) for any $0\neq \ell\in\mathbb{Z}^{\mathbb{Z}^d}$ with $|\ell|<\infty$, one has
\begin{equation}\label{040601}
\left|\left|\left|\sum_{\textbf n\in\mathbb{Z}^d}\ell_{\textbf n}\omega_{\textbf n}\right|\right|\right|\geq \gamma\prod_{\textbf n\in\mathbb{Z}^d}\frac{1}{1+|\ell_{\textbf n}|^3\langle \textbf n\rangle^{d+4}};
\end{equation}
(2) for any $0\neq \ell\in\mathbb{Z}^{\mathbb{Z}^d}$ with $|\ell|<\infty$ and  $ \left\|{\textbf n}_3^{*}(\ell)\right\|< \left\|{\textbf n}_2^{*}(\ell)\right\|$, one has
\begin{equation}\label{040602}
\left|\left|\left|\sum_{\textbf n\in\mathbb{Z}^d}\ell_{\textbf n}\omega_{\textbf n}\right|\right|\right|\geq\frac{\gamma^5}{100}\prod_{\textbf n\in\mathbb{Z}^d\atop \left\|\textbf n\right\|\leq \left\|{\textbf n}_3^{*}(\ell)\right\|}\left(\frac{1}{1+|\ell_{\textbf n}|^3\langle \textbf n\rangle^{d+7}}\right)^{10},
\end{equation}
where
\begin{equation*}
\left|\left|\left|x\right|\right|\right|=\inf_{j\in\mathbb{Z}}|x-j|.
\end{equation*}

Now  our main result is as follows:
\begin{thm}\label{031930}
Given any $\sigma>2,r\geq 1$ and any frequency vector $\omega\in\Pi$ satisfying the strong Diophantine conditions (\ref{040601}) and (\ref{040602}), then there exists a small $\epsilon_*(\sigma,r,\gamma)>0$ depending on $\sigma,r$ and $\gamma$ only, and  for any $0<\epsilon<\epsilon_*(\sigma,r,\gamma)$, there exist some $V\in[0,1]^{\mathbb{Z}^d}$ and a constant $\xi\in\mathbb{R}$ such that equation (\ref{061510}) has a full dimensional invariant torus $\mathcal{E}$  satisfying:\\
(1). the amplitude $I=(I_{\textbf n})_{\textbf n\in\mathbb{Z}^d}$ of $\mathcal{E}$ restricted as
\begin{equation*}
c_1\epsilon^2e^{-2r\ln^{\sigma}\langle\textbf n\rangle}\leq |I_{\textbf n}|\leq C_1\epsilon^2e^{-2r\ln^{\sigma}\langle \textbf n\rangle},\qquad \forall\ \textbf n\in\mathbb{Z}^d,
\end{equation*}
where $c_1<C_1$ are two positive constants depending on $\sigma$ and $r$ only; \\
(2). the frequency on $\mathcal{E}$ prescribed to be $\left(\left\|\textbf n\right\|^2+\xi+\omega_{\textbf n}\right)_{\textbf n\in\mathbb{Z}^d}$;\\
(3). the invariant torus $\mathcal{E}$ is linearly stable.
\end{thm}

Since 1990's, the KAM theory for infinite dimensional Hamiltonian system has been well developed to study the existence and linear stability of invariant tori for Hamiltonian PDEs. See \cite{BBM2014,BBM2016,BBHM2018,CY2000,CW1993,K1987,KP1996,K2000,Kuksin2004,LY2010,P'1996,P1996,W1990} for the related works for 1-dimensional PDEs. The solutions starting from the invariant torus stay on the torus all the time, which can be considered as permanent stability. For high dimensional PDEs, the situation becomes more complicated due to the multiple eigenvalues of Laplacian operator. Bourgain \cite{Bourgain1998Annals,Bourgain2005} developed a new method initialed by Craig-Wayne \cite{CW1993} to prove the existence of lower dimensional KAM tori for $d$-dimensional NLS and $d$-dimensional NLW with $d\geq 1$, based on the Newton iteration, Fr$\ddot{\mbox{o}}$hlich-Spencer techniques, Harmonic analysis and semi-algebraic set theory. This is so-called C-W-B method.
Later, Eliasson-Kuksin \cite{EK2010} obtained both the existence and the linear stability of KAM tori for $d$-dimensional NLS in a classical KAM way. Also see \cite{BB2013,BCM2015,BP2011,BWJEMS,GXY,PP2015,Wang2016Duke,Wang2020CMP} for example. Recently, an important progress is given by Baldi-Berti-Haus-Montalto \cite{BBHM2018}, where the existence and the linear stability of Cantor families of
small amplitude time quasi-periodic solutions for water wave equation are constructed. Here the main difficulties are the fully nonlinear nature
of the gravity water waves equations and the fact that
the linear frequencies grow just in a sublinear way at infinity. See \cite{BB2012,BCM2015,BM2020,BFM2021} for the related problem.

In the above works, the obtained KAM tori are of low (finite) dimension which are the support of the quasi-periodic solutions. Biasco-Massetti-Procesi \cite{BMP2021Poincare} pointed out that {\it the constructed quasi-periodic solutions are not typical in the sense that the low dimensional tori have measure zero for any reasonable measure on the infinite dimensional phase space}. It is natural at this point to find the full dimensional tori which are the support of the almost periodic solutions. The first result on the existence of almost periodic solutions for Hamiltonian PDEs was proved by Bourgain in \cite{Bourgain1996} using C-W-B method. Later, P$\ddot{\mbox{o}}$schel \cite{Poschel2002} (also see \cite{GX2013} by Geng-Xu) constructed the almost periodic solutions for 1-dimensional NLS by the classical KAM method. The basic idea to obtain these almost periodic solutions is by perturbing the quasi-periodic ones. That is why the action $I=(I_n)_{ n\in\mathbb{Z}}$ must satisfy some very strong compactness properties.

The first try to obtain the existence of full dimensional tori with a slower decay was given by Bourgain \cite{Bourgain2005JFA}, who proved that 1-dimensional NLS has a full dimensional KAM torus of prescribed frequencies with the actions of the tori obeying the estimate (\ref{031701}). {Different from \cite{Bourgain1996} and \cite{Poschel2002}, Bourgain \cite{Bourgain2005JFA} treated all Fourier modes at once, which caused a much worse small denominator problem. To this end, Bourgain took advantage of two key facts. Let $(n_i)_{i\geq1}$ be a finite set of modes, $|n_1|\geq |n_2|\geq \cdots$ and
\begin{equation}\label{060701}
n_1-n_2+n_3-\cdots=0.
\end{equation}Then
the following inequality holds
\begin{equation}\label{061501}
\sum_{i\geq1}\sqrt{|n_i|}-2\sqrt{|n_1|}\geq\frac14\sum_{i\geq 3}{\sqrt{|n_i|}}.
\end{equation}
Furthermore, there is also a relation
\begin{equation}\label{060702}
n_1^2-n_2^2+n_3^2-\cdots=o(1)
\end{equation}in  the case of a `near' resonance.
The conditions (\ref{060701}) and (\ref{060702}) implies that the first two biggest indices $|n_1|$ and $|n_2|$ can be controlled by other indices, i.e.
\begin{equation}\label{022004}|n_1|+|n_2|\leq C\left(|n_3|+|n_4|+\cdots\right),
 \end{equation} unless $n_1=n_2$. The inequalities (\ref{061501}) and (\ref{022004}) are essential to control the small divisor.

Recently, Cong-Liu-Shi-Yuan \cite{CLSY2018JDE} generalized Bourgain's result from $\theta=1/2$ to any $0<\theta<1$ in a classical KAM way, where the actions of the tori satisfying
\begin{equation*}\label{020203}I_n\sim e^{-{|n|^{\theta}}},\qquad \theta\in (0,1).
\end{equation*}
The authors also proved the obtained tori are stable in a sub-exponential long time. Another important progress is given by Biasco-Massetti-Procesi \cite{BMP2021Poincare}, who proved the existence
and linear stability of almost periodic solution for 1-dimensional NLS by constructing a
rather abstract counter-term theorem for infinite dimensional Hamiltonian system.

A problem raised by Bourgain in \cite{Bourgain2005JFA} is that: {\it do there exist the full dimensional tori for $2$-dimensional NLS which satisfy the estimate (\ref{031701})?}
As Bourgain pointed out that a main difficulty may be that $(\ref{060701})+ (\ref{060702})\Rightarrow (\ref{022004})$ holds true only in $1$-dimensional case. In other words, let $(\textbf n_i)_{i\geq 1}$ with $\textbf n_i\in\mathbb{Z}^d, d\geq 2$ satisfying $\left\|\textbf n_1\right\|\geq \left\|\textbf n_2\right\|\geq \cdots$. Assume that
\begin{equation*}
\textbf n_1-\textbf n_2+\textbf n_3-\cdots=0
\end{equation*}
and
\begin{equation*}
\left\|\textbf n_1\right\|^2-\left\|\textbf n_2\right\|^2+\left\|\textbf n_3\right\|^2-\cdots=o(1),
\end{equation*}one can not obtain the following inequality
\begin{equation}\label{061601}
\left\|\textbf n_1\right\|+\left\|\textbf n_2\right\|\leq C\left(\left\|\textbf n_3\right\|+\left\|\textbf n_4\right\|+\cdots\right)
\end{equation}even if $\textbf n_1\neq \textbf n_2$. As mentioned earlier, the inequality (\ref{061601}) (also see (\ref{022004}) for $1$-dimensional case) is important to obtain a suitable bound of the solution of homological equation. To overcome such a difficulty, we firstly observe that in the case $\left\|{\textbf n}_3^{*}(\ell)\right\|=\left\|{\textbf n}_2^{*}(\ell)\right\|$, combining with {momentum conservation (\ref{050901}), the inequality (\ref{061601}) still holds true. If $ \left\|{\textbf n}_3^{*}(\ell)\right\|< \left\|{\textbf n}_2^{*}(\ell)\right\|$, we will introduce the nonresonant conditions
(\ref{040602}), which seem like the second Melnikov conditions when dealing with the existence of lower dimensional tori. Note that the righthand of (\ref{040602}) only contains the terms $\left(\left\|\textbf n_i^*(\ell)\right\|\right)_{i\geq 3}$. We will prove most of $\omega$ satisfy the nonresonant conditions (\ref{040601}) and (\ref{040602}) in Lemma \ref{050601}.

Recently, Cong \cite{cong2021} improved the previous results in the sense that
the action of the obtained full dimensional tori for 1-dimensional NLS satisfies
\begin{equation}\label{032301.}I_n\sim e^{-r\ln^{\sigma}|n|},\qquad n\rightarrow \infty
\end{equation}with any $\sigma>2$ and $r\geq 1$. Using some techniques in \cite{cong2021}, we prove the existence of full dimensional tori for high dimensional NLS with a better estimate than Bourgain conjectured in \cite{Bourgain2005JFA}.

Finally, we will give some more remarks.
\begin{rem}When infinite systems of coupled harmonic oscillators with finite-range couplings are considered, P$\ddot{\mbox{o}}$schel in \cite{Poschel1990} proved the existence of full dimensional tori for infinite dimensional Hamiltonian system with spatial structure of short range couplings consisting of
connected sets only, where the action satisfies (\ref{032301.}).
 In particular, in the simplest case, i.e. only nearest neighbour coupling, the result can be optimized for any $\sigma>1$. Of course, Hamiltonian PDEs are not of short range, and do not contain such spatial structure. So our result may be optimal.
\end{rem}
\begin{rem}
The Momentum Conservation (\ref{050901}) is very important in our paper. One reason is that the condition (\ref{050901}) guarantees there are no small divisor  when $|k|+|k'|=2$. The other one is that we have to control the frequency shift by so-called $\rho$-T\"{o}plitz-Lipschitz property for the Hamiltonian (see Definition \ref{061502}), which generalizes the idea given by Geng-Xu-You in \cite{GXY}.
Hence our method can not be applied to $d$-dimensional nonlinear wave equation directly even to $1$-dimensional nonlinear wave equation under periodic boundary conditions.
\end{rem}
\begin{rem}
Recently, there are some important progresses on long time stability and possible growth of Sobolev norm for high dimensional Hamiltonian PDEs on irrational tori. See (\cite{Berti2019JDE,Deng2019,DG2019,Imekraz2016}) for example.
It is natural to ask whether our method can be applied to this case.
\end{rem}
\begin{rem}
An interesting result was proven by Biasco-Massetti-Procesi in \cite{BMP2021Poincare}.
The authors constructed almost periodic solutions of 1-dimensional NLS, which have Sobolev regularity both in time and space. This is the first result of this kind in KAM theory for PDEs.
\end{rem}

\section{The Norm of the Hamiltonian}
In this section, we will introduce some notations and definitions firstly.

Fixed $\sigma>2$ and given any $\rho\geq0$, define the Banach space $\mathfrak{H}_{\sigma,\rho}$ of all complex-valued sequences
$q=(q_{{\textbf n}})_{{\textbf n}\in\mathbb{Z}^d}$ with
\begin{equation*}\label{042501}
\left\|q\right\|_{\sigma,\rho}=\sup_{{\textbf n}\in\mathbb{Z}^d}\left|q_{\textbf n}\right| e^{\rho\ln^{\sigma}{\lfloor{\textbf n}\rfloor}}<\infty,
\end{equation*}
where
\begin{equation*}
\lfloor \textbf n\rfloor=\max\left\{2^{10},\left\|\textbf n\right\|\right\}.
\end{equation*}
\begin{rem}
Define another norm $\left\|\cdot\right\|_{\sigma,\rho}'$ by
\begin{equation*}
\left\|q\right\|'_{\sigma,\rho}=\sup_{{\textbf n}\in\mathbb{Z}^d}\left|q_{\textbf n}\right| e^{\rho\ln^{\sigma}{\langle\textbf n\rangle}}.
\end{equation*}
Then it is easy to see that the norms $\left\|\cdot\right\|_{\sigma,\rho}$ and $\left\|\cdot\right\|_{\sigma,\rho}'$ are equivalent. The reason why we use the norm $\left\|\cdot\right\|$ is that some estimates are easy to obtain. See the proof of Lemma \ref{122203} for example.
\end{rem}
Consider the Hamiltonian $R(q,\bar q)$ with the form of
\begin{equation}\label{052301}
R(q,\bar q)=\sum_{a,k,k'\in\mathbb{N}^{\mathbb{Z}^d}}R_{akk'}\mathcal{M}_{akk'},
\end{equation}
where
\begin{equation*}\label{030502}
\mathcal{M}_{akk'}=\prod_{{\textbf n}\in\mathbb{Z}^d}I_{\textbf n}(0)^{a_{\textbf n}}q_{\textbf n}^{k_{\textbf n}}\bar q_{\textbf n}^{k_{\textbf n}'}
\end{equation*}is the so-called monomial, $R_{akk'}$ is the corresponding coefficient and $I_{\textbf n}(0)$ is considered as the initial data.
Given a monomial $\mathcal{M}_{akk'}$,
define by
\begin{equation*}
\mbox{supp}\ \mathcal{M}_{akk'}=\mbox{supp}\ (a,k,k'):=\{{\textbf n}\in\mathbb{Z}^d:a_{{\textbf n}}+k_{\textbf n}+k_{\textbf n}'\neq 0\}.
\end{equation*}
Furthermore, we always assume that each monomial $\mathcal{M}_{akk'}$ in $R(q,\bar q)$ satisfies:\\
(1) \textbf{Mass Conservation}
\begin{equation}\label{051702}
\sum_{\textbf n\in\mathbb{Z}^d}(k_{\textbf n}-k'_{\textbf n})=0;
\end{equation}
(2) \textbf{Momentum Conservation}
\begin{equation}\label{050901}
\sum_{\textbf n\in\mathbb{Z}^d}(k_{\textbf n}-k'_{\textbf n})\textbf{n}=0.
\end{equation}
Similar as (\ref{040603}), we define the system by
\begin{equation*}\label{040604}
\textbf{n}(a,k,k'):=\left\{\textbf n\in\mathbb{Z}^d:\ \mbox{which is repeated $2|a_{\textbf n}|+|k_{\textbf n}|+|k'_{\textbf n}|$ times}\right\},
\end{equation*}
and write
 \begin{equation}\label{053101}
\textbf n(a,k,k')=\left(\textbf n_{i}^{*}(a,k,k')\right)_{i\geq 1},\qquad \textbf{n}_i^{*}(a,k,k')\in\mathbb{Z}^d,
\end{equation} which satisfies
\begin{equation*}
\left\|\textbf n_1^{*}(a,k,k')\right\|\geq \dots\geq\left\|\textbf n_{m}^{*}(a,k,k')\right\|
\end{equation*}
and  $m=2|a|+|k|+|k'|$.

\begin{rem}
Here we always assume that $$2|a|+|k|+|k'|\geq 4,$$
since the cubic NLS is considered.
\end{rem}
Before defining the norm of the Hamiltonian $R(q,\bar q)$, we introduce the following lemma firstly:
\begin{lem}\label{005}
Fixing $\sigma>2$ and given any $a,k,k'\in\mathbb{N}^{\mathbb{Z}^d}$, assume {Momentum Conservation} (\ref{050901}) is satisfied.
Then one has
\begin{equation}\label{001}
\sum_{ {\textbf n}\in\mathbb{Z}^d}\left(2a_{\textbf n}+k_{\textbf n}+k_ {\textbf n}'\right){  \ln^{\sigma}\lfloor{\textbf n}\rfloor}-2{\ln^{\sigma}\lfloor\textbf n_1^{*}(a,k,k')\rfloor}\geq \frac12\sum_{i\geq 3}{\ln^{\sigma}\lfloor {\textbf n}_i^{*}(a,k,k')\rfloor}.
\end{equation}
\end{lem}
\begin{proof}
Based on  {Momentum Conservation} (\ref{050901}) and the triangle inequality, one has
\begin{equation}\label{51503}
\lfloor \textbf n^*_1(a,k,k')\rfloor \leq\sum_{i\geq 2} \lfloor \textbf n^*_i(a,k,k')\rfloor.
\end{equation}
Note that
\begin{equation}\label{0528001}
\sum_{ {\textbf n}\in\mathbb{Z}^d}\left(2a_{\textbf n}+k_{\textbf n}+k_ {\textbf n}'\right){  \ln^{\sigma}\lfloor{\textbf n}\rfloor}=\sum_{i\geq 1}{\ln^{\sigma}\lfloor {\textbf n}_i^{*}(a,k,k')\rfloor}.
\end{equation}
Using (\ref{51503}) and (\ref{0528001}), the inequality (\ref{001}) follows from
\begin{equation}\label{51506}
\sum_{i\geq 2} \ln^{\sigma}\lfloor \textbf n^*_i(a,k,k')\rfloor-\ln^{\sigma}\left(\sum_{i\geq 2} \lfloor \textbf n^*_i(a,k,k')\rfloor\right)\geq \frac12\sum_{i\geq 3} \ln^{\sigma}\lfloor \textbf n^*_i(a,k,k')\rfloor.
\end{equation}
In view of (\ref{51504}) in Remark \ref{51505} and by induction, we finish the proof of (\ref{51506}).
\end{proof}

\begin{defn}\label{083103}
Consider the Hamiltonian $R(q,\bar q)$ with the form of (\ref{052301}). Fixed $\sigma>2,r\geq 1$ and given any $\rho\geq 0$, define
\begin{equation}\label{042602}
\left\|R\right\|_{\sigma,\rho}=\sup_{a,k,k'\in\mathbb{N}^{\mathbb{Z}^d}}\frac{\left|R_{akk'}\right|}{e^{\rho\left(\sum_{\textbf n\in\mathbb{Z}^d}\left(2a_{\textbf n}+k_{\textbf n}+k_{\textbf n}'\right) \ln^{\sigma}\lfloor{\textbf n}\rfloor-2\ln^{\sigma}\lfloor{\textbf n}_1^{*}(a,k,k')\rfloor\right)}}
\end{equation}
and
\begin{equation}\label{d2}
\left\|R\right\|^*_{\sigma,r,\rho}:=\sum_{a,k,k'\in\mathbb{N}^{\mathbb{Z}^d}}\left|R_{akk'}\right|e^{-2r\sum_{\textbf n\in\mathbb{Z}^d}a_{\textbf n}\ln^{\sigma}\lfloor\textbf n\rfloor}\cdot e^{-\rho\sum_{\textbf n\in\mathbb{Z}^d}\left(k_{\textbf n}+k'_{\textbf n}\right)\ln^{\sigma}\lfloor\textbf n\rfloor}.
\end{equation}
\end{defn}
\begin{rem}
The indices $\sigma>2$ and $r\geq1$ will be fixed during the KAM iteration while the index $\rho$ will change. For simplicity, denote by
\begin{equation*}
\left\|\cdot\right\|_{\sigma,\rho}:=\left\|\cdot\right\|_{\rho}
\end{equation*}
and
\begin{equation*}
\left\|\cdot\right\|^*_{\sigma,r,\rho}:=\left\|\cdot\right\|^*_{\rho}.
\end{equation*}
In fact, in the $s$-th KAM iterative step (also in some technical lemmas), one can consider $\sigma=2.001, r=1$ for example and $\rho_s$ will be given at the beginning of Subsection \ref{031501}.
\end{rem}
\begin{rem}
In view of (\ref{d2}), it is easy to show that
\begin{equation}\label{041510}
\left\|R\right\|^*_{\rho}:=\sup_{\left\|q\right\|_{\sigma,\rho,\infty}\leq 1}\sum_{a,k,k'\in\mathbb{N}^{\mathbb{Z}^d}}\left|R_{akk'}\mathcal{M}_{akk'}\right|,
\end{equation}
with
\begin{equation}\label{M13}
I_{\textbf n}(0):=e^{-2r\ln^{\sigma}\lfloor\textbf n\rfloor}.
\end{equation}
\end{rem}
Therefore we have
\begin{lem}\label{l2}
Given two Hamiltonian  $ F,G$,\ then  the following estimate holds true
\begin{eqnarray}\label{002}
\|F\cdot G\|^{*}_{\rho}\leq \|F\|^{*}_{\rho}\cdot\|G\|^{*}_{\rho}.
\end{eqnarray}
\end{lem}
\begin{proof}
The inequality (\ref{002}) deduces from (\ref{041510}) directly.
\end{proof}
\begin{lem}\label{060207}
Given any $\rho,\delta\geq 0$, then one has
\begin{equation}\label{060208}
\left\|R\right\|_{\rho+\delta}^*\leq \left\|R\right\|_{\rho}^*.
\end{equation}
\end{lem}
\begin{proof}
The inequality (\ref{060208}) follows from (\ref{d2}) in Definition \ref{083103}.
\end{proof}
\begin{lem}(\textbf{Hamiltonian Vector Field})\label{063004}
Given $\sigma>2,r\geq 1,\rho\geq 0,\delta>0$ satisfying $\rho+\delta<r$, assume the Hamiltonian $R(q,\bar q)$ with the form of (\ref{052301}) satisfying $\left\|R\right\|_{\rho} < \infty$.
 Then one has
\begin{equation}\label{050907.}
\sup_{\left\|q\right\|_{\sigma,\rho+\delta}\leq1}\left\|X_R\right\|_{\sigma,\rho+\delta}\leq \exp\left\{10d\left(\frac{2000d}{\delta^2}\right)^d\cdot \exp\left\{d{\left(\frac{10d}{\delta} \right)}^{\frac{1}{\sigma-1}}\right\}\right\}\left\|R\right\|_{\rho}.
\end{equation}\color{black}
In particular, if $\delta\geq\frac r2$, one gets
\begin{equation}\label{050907}
\sup_{\left\|q\right\|_{\sigma,r}\leq1}\left\|X_R\right\|_{\sigma,r}\leq C_1(d)\left\|R\right\|_{\rho},
\end{equation}
where $C_1(d)$ is a universal constant depending on $d$ only.\color{black}
\end{lem}
\begin{proof}
Fixing any $\textbf j\in\mathbb{Z}^d$, it suffices to estimate the upper bound for
\begin{equation*}
\left|\frac{\partial{R}}{\partial q_{\textbf j}}e^{\left(\rho+\delta\right)\ln^{\sigma}\lfloor \textbf j\rfloor}\right|.
\end{equation*}In view of (\ref{052301}),
one has
\begin{equation*}
\frac{\partial{R}}{\partial q_{\textbf j}}=\sum_{a,k,k'\in\mathbb{N}^{\mathbb{Z}^d}}R_{akk'}\left(\prod_{\textbf n\neq \textbf j}I_{\textbf n}(0)^{a_{\textbf n}}q_{\textbf n}^{k_{\textbf n}}\bar{q}_{\textbf n}^{k_{\textbf n}'}\right)\left(k_{\textbf j}I_{\textbf j}(0)^{a_{\textbf j}}q_{\textbf j}^{k_{\textbf j}-1}\bar{q}_{\textbf j}^{k_{\textbf j}'}\right).
\end{equation*}
Based on (\ref{042602}), one has
\begin{equation}\label{050903}
\left|R_{akk'}\right|\leq \left\|R\right\|_{\rho}e^{\rho\left(\sum_{\textbf n\in\mathbb{Z}^d}\left(2a_{\textbf n}+k_{\textbf n}+k_{\textbf n}'\right)\ln^{\sigma}\lfloor{\textbf n}\rfloor-2\ln^{\sigma}\lfloor{\textbf n}_1^*(a,k,k')\rfloor\right)}.
\end{equation}\label{060310}
The conditions $\left\|q\right\|_{\sigma,\rho+\delta}\leq1$ implies
\begin{equation}\label{060310}
\left|q_{\textbf n}\right|\leq e^{-\left(\rho+\delta\right)\ln^{\sigma}\lfloor \textbf n\rfloor },\qquad \textbf n\in\mathbb{Z}^d.
\end{equation}Using (\ref{M13}), (\ref{050903}), (\ref{060310}) and $\rho+\delta<r$, one has
\begin{eqnarray}
\left|\frac{\partial{R}}{\partial q_{\textbf j}}e^{(\rho+\delta)\ln^{\sigma}\lfloor{\textbf j}\rfloor}\right|\nonumber
&\leq&C\left\|R\right\|_{\rho},
\end{eqnarray}
where
\begin{equation*}C:=\left|\sum_{a,k,k'\in\mathbb{N}^{\mathbb{Z}^d}}k_{\textbf j}e^{-\delta\sum_{{\textbf n}\in\mathbb{Z}^d}\left(2a_{\textbf n}+k_{\textbf n}+k_{\textbf n}'\right)\ln^{\sigma}\lfloor{\textbf n}\rfloor-2\rho\ln^{\sigma}\lfloor{\textbf n}_1^*(a,k,k')\rfloor+2(\rho+\delta)\ln^{\sigma}\lfloor{\textbf j}\rfloor}\right|.
\end{equation*}
Then the inequality (\ref{050907.}) will follow from
\begin{eqnarray}
C
&\leq& \exp\left\{10d\left(\frac{2000d}{\delta^2}\right)^d\cdot \exp\left\{d{\left(\frac{10d}{\delta} \right)}^{\frac{1}{\sigma-1}}\right\}\right\}\label{022302}.
\end{eqnarray}
Now we will estimate (\ref{022302}) in the following two cases:

\textbf{Case 1.} $$\lfloor{\textbf j}\rfloor \leq \lfloor {\textbf n}_3^{*}(a,k,k')\rfloor.$$
Then one has
\begin{eqnarray*}
C
\leq  \sum_{a,k,k'\in\mathbb{N}^{\mathbb{Z}^d}}k_{\textbf j}e^{-\frac{\delta}3\sum_{i\geq 1}\ln^{\sigma}\lfloor{\textbf n}_i^{*}(a,k,k')\rfloor}.
\end{eqnarray*}
Note that
\begin{equation}\label{060320}
k_{\textbf j}\leq \sum_{\textbf n\in\mathbb{Z}^d}\left(2a_{\textbf n}+k_{\textbf n}+k'_{\textbf n}\right)\leq \sum_{i\geq 1}\ln^{\sigma}\lfloor \textbf n^*_i(a,k,k')\rfloor,
\end{equation}and  we have
\begin{eqnarray}
&&\nonumber \sum_{a,k,k'\in\mathbb{N}^{\mathbb{Z}^d}}k_{\textbf j}e^{-\frac{\delta}3\sum_{i\geq 1}\ln^{\sigma}\lfloor{\textbf n}_i^{*}(a,k,k')\rfloor}\\
\nonumber&\leq&\sum_{a,k,k'\in\mathbb{N}^{\mathbb{Z}^d}}\left(\sum_{i\geq 1}\ln^{\sigma}\lfloor {\textbf n}_i^*(a,k,k')\rfloor\right)e^{-\frac{\delta}3\sum_{i\geq 1}\ln^{\sigma}\lfloor{\textbf n}_i^{*}(a,k,k')\rfloor}\\
\nonumber&\leq&\frac{12}{e\delta}\sum_{a,k,k'\in\mathbb{N}^{\mathbb{Z}^d}}e^{-\frac{\delta}4\sum_{i\geq 1}\ln^{\sigma}\lfloor{\textbf n}_i^{*}(a,k,k')\rfloor
}\qquad\mbox{(in view of (\ref{042805*}))}\\
\nonumber&\leq&\frac{12}{e\delta}\prod_{\textbf n\in\mathbb{Z}^d}\left({1-e^{-\frac {\delta}2\ln^{\sigma}\lfloor{\textbf n}\rfloor}}\right)^{-1}\prod_{\textbf n\in\mathbb{Z}^d}
\left({1-e^{-\frac {\delta}{4}\ln^{\sigma}\lfloor{\textbf n}\rfloor}}\right)^{-2} \quad {\mbox{(by (\ref{041809}))}}    \\
&\leq&\label{050904}\frac{12}{e\delta}\exp\left\{3\left(\frac{1600d}{\delta^2}\right)^d\cdot \exp\left\{d{\left(\frac{8d}{\delta} \right)}^{\frac{1}{\sigma-1}}\right\}\right\},
\end{eqnarray}
where the last inequality is based on (\ref{122401}) in Lemma \ref{a5}.

\textbf{Case 2.} $$\lfloor\textbf j\rfloor > \lfloor \textbf n_3^{*}(a,k,k')\rfloor.$$
If $k_{\textbf j}\geq 3$, then one has $\lfloor\textbf j\rfloor \leq \lfloor \textbf n_3^{*}(a,k,k')\rfloor$ which is in \textbf{Case\ 1.}. Hence we always assume $k_{\textbf j}\leq 2$. Then using (\ref{001}), we have
\begin{equation}
C
\label{201606032}\leq 2\left|\sum_{a,k,k'\in\mathbb{N}^{\mathbb{Z}^d}}e^{-\frac{\delta}2\sum_{i\geq 3}\ln^{\sigma}\lfloor \textbf n_i^{*}(a,k,k')\rfloor}\right|.
\end{equation}
For simplicity, we denote that $\textbf n_i=\textbf n_i^{*}(a,k,k')$ below. Note that if $(\textbf n_i)_{i\geq 1}$ is given, then $\textbf n(a,k,k')$ is specified, and hence $\textbf n(a,k,k')$ is specified up to a factor of
\begin{equation}\label{030204}\prod_{\textbf n\in\mathbb{Z}^d}\left(1+l_{\textbf n}^2\right),
\end{equation}
where
$$l_{\textbf n}=\#\{j:\textbf n_j=\textbf n\}.$$
In view of $\lfloor\textbf j\rfloor >\lfloor \textbf n_3\rfloor$ again, one has $\textbf j\in\left\{\textbf n_1,\textbf n_2\right\}$ and
\begin{equation*}
l_{\textbf n_1}+l_{\textbf n_2}\leq 2,
\end{equation*}
which implies
\begin{equation}\label{022301}
\prod_{\textbf n\in\mathbb{Z}^d\atop \textbf n=\textbf n_1,\textbf n_2}\left(1+l_{\textbf n}^2\right)\leq 5.
\end{equation}Furthermore if $(\textbf n_i)_{i\geq 3}$ and $\textbf j$ are given, $\textbf n_1$ and $\textbf n_2$ are uniquely determined. Then using (\ref{030204}) and (\ref{022301}) one has
\begin{eqnarray}\label{050905}
\nonumber(\ref{201606032})
&\leq&\nonumber 10\left|\sum_{\left(\textbf n_i\right)_{i\geq3}}\prod_{\textbf n\in\mathbb{Z}^d\atop\lfloor\textbf n\rfloor\leq \lfloor\textbf n_3\rfloor}\left(1+l_{\textbf n}^2\right)e^{-\frac{\delta}2\sum_{i\geq 3}\ln^{\sigma}\lfloor\textbf n_i\rfloor}\right| \\
\nonumber&\leq&10\left(\sum_{(\textbf n_i)_{i\geq3}}e^{-\frac {\delta}4\sum_{i\geq 3}\ln^{\sigma}\lfloor\textbf n_i\rfloor}\right)\cdot\sup_{(\textbf n_i)_{i\geq3}}\left(\prod_{\lfloor\textbf n\rfloor\leq \lfloor\textbf n_3\rfloor}\left(1+l_{\textbf n}^2\right)e^{-\frac {\delta}4\sum_{i\geq 3}\ln^{\sigma}\lfloor \textbf n_i\rfloor}\right)\\
&\leq&\label{022303}10\exp\left\{3\left(\frac{1600d}{\delta^2}\right)^\cdot \exp\left\{d{\left(\frac{8d}{\delta} \right)}^{\frac{1}{\sigma-1}}\right\}\right\}\nonumber\\&&\times\exp\left\{6d\left(\frac{16}{\delta}\right)^{\frac 1{\sigma-1}}\cdot\exp\left\{\left(\frac8{\delta}\right)^{\frac1\sigma}\right\}
\right\}\label{050904.},
\end{eqnarray}
where the last inequality is based on (\ref{041809}), (\ref{122401}) and (\ref{042807}).

In view of (\ref{050904}) and (\ref{050904.}), we finish the proof of (\ref{022302}).

When $\delta\geq \frac r2$, taking
\begin{equation*}
C_1(d)= \exp\left\{10d\left({8000d}\right)^d\cdot e^{20d^2}\right\},
\end{equation*}
we finish the proof of (\ref{050907}) by using $\sigma>2$ and $r\geq 1$.
\end{proof}
Furthermore, we have
\begin{lem}\label{052302}
Given $\sigma>2,r\geq 1,\rho\geq 0,\delta>0$ satisfying $\rho+\delta<r$, assume the Hamiltonian $R(q,\bar q)$ with the form of (\ref{052301}) satisfying $\left\|R\right\|_{\rho} < \infty$. Then for any $ \textbf m,\textbf l \in \mathbb{Z}^d$, we have
\begin{equation}\label{003}
\left\|\frac{\partial^{2}R}{\partial q_{\textbf m}\partial {q}_{\textbf l}}\right\|^{*}_{\rho+\delta},\left\|\frac{\partial^{2}R}{\partial q_{\textbf m}\partial \bar{q}_{\textbf l}}\right\|^{*}_{\rho+\delta}, \left\|\frac{\partial^{2}R}{\partial \bar q_{\textbf m}\partial \bar{q}_{\textbf l}}\right\|^{*}_{\rho+\delta} \leq C\left\|R\right\|_{\rho},
\end{equation}
where the constant $C$ is given by
\begin{equation*}
C=\left(\frac{12}{e\delta}\right)^2\cdot \exp\left\{\left(\frac{3600d}{\delta^2}\right)^d\cdot \exp\left\{d{\left(\frac{12d}\delta \right)}^{\frac{1}{\sigma-1}}\right\}\right\}.
\end{equation*}
\end{lem}
\begin{proof}
The proof of this lemma will be given in Appendix.
\end{proof}
\begin{defn}\label{061502}
Fix $\sigma>2$ and $r\geq 1$. For any $\rho\geq 0$,
the Hamiltonian  $R(q,\bar q)$ is called $\rho$-T\"{o}plitz-Lipschitz if the following limits exist
\begin{eqnarray}\nonumber
\underset{t\rightarrow\infty}{\lim}\frac{\partial^{2}R}{\partial q_{\textbf n+t\textbf l}\partial q_{\textbf m-t\textbf l}},\quad \underset{t\rightarrow\infty}{\lim}\frac{\partial^{2}R}{\partial {q}_{\textbf n+t\textbf l}\partial \bar{q}_{\textbf m+t\textbf l}},\quad \underset{t\rightarrow\infty}{\lim}\frac{\partial^{2}R}{\partial \bar{q}_{\textbf n+t\textbf l}\partial \bar{q}_{\textbf m-t\textbf l}},
\end{eqnarray}
for any fixed $\textbf  n,\textbf m,\textbf l \in \mathbb{Z}^d$.
{Moreover, there exists $ K> 0$,\ such that when $ |t| > K $, there exists a universal positive constant $C$ such that the function $ R(q,\bar q)$ satisfies
\begin{eqnarray}\label{t1}
\left\|\frac{\partial^{2}R}{\partial q_{\textbf n+t\textbf l}\partial q_{\textbf m-t\textbf l}}-\underset{t\rightarrow\infty}{\lim}\frac{\partial^{2}R}{\partial q_{\textbf n+t\textbf l}\partial q_{\textbf m-t\textbf l}}\right\|^{*}_{\rho} \leq
\frac{C}{|t|},\\
\label{t2} \left\|\frac{\partial^{2}R}{\partial {q}_{\textbf n+t\textbf l}\partial \bar{q}_{\textbf m+t\textbf l}}-\underset{t\rightarrow\infty}{\lim}\frac{\partial^{2}R}{\partial {q}_{\textbf n+t\textbf l}\partial \bar{q}_{\textbf m+t\textbf l}}\right\|^{*}_{\rho} \leq \frac{C}{|t|},\\
\label{t3}\left\|\frac{\partial^{2}R}{\partial \bar{q}_{\textbf n+t\textbf l}\partial \bar{q}_{\textbf m-t\textbf l}}-\underset{t\rightarrow\infty}{\lim}\frac{\partial^{2}R}{\partial \bar{q}_{\textbf n+t\textbf l}\partial \bar{q}_{\textbf m-t\textbf l}}\right\|^{*}_{\rho}\leq \frac{C}{|t|}.
\end{eqnarray}}
\end{defn}

\begin{lem}(\textbf{Poisson Bracket})\label{010}
Let $\sigma>2,\rho>0$ and $$0<\delta_1,\delta_2<\min\left\{\frac14\rho,3-2\sqrt{2}\right\}.$$ Then one has
\begin{equation}\label{042704}
\left\|\{R_1,R_2\}\right\|_\rho\leq C\left\|R_1\right\|_{\rho-\delta_1}\left\|R_2\right\|_{\rho-\delta_2},
\end{equation}
where the constant $C$ is given by
\begin{equation*}
C=\frac{1}{\delta_2}\cdot\exp\left\{3\left(\frac{14400d}{\delta_1^2}\right)^d\cdot \exp\left\{d{\left(\frac{24d}{\delta_1} \right)}^{\frac{1}{\sigma-1}}\right\}\right\}.
\end{equation*}
\end{lem}
\begin{proof}
The proof of this lemma will be given in Appendix.
\end{proof}

Next, we will estimate the symplectic transformation $\Phi_F$ induced by the Hamiltonian function $F$. Actually, we have
\begin{lem}(\textbf{Hamiltonian Flow})\label{E1}
Let $\rho>0$ and
$$0<\delta<\min \left\{\frac14\rho,3-2\sqrt{2}\right\}.$$
 Assume further \begin{equation}\label{042801} \frac{2e}{\delta}\cdot \exp\left\{3\left(\frac{14400d}{\delta^2}\right)^d\cdot \exp\left\{d{\left(\frac{24d}{\delta} \right)}^{\frac{1}{\sigma-1}}\right\}\right\}\left\|F\right\|_{\rho-\delta} <\frac 12.
\end{equation}
Then for any Hamiltonian function $H$, we get
\begin{equation}\label{052101}
\left\|H\circ\Phi_F\right\|_\rho
\leq\left(1+ \frac{4e}{\delta}\cdot\exp\left\{3\left(\frac{14400d}{\delta^2}\right)^d\cdot \exp\left\{d{\left(\frac{24d}{\delta} \right)}^{\frac{1}{\sigma-1}}\right\}\right\}\left\|F\right\|_{\rho-\delta}\right)
||H||_{\rho-\delta}.
\end{equation}
\end{lem}
\begin{proof}
 Firstly,  we expand $H\circ\Phi_F$ into the Taylor series
 \begin{equation}\label{3.2}
 H\circ\Phi_F=\sum_{n\geq 0}\frac{1}{n!}H^{(n)},
 \end{equation}
where $H^{(n)}=\{H^{(n-1)},F\}$ and $H^{(0)}=H$.

We will estimate $\left\|H^{(n)}\right\|_\rho$ by using Lemma \ref{010} again and again:
\begin{equation}
\left\|H^{(n)}\right\|_\rho
\label{3.3}\leq\left(\exp\left\{3\left(\frac{14400d}{\delta^2}\right)^d\cdot \exp\left\{d{\left(\frac{24d}{\delta} \right)}^{\frac{1}{\sigma-1}}\right\}\right\}
\left\|F\right\|_{\rho-\delta}\right)^n\left(\frac{2n}{\delta}\right)^n\left\|H\right\|_{\rho-{\delta}}.
\end{equation}
In view of (\ref{3.2}), (\ref{3.3}) and the following inequality
\begin{equation*}
j^j<j!e^j,
\end{equation*}one has
\begin{equation*}
\left\|H\circ\Phi_F\right\|_\rho
\leq \left(1+\frac{4e}{\delta}\cdot\exp\left\{3\left(\frac{14400d}{\delta^2}\right)^d\cdot \exp\left\{d{\left(\frac{24d}{\delta} \right)}^{\frac{1}{\sigma-1}}\right\}\right\}\left\|F\right\|_{\rho-\delta}\right)
\left\|H\right\|_{\rho-\delta},
\end{equation*}
which finishes the proof of (\ref{052101}).
\end{proof}

\section{KAM iteration}
\subsection{Derivation of homological equations}
According the basic idea of KAM theory (see \cite{Bourgain2005JFA} for example),
the proof of Main Theorem employs the rapidly
converging iteration scheme of Newton type to deal with small divisor problems
introduced by Kolmogorov, involving the infinite sequence of coordinate transformations.
At the $s$-th step of the scheme, a Hamiltonian
$H_{s} = N_{s} + R_{s}$
is considered, as a small perturbation of some normal form $N_{s}$. A transformation $\Phi_{s}$ is
set up so that
$$H_{s}\circ \Phi_{s} = N_{s+1} + R_{s+1}$$
with another normal form $N_{s+1}$ and a much smaller perturbation $R_{s+1}$. We drop the index $s$ of $H_{s}, N_{s}, R_{s}, \Phi_{s}$ and shorten the index $s+1$ as $+$.

Now consider the Hamiltonian $H$ of the form
\begin{eqnarray}\label{N1}
{H}=N+R,
\end{eqnarray}
where
\begin{equation*}
N=\sum_{\textbf n\in\mathbb{Z}^d}\left(\left\|\textbf n\right\|^2+\widetilde V_{\textbf n}\right)|q_{\textbf n}|^2,
\end{equation*}
and
\begin{equation*}
R=R_0+R_1+R_2
\end{equation*}
with $\left|\widetilde V_{\textbf n}\right|\leq 2$ for all $\textbf n\in\mathbb{Z}^d$,
\begin{eqnarray*}
{R}_0&=&\sum_{a,k,k'\in\mathbb{N}^{\mathbb{Z}^d}\atop\mbox{supp}\ k\bigcap \mbox{supp}\ k'=\emptyset}R_{akk'}\mathcal{M}_{akk'},\\
{R}_1&=&\sum_{\textbf n\in\mathbb{Z}^d}J_{\textbf n}\left(\sum_{a,k,k'\in\mathbb{N}^{\mathbb{Z}^d}\atop\mbox{supp}\ k\bigcap \mbox{supp}\ k'=\emptyset}R_{akk'}^{(\textbf n)}\mathcal{M}_{akk'}\right),\\
{R}_2&=&\sum_{\textbf n,\textbf m\in\mathbb{Z}^d}J_{\textbf n}J_{\textbf m}\left(\sum_{a,k,k'\in\mathbb{N}^{\mathbb{Z}^d}\atop\mbox{no assumption}}R_{akk'}^{(\textbf n,\textbf m)}\mathcal{M}_{akk'}\right).
\end{eqnarray*}
Here $J_{\textbf n}=|q_{\textbf n}|^2-I_{\textbf n}(0)$ for any $\textbf n\in\mathbb{Z}^d$.
We desire to eliminate the terms $R_0,R_1$ in (\ref{N1}) by the coordinate transformation $\Phi$, which is obtained as the time-1 map $X_F^{t}|_{t=1}$ of a Hamiltonian
vector field $X_F$ with $F=F_0+F_1$. Let ${F}_{0}$ (resp. ${F}_{1}$) has the form of ${R}_0$ (resp. ${R}_{1}$),
that is \begin{eqnarray}\label{052001}
&&{F}_0=\sum_{a,k,k'\in\mathbb{N}^{\mathbb{Z}^d}\atop\mbox{supp}\ k\bigcap \mbox{supp}\ k'=\emptyset}F_{akk'}\mathcal{M}_{akk'},\\
&&{F}_1=\sum_{\textbf n\in\mathbb{Z}^d}J_{\textbf n}\left(\sum_{a,k,k'\in\mathbb{N}^{\mathbb{Z}^d}\atop\mbox{supp}\ k\bigcap \mbox{supp}\ k'=\emptyset}F_{akk'}^{(\textbf n)}\mathcal{M}_{akk'}\right)\label{052002},
\end{eqnarray}
and the homological equations become
\begin{equation}\label{4.27}
\{N,{F}\}+R_0+R_{1}=[R_0]+[R_1],
\end{equation}
where
\begin{equation}\label{051501.}
[R_0]=\sum_{a\in\mathbb{N}^{\mathbb{Z}^d}}R_{a00}\mathcal{M}_{a00},
\end{equation}
and
\begin{equation}\label{051502}
[R_1]=\sum_{\textbf n\in\mathbb{Z}^d}J_{\textbf n}\sum_{a\in\mathbb{N}^{\mathbb{Z}^d}}R_{a00}^{(\textbf n)}\mathcal{M}_{a00}.
\end{equation}
Define $\widetilde{\omega}=(\widetilde{\omega}_{\textbf n})_{\textbf n\in\mathbb{Z}^d}$ with
\begin{equation*}
\widetilde{\omega}_{\textbf n}=\sum_{a\in\mathbb{N}^{\mathbb{Z}^d}}R_{a00}^{(\textbf n)}\mathcal{M}_{a00}
\end{equation*}
is the so-called frequency shift.
The solutions of the homological equations (\ref{4.27}) are given by
\begin{equation}\label{051304}
F_{akk'}=\frac{R_{akk'}}{\sum_{\textbf n\in\mathbb{Z}^d}\left(k_{\textbf n}-k^{'}_{\textbf n}\right)(\left\|\textbf n\right\|^2+\widetilde{V}_{\textbf n})},
\end{equation}
where
\begin{equation}\label{051305}
F_{akk'}^{(\textbf m)}=\frac{R_{akk'}^{(\textbf m)}}{\sum_{\textbf n\in\mathbb{Z}^d}\left(k_{\textbf n}-k^{'}_{\textbf n}\right)(\left\|\textbf n\right\|^2+\widetilde{V}_{\textbf n})},
\end{equation}
and the new Hamiltonian ${H}_{+}$ has the form
\begin{eqnarray*}
H_{+}\nonumber&=&H\circ\Phi\\
&=&\nonumber N+\{N,F\}+R_0+R_1\\
&&\nonumber+\int_{0}^1\{(1-t)\{N,F\}+R_0+R_1,F\}\circ X_F^{t}\ {d}{t}
+\nonumber R_2\circ X_F^1\\
&=&N_++R_+,
\end{eqnarray*}
where
\begin{equation}\label{051402}
N_+=N+[R_0]+[R_1],
\end{equation}
and
\begin{equation}\label{051403}
R_+=\int_{0}^1\{(1-t)\{N,F\}+R_0+R_1,F\}\circ X_F^{t}\ {d} t+R_2\circ X_F^1.
\end{equation}
\subsection{The new norm}To estimate
the solutions of the homological equation (\ref{4.27}), it is convenient to define a new norm for the Hamiltonian ${R}$ with the form of (\ref{N1}) by
\begin{equation*}
\left\|{R}\right\|_{\rho}^{+}=\max\left\{\left\|R_0\right\|_\rho^{+},\left\|R_1\right\|_\rho^{+}|,
\left\|R_2\right\|_\rho^{+}\right\},
\end{equation*}
where
\begin{equation}
\label{051302}\left\|R_0\right\|_\rho^{+}=\sup_{a,k,k'\in\mathbb{N}^{\mathbb{Z}^d}}\frac{\left|R_{akk'}\right|}
{e^{\rho\left(\sum_{\textbf n\in\mathbb{Z}^d}\left(2a_{\textbf n}+k_{\textbf n}+k_{\textbf n}'\right)\ln^{\sigma}\lfloor\textbf n\rfloor-2\ln^{\sigma}\lfloor\textbf n_1^{*}(a,k,k')\rfloor\right)}},
\end{equation}
\begin{equation}
\left\|R_1\right\|_\rho^{+}=\sup_{a,k,k'\in\mathbb{N}^{\mathbb{Z}^d}\atop \textbf m\in\mathbb{Z}^d}\frac{\left|R^{(\textbf m)}_{akk'}\right|}{e^{\rho\left(\sum_{\textbf n\in\mathbb{Z}^d}
\left(2a_{\textbf n}+k_{\textbf n}+k_{\textbf n}'\right)\ln^{\sigma}\lfloor{\textbf n}\rfloor+2\ln^{\sigma}\lfloor{\textbf m}\rfloor-2\ln^{\sigma}\lfloor{\textbf n}_1^{*}(a,k,k';\textbf m)\rfloor\right)}}\label{031920},
\end{equation}
and
\begin{equation}\label{031921}
\left\|R_2\right\|_\rho^{+}=\sup_{a,k,k'\in\mathbb{N}^{\mathbb{Z}^d}\atop
\textbf m_1,\textbf m_2\in\mathbb{Z}^d}\frac{\left|R^{(\textbf m_1,\textbf m_2)}_{akk'}\right|}
{e^{\rho\left(\sum_{\textbf n\in\mathbb{Z}^d}\left(2a_{\textbf n}+k_{\textbf n}+k_{\textbf n}'\right)\ln^{\sigma}\lfloor{\textbf n}\rfloor
+2\ln^{\sigma}\lfloor{\textbf m}_1\rfloor+2\ln^{\sigma}\lfloor{\textbf m}_2\rfloor-2\ln^{\sigma}\lfloor\textbf n_1^{*}(a,k,k';{\textbf m}_1,{\textbf m}_2)\rfloor\right)}},
\end{equation}
noting that
\begin{equation*}
\lfloor \textbf n_1^{*}(a,k,k';{\textbf m})\rfloor=\max\left\{\lfloor{\textbf n}_1^{*}(a,k,k')\rfloor,\lfloor{\textbf m}\rfloor\right\};
\end{equation*}
and
\begin{equation*}
\lfloor\textbf n_1^{*}(a,k,k';{\textbf m}_1,{\textbf m}_2)\rfloor=\max\left\{\lfloor{\textbf n}_1^{*}(a,k,k')\rfloor,\lfloor{\textbf m}_1\rfloor,\lfloor{\textbf m}_2\rfloor\right\}.
\end{equation*}

Then one has the following estimates:
\begin{lem}\label{051301}
Given any $\rho,\delta>0$ and a Hamiltonian $R$, one has
\begin{equation}\label{N6}
\left\|R\right\|_{\rho+\delta}^{+}\leq\exp\left\{10d\left(\frac{10}{\delta}\right)^{\frac 1{\sigma-1}}\cdot\exp\left\{\left(\frac{10}\delta\right)^{\frac1\sigma}\right\}
\right\}\left\|R\right\|_{\rho}
\end{equation}
and
\begin{equation}\label{N7}
\left\|R\right\|_{\rho+\delta}\leq\frac{64}{e^2\delta^2}\left\|R\right\|_{\rho}^{+}.
\end{equation}
\end{lem}

\begin{proof}
The details of the proof will be given in the Appendix.
\end{proof}
\begin{rem}
For $R_1$ (also $F_1$), one needs to define
 the system $\textbf {n}(a,k,k';\textbf m)$  by
\begin{equation*}
\textbf {n}(a,k,k';\textbf m):=\textbf {n}\left(a+e_{\textbf m},k,k'\right),
\end{equation*}
where $e_{\textbf m}$ is the unit vector with all components equal to zero but the $\textbf m$-th one equal to $1$. In another word, one can think that $J_{\textbf m}$ is replaced by $I_{m}(0)$ in $R_1$ and the responding coefficient has the same upper bound.
Similarly for $R_2$, one needs to define
 the system $\textbf {n}(a,k,k';\textbf m_1,\textbf m_2)$  by
\begin{equation*}
\textbf {n}(a,k,k';\textbf m_1,\textbf m_2):=\textbf n(a+e_{\textbf m_1}+e_{\textbf m_2},k,k').
\end{equation*}

Abusing the notations, we will denote
\begin{equation*}
\textbf {n}(a,k,k';\textbf m)=\textbf {n}(a,k,k')
\end{equation*}
and
\begin{equation*}
\textbf {n}(a,k,k';\textbf m_1,\textbf m_2)=\textbf {n}(a,k,k')
\end{equation*}
for simplicity in the next subsection.
\end{rem}
\subsection{KAM Iteration}\label{031501}

Now we give the precise
set-up of iteration parameters. For any $\sigma>2$, let $s\geq0$ be the $s$-th KAM
step.
 \begin{itemize}
 \item[]$\delta_{s}=\frac{\rho_0}{(s+4)\ln^2(s+4)},\qquad$ with $\rho_0= \frac{3-2\sqrt{2}}{100}$,

 \item[]$\rho_{s+1}=\rho_{s}+3\delta_s,$

 \item[]$\epsilon_s=\epsilon_{0}^{(\frac{3}{2})^s}$, which dominates the size of
 the perturbation,

 \item[]$\lambda_s=\epsilon_s^{0.01}$,

 \item[]$\eta_{s+1}=\frac{1}{20}\lambda_s\eta_s$, with $\eta_0=\lambda_0$,

 \item[]$d_{s+1}=d_{s}+\frac{1}{\pi^2(s+1)^2}$, with $d_0=0$,

 \item[]$D_s=\left\{(q_{\textbf n})_{\textbf n\in\mathbb{Z}^d}:\frac{1}{2}+d_s\leq\left|q_{\textbf n}\right|e^{r\ln^{\sigma}\lfloor \textbf n\rfloor}\leq1-d_s\right\}$.
 \end{itemize}
\begin{rem}
Then one has
\begin{equation*}
\sum_{s\geq 0}\delta_s\leq \frac53\rho_0,
\end{equation*}
\begin{equation*}
\rho_0\leq \rho_s\leq 6\rho_0<3-2\sqrt{2},\qquad \forall\ s\geq 0,
\end{equation*}
\begin{equation*}\delta_s<\min \left\{\frac14\rho_s,3-2\sqrt{2}\right\},\qquad \forall\ s\geq 0.
\end{equation*}
\end{rem}
Denote the complex cube of size $\lambda>0$:
\begin{equation*}\label{M9}
\mathcal{C}_{\lambda}\left(\widetilde{V}\right)=\left\{\left(V_{\textbf n}\right)_{\textbf n\in\mathbb{Z}^d}\in\mathbb{C}^{\mathbb{Z}^d}:\left|V_{\textbf n}-\widetilde{V}_{\textbf n}\right|\leq \lambda\right\}.
\end{equation*}

\begin{lem}{\label{IL}}
Suppose $H_{s}=N_{s}+R_{s}$ is real analytic on $D_{s}\times\mathcal{C}_{\eta_{s}}\left(V_{s}^*\right)$,
where $$N_{s}=\sum_{\textbf n\in\mathbb{Z}^d}\left(\left\|\textbf n\right\|^2+\widetilde V_{\textbf n,s}\right)\left|q_{\textbf n}\right|^2$$ is a normal form with
\begin{equation}\label{060401}
\widetilde V_{\textbf n,s}=\breve V_s+\widehat V_{\textbf n,s},
\end{equation}where $\breve V_s$ does not depend on $\textbf n$ and satisfies
\begin{equation}\label{061102}
\left|\breve V_s\right|\leq \sum_{i=0}^{s}\epsilon^{0.5}_i,
\end{equation}and $\widehat V_s=\left(\widehat V_{\textbf n,s}\right)_{\textbf n\in\mathbb{Z}^d}$ satisfies
\begin{eqnarray}
\label{198}&&\widehat{V}_{s}\left(V_{s}^*\right)=\omega,\\
\label{053190}&&\left|\widehat{V}_{\textbf n,s}\right|\leq \left(2\sum_{i=0}^s{\epsilon_{i}^{0.5}}\right)\frac1{\langle\textbf n\rangle},\\
\label{199}&&\left|\left|\frac{\partial \widehat{V}_s}{{\partial V}}-Id\right|\right|_{l^{\infty}\rightarrow l^{\infty}}<d_s\epsilon_{0}^{\frac{1}{10}}.
\end{eqnarray}
Assume that $R_{s}=R_{0,s}+R_{1,s}+R_{2,s}$ satisfies
\begin{eqnarray}
\label{200}&&\left\|R_{0,s}\right\|_{\rho_{s}}^{+}\leq \epsilon_{s}:=\varepsilon_{0s},\\
\label{201}&&\left\|R_{1,s}\right\|_{\rho_{s}}^{+}\leq \epsilon_{s}^{0.6}:=\varepsilon_{1s},\\
\label{202}&&\left\|R_{2,s}\right\|_{\rho_{s}}^{+}\leq (1+d_s)\epsilon_0:=\varepsilon_{2s}.
\end{eqnarray}
Moreover, assume that $R_s$ satisfies $\rho_s$--T\"{o}plitz-Lipschitz property
and for any fixed $\textbf  n,\textbf m,\textbf l \in \mathbb{Z}^d$, there exists $ K_s> 0$,\ such that when $ |t| > K_s $
\begin{eqnarray}\label{T1}
\left\|\frac{\partial^{2}R_{i,s}}{\partial q_{\textbf n+t\textbf l}\partial q_{\textbf m-t\textbf l}}-\underset{t\rightarrow\infty}{\lim}\frac{\partial^{2}R_{i,s}}{\partial q_{\textbf n+t\textbf l}\partial q_{\textbf m-t\textbf l}}\right\|^{*}_{\rho_s} \leq
\frac{\varepsilon_{is}}{|t|},\\
\label{T2} \left\|\frac{\partial^{2}R_{i,s}}{\partial {q}_{\textbf n+t\textbf l}\partial \bar{q}_{\textbf m+t\textbf l}}-\underset{t\rightarrow\infty}{\lim}\frac{\partial^{2}R_{i,s}}{\partial {q}_{\textbf n+t\textbf l}\partial \bar{q}_{\textbf m+t\textbf l}}\right\|^{*}_{\rho_s} \leq \frac{\varepsilon_{is}}{|t|},\\
\label{T3}\left\|\frac{\partial^{2}R_{i,s}}{\partial \bar{q}_{\textbf n+t\textbf l}\partial \bar{q}_{\textbf m-t\textbf l}}-\underset{t\rightarrow\infty}{\lim}\frac{\partial^{2}R_{i,s}}{\partial \bar{q}_{\textbf n+t\textbf l}\partial \bar{q}_{\textbf m-t\textbf l}}\right\|^{*}_{\rho_s}\leq \frac{\varepsilon_{is}}{|t|},
\end{eqnarray}for $i=0,1,2$.

Then for all $V\in\mathcal{C}_{\eta_{s}}\left(V_{s}^*\right)$ satisfying $\widehat V_{s}(V)\in\mathcal{C}_{\lambda_s}(\omega)$, there exists a real analytic symplectic coordinate transformations
$\Phi_{s+1}:D_{s+1}\rightarrow D_{s}$ satisfying
\begin{eqnarray}
\label{203}&&\left\|\Phi_{s+1}-id\right\|_{\sigma,r}\leq \epsilon_{s}^{0.5},\\
\label{204}&&\left\|D\Phi_{s+1}-Id\right\|_{(\sigma,r)\rightarrow(\sigma,r)}\leq \epsilon_{s}^{0.5},
\end{eqnarray}
such that for
$H_{s+1}=H_{s}\circ\Phi_{s+1}=N_{s+1}+R_{s+1}$, the same assumptions as above are satisfied with `$s+1$' in place of `$s$', where $\mathcal{C}_{\eta_{s+1}}\left(V_{s+1}^*\right)\subset\widehat V_{s}^{-1}(\mathcal{C}_{\lambda_s}(\omega))$ and
\begin{equation}\label{206}
\left\|\widehat{V}_{s+1}-\widehat{V}_{s}\right\|_{\infty}\leq\epsilon_{s}^{0.5},
\end{equation}
\begin{equation}\label{205}
\left\|V_{s+1}^*-V_{s}^*\right\|_{\infty}\leq2\epsilon_{s}^{0.5}.
\end{equation}
\end{lem}
\begin{proof}
\textbf {Step. 1. Truncation step.}

We would like to estimate $F_0$ and $F_1$, which are the solutions of the homological equation (see (\ref{052001}) and (\ref{052002})). Without loss of generality, we only consider $F_0$ below.
In the step $s\rightarrow s+1$, there is saving of a factor
\begin{equation*}
\mathcal{A}_1:=e^{-\delta_{s}\left(\sum_{{\textbf n}\in\mathbb{Z}^d}\left(2a_{\textbf n}+k_{\textbf n}+k'_{\textbf n}\right)\ln^{\sigma}\lfloor{\textbf n}\rfloor-2\ln^{\sigma}\lfloor{\textbf n}_1^{*}(a,k,k')\rfloor\right)}.
\end{equation*}
 Using (\ref{001}) in Lemma \ref{005},
one has
\begin{equation*}
\mathcal{A}_1\leq e^{-\frac12\delta_{s}\sum_{i\geq3}\ln^{\sigma}\lfloor{\textbf n} _i^{*}(a,k,k')\rfloor}.
\end{equation*}
 Recalling after this step, we need
\begin{eqnarray*}
\left\|R_{0,s+1}\right\|_{\rho_{s+1}}^{+}\leq \epsilon_{s+1}
\end{eqnarray*}
and
\begin{eqnarray*}
\left\|R_{1,s+1}\right\|_{\rho_{s+1}}^{+}\leq \epsilon_{s+1}^{0.6}.
\end{eqnarray*}
Consequently, in $R_{0,s}$ and $R_{1,s}$, it suffices to eliminate the nonresonant monomials $\mathcal{M}_{akk'}$ for which
\begin{equation*}
e^{-\frac12\delta_{s}\sum_{i\geq3}\ln^{\sigma}\lfloor \textbf n_i^*(a,k,k')\rfloor}\geq\epsilon_{s+1}.
\end{equation*}
That is to say
\begin{equation}\label{M1}
 \sum_{i\geq3}\ln^{\sigma}\lfloor{\textbf n}_i^{*}(a,k,k')\rfloor\leq\frac{2(s+4)\ln^2(s+4)}{\rho_0}\cdot\ln\frac{1}{\epsilon_{s+1}}:=B_s.
\end{equation}
Hence we finished the truncation step.

\textbf {Step. 2. Estimate the lower bound of the small divisor.}

For any $\textbf{n}(a,k,k')$ the divisor
\begin{equation*}
\frac{1}{\sum_{\textbf n\in\mathbb{Z}^d}(k_{\textbf n}-k_{\textbf n}')(\left\|\textbf n\right\|^2+\widetilde V_{\textbf n})}
\end{equation*}will appear when estimating $F_0$ by (\ref{051304}).
 If
\begin{equation}\label{051701}
\left\|\textbf n_2^{*}(a,k,k')\right\|=\left\|\textbf n_3^{*}(a,k,k')\right\|,
\end{equation}
we will use the nonresonant conditions (\ref{040601}); otherwise
we will use the nonresonant conditions (\ref{040602}).
Now we will prove there is a lower bound on the right hand side of (\ref{040601}) and (\ref{040602}) respectively where the condition (\ref{M1}) satisfies.

Firstly we assume that the condition (\ref{051701}) holds. In view of momentum conservation (\ref{050901}) and (\ref{051701}), one has
\begin{equation*}
\left\|\textbf n_1^{*}(a,k,k')\right\|\leq \sum_{i\geq 2}\left\|\textbf n_i^{*}(a,k,k')\right\|\leq 2\sum_{i\geq 3}\left\|\textbf n_i^{*}(a,k,k')\right\|,
\end{equation*}
which implies
\begin{equation}\label{051704}
\lfloor\textbf n_1^{*}(a,k,k')\rfloor\leq  2\sum_{i\geq 3}\lfloor\textbf n_i^{*}(a,k,k')\rfloor.
\end{equation}
Then we have
\begin{eqnarray*}
&&\sum_{\textbf n\in\mathbb{Z}^d}\left|k_{\textbf n}-k'_{\textbf n}\right|\ln^{\sigma}\lfloor \textbf n\rfloor\\
&\leq&\sum_{i\geq 1}\ln^{\sigma}\lfloor \textbf n_i^*(a,k,k')\rfloor\\
&\leq &\ln^{\sigma}\left(2\sum_{i\geq 3}\lfloor\textbf n_i^{*}(a,k,k')\rfloor\right)+2\sum_{i\geq 3}\ln^{\sigma}\lfloor\textbf n_i^{*}(a,k,k')\rfloor\qquad \mbox{(by (\ref{051701}) and (\ref{051704}))}\\
&\leq&5\sum_{i\geq 3}\ln^{\sigma}\lfloor\textbf n_i^{*}(a,k,k')\rfloor,
\end{eqnarray*}
where the last inequality follows from (\ref{51504}) in Remark \ref{51505}.

Let
\begin{equation}\label{030701}N_s=B_s^{\frac{\sigma-1}{d\sigma}}\cdot \left(\ln B_s\right)^{-\frac{\sigma}d}
\end{equation} and then one has
\begin{eqnarray}
&&\nonumber\prod_{\textbf n\in\mathbb{Z}^d\atop (\ref{M1})\ satisfies}\frac{1}{1+|k_{\textbf n}-k'_{\textbf n}|^3\langle \textbf n\rangle^{d+4}}\\
\nonumber&=&\exp\left\{\sum_{\left\|\textbf n\right\|\leq N_s\atop (\ref{M1})\ satisfies}\ln\left(\frac{1}{1+|k_{\textbf n}-k'_{\textbf n}|^3\langle \textbf n\rangle^{d+4}}\right)\right\}\\
&&\nonumber\times\exp\left\{\sum_{\left\|\textbf n\right\|>N_s\atop (\ref{M1})\ satisfies}\ln\left(\frac{1}{1+|k_{\textbf n}-k'_{\textbf n}|^3\langle \textbf n\rangle^{d+4}}\right)\right\}\\
\nonumber
&\geq& \exp\left\{-10d\sum_{\left\|\textbf n\right\|\leq N_s\atop (\ref{M1})\ satisfies}\left|k_{\textbf n}-k'_{\textbf n}\right|^{\frac1{\sigma}}\ln\lfloor\textbf n\rfloor\right\}\\
&&\times\exp\left\{-10d\sum\limits_{\left\|\textbf n\right\|>N_s\atop (\ref{M1})\ satisfies}|k_{\textbf n}-k'_{\textbf n}|\ln\lfloor\textbf n\rfloor\right\}\nonumber\\
\nonumber&\geq& \exp\left\{-50d\cdot(3N_s)^{d} B^{\frac1\sigma}_s-50dB_s(\ln N_s)^{1-\sigma}\right\}\qquad\qquad \ \mbox{(in view of (\ref{M1}))}\\
\label {M6}&\geq& \exp\left\{-{100\cdot6^dB_s\left(\ln B_s\right)^{1-\sigma}}\right\},
\end{eqnarray}where the last inequality is based on (\ref{030701}).

Furthermore,
by (\ref{M1}) one has
\begin{eqnarray}
&&\nonumber100\cdot 6^dB_s\left(\ln B_s\right)^{1-\sigma}\\
&=&\nonumber100\cdot 6^d\left(\frac{2(s+4)\ln^2(s+4)}{\rho_0}\cdot\ln\frac{1}{\epsilon_{s+1}}\right)\left(\ln\left(\frac{2(s+4)\ln^2(s+4)}{\rho_0}\cdot\ln\frac{1}{\epsilon_{s+1}}\right)\right)^{1-\sigma}\\
&=&\ln\frac{1}{\epsilon_{s}}\cdot\left(100\cdot 6^d\cdot\frac{2(s+4)\ln^2(s+4)}{\rho_0}\cdot\frac32\right)\nonumber\\
&&\times\left(\ln\left(\frac{2(s+4)\ln^2(s+4)}{\rho_0}\cdot\left(\frac32\right)^{s+1}\ln\frac{1}{\epsilon_{0}}\right)\right)^{1-\sigma}\label{022402}.
\end{eqnarray}
Note that
\begin{equation}
\ln\left(\frac{2(s+4)\ln^2(s+4)}{\rho_0}\cdot\left(\frac32\right)^{s+1}\ln\frac{1}{\epsilon_{0}}\right)\geq (s+1)\ln\frac32 +\ln\ln\frac1{\epsilon_0},
\end{equation}
and then one has
\begin{equation}\label{053110}
\left(100\cdot 6^d\cdot\frac{2(s+4)\ln^2(s+4)}{\rho_0}\cdot\frac32\right)\left((s+1)\ln\frac32 +\ln\ln\frac1{\epsilon_0}\right)^{1-\sigma}\leq 0.01
\end{equation}
where using $\sigma>2$ and $\epsilon_0$ is sufficiently small depending on $d$ only. Furthermore by (\ref{M6})-(\ref{053110}), we have
\begin{equation}
\prod_{\textbf n\in\mathbb{Z}^d\atop (\ref{M1})\ satisfies}\frac{1}{1+|k_{\textbf n}-k'_{\textbf n}|^3\langle \textbf n\rangle^{d+4}}
 \geq\exp\left\{-0.01\cdot \ln \frac{1}{\epsilon_s}\right\}=\epsilon_s^{0.01}
\label{M7}.
\end{equation}
Following the proof of (\ref{M7}), one has
\begin{equation}\label{M8}
\prod_{\textbf n\in\mathbb{Z}^d\atop \left\|\textbf n\right\|\leq \left\|\textbf n_3^*(a,k,k')\right\|}\left(\frac{1}{1+|k_{\textbf n}-k'_{\textbf n}|^3\langle \textbf n\rangle^{d+7}}\right)^{10}\geq \epsilon_s^{0.01}.
\end{equation}
\textbf{Step. 3. The solutions of homological equation.}

In view of (\ref{060401}) and noting that $\breve V_{s}$ does not depend on $\textbf n$, one has
\begin{equation}
\sum_{\textbf{n}\in\mathbb{Z}^d}\left(k_{\textbf n}-k'_{\textbf n}\right)\left(\left\|\textbf n\right\|^2+\widetilde V_{\textbf n}\right)=\sum_{\textbf{n}\in\mathbb{Z}^d}\left(k_{\textbf n}-k'_{\textbf n}\right)\left(\left\|\textbf n\right\|^2+\widehat V_{\textbf n}\right).
\end{equation}
By (\ref{198}), one gets
\begin{eqnarray}
&&\nonumber\left|\sum_{\textbf{n}\in\mathbb{Z}^d}\left(k_{\textbf n}-k'_{\textbf n}\right)\left(\left\|\textbf n\right\|^2+\widehat V_{\textbf n}\left(V^*_s\right)\right)\right|\\
&=&\nonumber\left|\sum_{\textbf{n}\in\mathbb{Z}^d}\left(k_{\textbf n}-k'_{\textbf n}\right)\left(\left\|\textbf n\right\|^2+\omega_{\textbf n}\right)\right|\\
&\geq &\label{060402}\left|\left|\left|\sum_{\textbf{n}\in\mathbb{Z}^d}\left(k_{\textbf n}-k'_{\textbf n}\right)\omega_{\textbf n}\right|\right|\right|.
\end{eqnarray}
Using (\ref{M7})-(\ref{060402}), we have
\begin{equation}\label{060403}
\left|\frac1{\sum_{\textbf{n}\in\mathbb{Z}^d}\left(k_{\textbf n}-k'_{\textbf n}\right)\left(\left\|\textbf n\right\|^2+\widetilde V_{\textbf n}(V_s^*)\right)}\right|\leq \epsilon^{-0.01}_s.
\end{equation}
Recalling $\lambda_s=\epsilon_s^{0.01}$, if $V\in \mathcal{C}_{\lambda_s}(\widehat{V}_s)$, the estimate (\ref{060403}) remains true when substituting $V$ for $\widehat{V}_s$. Moreover, there is analyticity on $\mathcal{C}_{\lambda_s}(\widehat{V}_s)$. The transformations $\Phi_{s+1}$ is obtained as the time-1 map $X_{F_s}^{t}|_{t=1}$ of the Hamiltonian
vector field $X_{F_s}$ with $F_s=F_{0,s}+F_{1,s}$.
Using (\ref{M7}) and (\ref{M8}) we get
\begin{eqnarray}\label{022410}
\left\|F_{i,s}\right\|_{\rho_s}^{+}\leq {\gamma}^{-1} \epsilon_s^{-0.01}\left\|R_{i,s}\right\|_{\rho_s}^{+},\qquad i=0,1.
\end{eqnarray}
In view of (\ref{N7}) in Lemma \ref{051301}, (\ref{200}), (\ref{201}) and (\ref{022410}), we get
\begin{equation}\label{053120}
\left\|F_{0,s}\right\|_{\rho_s+\delta_s}\leq\frac{64}{e^2\delta_s^2}\left\|F_{0,s}\right\|_{\rho_s}^{+}\leq \epsilon_s^{0.95},
\end{equation}
and
\begin{equation}\label{053130}
\left\|F_{1,s}\right\|_{\rho_s+\delta_s}\leq\frac{64}{e^2\delta_s^2}\left\|F_{1,s}\right\|_{\rho_s}^{+}\leq \epsilon_s^{0.55}.
\end{equation}
Based on (\ref{053120}) and (\ref{053130}) one has
\begin{equation}\label{022501}
\left\|F_{s}\right\|_{\rho_s+\delta_s}\leq \epsilon_s^{0.55}
\end{equation}
and
\begin{equation}\label{Liu4}
\sup_{q\in D_s}\|X_{F_s}\|_{\sigma,r}\leq\epsilon_{s}^{0.52},
\end{equation}
which follows from (\ref{050907}) in Lemma \ref{063004}.

Since $$\epsilon_{s}^{0.52}\ll\frac{1}{\pi^2(s+1)^2}=d_{s+1}-d_s,$$ we have $\Phi_{s+1}:D_{s+1}\rightarrow D_{s}$ with
\begin{equation}\label{Liu4}
\left\|\Phi_{s+1}-id\right\|_{\sigma,r}\leq\sup_{q\in D_s}\|X_{F_s}\|_{\sigma,r}<\epsilon_{s}^{0.52},
\end{equation}
which is the estimate (\ref{203}). Moreover, from (\ref{Liu4}) we get
\begin{equation*}
\sup_{q\in D_s}\left\|DX_{F_s}-Id\right\|_{(\sigma,r)\rightarrow(\sigma,r)}\leq\frac{1}{d_s}\epsilon_{s}^{0.52}<\epsilon_{s}^{0.5},
\end{equation*}
and thus the estimate (\ref{204}) follows.

\textbf{Step. 4. The estimate of the remainder terms.}

Now we will estimate $R_{s+1}$. Recalling (\ref{051403}), $R_{s+1}=R_{2,s}+\mathcal{R}_s$, where
\begin{equation*}
\mathcal{R}_s=\int_{0}^1\{(1-t)\{N_s,F_s\}+R_{0,s}+R_{1,s},F_s\}\circ X_{F_s}^{t}dt+\int_0^1\{R_{2,s},F_s\}\circ X_{F_s}^tdt.
\end{equation*}
Rewrite
\begin{equation*}
R_{s+1}=R_{0,s+1}+R_{1,s+1}+R_{2,s+1},
\end{equation*}
and
\begin{eqnarray}
\mathcal{R}_s
&=&\label{051510}\int_0^1\{R_{0,s},F_s\}\circ X_{F_s}^tdt\\
&&\label{051511}+\int_0^1\{R_{1,s},F_s\}\circ X_{F_s}^tdt\\
&&\label{051513}+\int_0^1\{R_{2,s},F_s\}\circ X_{F_s}^tdt\\
&&\label{051520}+\int_0^1(1-t)\{\{N_s,F_s\},F_s\}\circ X_{F_s}^tdt.
\end{eqnarray}

Firstly we consider the term  $(\ref{051510})$. Note that the term $(\ref{051510})$ contributes to $R_{0,s+1},R_{1,s+1},R_{2,s+1}$ and we get
\begin{eqnarray}
&&\nonumber\left|\left|\int_0^1\{R_{0,s},F_s\}\circ X_{F_s}^tdt\right|\right|_{\rho_s+2\delta_s}
\nonumber\\ \nonumber&=&\left|\left|\sum_{n\geq1}\frac{1}{n!}\underbrace{\left\{\cdots\left\{R_{0,s},{F_s}\right\},\cdots,{F_s}\right\}}_{n-\mbox{fold}}\right|\right|_{\rho_s+2\delta_s}\\
\nonumber&\leq& \frac{4e}{\delta_s}\exp\left\{3\cdot\left(\frac{14400d}{\delta_s^2}\right)^{d}\cdot\exp\left\{d\left(\frac{24d}{\delta_s}\right)^{\frac1{\sigma-1}}\right\}\right\}\left\|{F_s}\right\|_{\rho_s+\delta_s}\left\|R_{0,s}\right\|_{\rho_s+\delta_s}\ \end{eqnarray}
where the last inequality follows from the proof of Lemma \ref{E1}.

In view of
\begin{equation*}
\delta_s=\frac{\rho_0}{(s+4)\ln^2(s+4)}
\end{equation*}
and by (\ref{042805*}),
there exists a constant $C(d)$ depending on $d$ only such that
\begin{eqnarray}
&&\nonumber\frac{4e}{\delta_s}\exp\left\{3\cdot
\left(\frac{14400d}{\delta_s^2}\right)^{d}\cdot\exp\left\{d\left(\frac{24d}{\delta_s}\right)^{\frac1{\sigma-1}}\right\}\right\}\\
&\leq&\nonumber\exp\left\{\exp\left\{C(d)\left({(s+4)\ln^2{(s+4)}}\right)^{\frac1{\sigma-1}}\right\}\right\}\\
&\leq&\label{052102}\exp\left\{0.01\cdot\left|\ln \epsilon_0\right|\cdot \exp\left\{s\ln\frac32\right\}\right\}=\epsilon_s^{-0.01},
\end{eqnarray}
where the last inequality is based on $\sigma>2$ and $\epsilon_0$ is very small depend on $d$.
Using (\ref{200}),  (\ref{022501}) and (\ref{052102}), one has
\begin{equation}\label{051203}
\left|\left|\int_0^1\{R_{0,s},F_s\}\circ X_{F_s}^tdt\right|\right|_{\rho_s+2\delta_s}\leq \epsilon_s^{1.54}.
\end{equation}
Following the proof of (\ref{052102}), we have
\begin{equation}\label{053140}
\exp\left\{10d\left(\frac{10}{\delta_s}\right)^{\frac 1{\sigma-1}}\cdot\exp\left\{\left(\frac{10}{\delta_s}\right)^{\frac1\sigma}\right\}
\right\}\leq \epsilon_s^{-0.01}.
\end{equation}
Consequently from (\ref{N6}), (\ref{051203}) and (\ref{053140}), one has
\begin{eqnarray}
\left|\left|\int_0^1\{R_{0,s},F_s\}\circ X_{F_s}^tdt\right|\right|_{\rho_s+3\delta_s}^{+}
\leq\epsilon_s^{1.53}\label{N11}.
\end{eqnarray}

Secondly, we consider the term (\ref{051511}) and write
\begin{eqnarray}
(\ref{051511})&=&
\nonumber\sum_{n\geq1}\frac{1}{n!}\underbrace{\{\cdots\{R_{1,s},{F_s}\},{F_s},\cdots,{F_s}\}}_{n-\mbox{fold}}\\
&=&\label{051505}\sum_{n\geq1}\frac{1}{n!}\underbrace{\{\cdots\{R_{1,s},{F}_{0,s}\},{F_s},\cdots,{F_s}\}}_{(n-1)-\mbox{fold}}\\
&&\label{051512}+\{R_{1,s},F_{1,s}\}\\
&&\label{051506}+\sum_{n\geq2}\frac{1}{n!}\underbrace{\{\cdots\{R_{1,s},F_{1,s}\},{F_s},\cdots,{F_s}\}}_{(n-1)-\mbox{fold}}.
\end{eqnarray}
Note that (\ref{051505}) and (\ref{051506}) contribute to $R_{0+},R_{1+},R_{2+}$, and (\ref{051512}) contributes to $R_{1+},R_{2+}$.
Moreover, following the proof of (\ref{N11}), one has
\begin{eqnarray}
\label{N12}\left|\left|(\ref{051505})\right|\right|_{\rho_s+3\delta_s}^{+},\left|\left|(\ref{051506})\right|\right|_{\rho_s+3\delta_s}^{+}
&\leq& \epsilon_{s}^{1.53},\\
\left|\left|(\ref{051512})\right|\right|_{\rho_s+3\delta_s}^{+}&\leq&\epsilon_s^{1.1}.
\end{eqnarray}

Thirdly, we consider the term (\ref{051513}) and write
\begin{eqnarray}
(\ref{051513})&=&\nonumber\sum_{n\geq1}\frac{1}{n!}\underbrace{\{\cdots\{R_{2,s},{F_s}\},{F_s},\cdots,{F_s}\}}_{n-\mbox{fold}}\\
&=&\label{051514}\{R_{2,s},F_{0,s}\}\\
&&+\label{051515}\{R_{2,s},F_{1,s}\}\\
&&+\label{051516}\sum_{n\geq2}\frac{1}{n!}\underbrace{\{\cdots\{R_{2,s},{F}_{0,s}\},{F_s},\cdots,{F_s}\}}_{(n-1)-\mbox{fold}}\\
&&\label{051517}+\{\{R_{1,s},F_{1,s}\},F_s\}\\
&&\label{051518}+\sum_{n\geq3}\frac{1}{n!}\underbrace{\{\cdots\{R_{2,s},F_{1,s}\},{F_s},\cdots,{F_s}\}}_{(n-1)-\mbox{fold}}.
\end{eqnarray}
Note that (\ref{051514}) and (\ref{051517}) contribute to $R_{1+},R_{2+}$, (\ref{051516}) and (\ref{051518}) contribute to $R_{0+},R_{1+},R_{2+}$, and (\ref{051515}) contributes to $R_{2+}$.

Similarly, one has
\begin{eqnarray}
\label{N15}||(\ref{051514})||_{\rho_s+3\delta_s}^{+},\left|\left|(\ref{051517})\right|\right|_{\rho_s+3\delta_s}^{+}&\leq& \epsilon_s^{0.93},\\
\label{N16}\left|\left|(\ref{051516})\right|\right|_{\rho_s+3\delta_s}^{+},\left|\left|(\ref{051518})\right|\right|_{\rho_s+3\delta_s}^{+}
&\leq& \epsilon_s^{1.53}
\end{eqnarray}
and
\begin{equation}
\label{N17}||(\ref{051515})||_{\rho_s+3\delta_s}^{+}
\leq\epsilon_s^{0.5}.
\end{equation}

Finally, we consider the term (\ref{051520}) and
write
\begin{eqnarray}
\nonumber(\ref{051520})&=&\sum_{n\geq2}\frac{1}{n!}\underbrace{\{\cdots\{N_s,{F_s}\},{F_s},\cdots,{F_s}\}}_{n-\mbox{fold}}\\
&=&\nonumber\sum_{n\geq2}\frac{1}{n!}\underbrace{\{\cdots\{-R_{0,s}-R_{1,s}+[R_{0,s}]+[R_{1,s}]\},{F},\cdots,{F}\}}_{(n-1)-\mbox{fold}}\qquad \\
&&\nonumber\mbox{(in view of (\ref{4.27}))}\\
&=&\label{051530}\sum_{n\geq2}\frac{1}{n!}\underbrace{\{\cdots\{-R_{0,s}+[R_{0,s}],{F_s}\},{F_s},\cdots,{F_s}\}}_{(n-1)-\mbox{fold}}\\
&&+\label{051531}\sum_{n\geq2}\frac{1}{n!}\underbrace{\{\cdots\{-R_{1,s}+[R_{1,s}],{F}_{0,s}\},{F_s},\cdots,{F_s}\}}_{(n-2)-\mbox{fold}}\\
&&\label{051532}+\{-R_{1,s}+[R_{1,s}],F_{1,s}\}\\
&&\label{051533}+\sum_{n\geq3}\frac{1}{n!}\underbrace{\{\cdots\{-R_{1,s}+[R_{1,s}],F_{1,s}\},{F_s},\cdots,{F_s}\}}_{(n-1)-\mbox{fold}}.
\end{eqnarray}
Note that (\ref{051530}),  (\ref{051531})  and (\ref{051533}) contribute to $R_{0+},R_{1+},R_{2+}$, and (\ref{051532}) contributes to $R_{1+},R_{2+}$.
Moreover, one has
\begin{eqnarray}
\left|\left|(\ref{051530})\right|\right|_{\rho_s+3\delta_s}^{+},\left|\left|(\ref{051531})\right|\right|_{\rho_s+3\delta_s}^{+},\left|\left|(\ref{051533})\right|\right|_{\rho_s+3\delta_s}^{+}
&\leq&\epsilon_s^{1.53}
\end{eqnarray}
and
\begin{equation}
\left|\left|(\ref{051532})\right|\right|_{\rho_s+3\delta_s}^{+}\leq\epsilon_s^{0.93}.
\end{equation}

Consequently we get
\begin{eqnarray*}
||R_{0,s+1}||_{\rho_{s+1}}^{+}
&\leq&10\epsilon_s^{1.53}\leq \epsilon_{s+1},\\
||R_{1,s+1}||_{\rho_{s+1}}^{+}
&\leq& 10\epsilon_s^{0.93}\leq \epsilon_{s+1}^{0.6}
\end{eqnarray*}
and
\begin{eqnarray*}
||R_{2,s+1}||_{\rho_{s+1}}^{+}
\leq(1+d_s)\epsilon_0+\epsilon_s^{0.5}\leq(1+d_{s+1})\epsilon_0,
\end{eqnarray*}
which are just the assumptions (\ref{200})-(\ref{202}) at stage $s+1$.

\textbf{Step. 5. $(\rho_s+2\delta_s)$-T\"{o}plitz-Lipschitz property of the solution of homological equation}

Recall the solutions of homological equation are given in (\ref{052001}) and (\ref{052002}).
Without loss of generality, it suffices to prove that
given any $\textbf n,\textbf m,\textbf l\in\mathbb{Z}^d$ and $t\in\mathbb{Z}$ the following limit exists
\begin{equation*}
\lim_{t\rightarrow\infty}\frac{\partial^2 F_{0,s}}{\partial q_{\textbf n+t\textbf l}\partial{\bar q_{\textbf m+t\textbf l}}}
\end{equation*}
and there exists $K_{s+1}$ such that when $|t|>K_{s+1}$
\begin{eqnarray}\label{060206}
\left\|\frac{\partial^2F_{0,s}}{\partial q_{\textbf n+t\textbf l}\partial{\bar q_{\textbf m-t\textbf l}}}-\lim_{t\rightarrow\infty}\frac{\partial^2F_{0,s}}{\partial q_{\textbf n+t\textbf l}\partial{\bar q_{\textbf m-t\textbf l}}}\right\|_{\rho_s+2\delta_s}^*
\leq\frac{\epsilon_{s}^{0.95}}{|t|}.
\end{eqnarray}
For any $a,k,k'\in\mathbb{N}^{\mathbb{Z}^d}$, we suppose that
\begin{equation*}
\prod_{i\geq 3}\lfloor\textbf n_{i}^*(a,k,k')\rfloor \leq \epsilon_s^{-0.01},
\end{equation*}
which follows the proof of (\ref{M7}).

Note that if
\begin{equation}\label{052401}
|t|\geq 2\left(\epsilon_s^{-0.02}+\left\|\textbf n\right\|^2+\left\|\textbf m\right\|^2\right),
\end{equation}
then one has
\begin{equation*}
\lfloor \textbf n+t\textbf l\rfloor,\lfloor \textbf m+t\textbf l\rfloor> \epsilon_s^{-0.01}\geq \lfloor\textbf n_3^*(a,k,k')\rfloor.
\end{equation*}In view of momentum conservation (\ref{050901}), we have
\begin{equation*}
\frac{\partial^2\mathcal{M}_{akk'}}{\partial q_{\textbf n+t\textbf l}\partial{\bar q_{\textbf m-t\textbf l}}}=0
\end{equation*}unless
\begin{equation*}
k_{\textbf n+t\textbf l}=k'_{\textbf m+t\textbf l}=1.
\end{equation*}
Here we assume that
\begin{equation*}
\textbf n+t\textbf l=\textbf n_1^*(a,k,k'),\qquad \textbf m+t\textbf l=\textbf n_2^*(a,k,k').
\end{equation*}
Furthermore, we have
\begin{equation}\label{052103}
\left\|\textbf n+t\textbf l\right\|^2-\left\|\textbf m+t\textbf l\right\|^2=2t\left(\textbf n-\textbf m\right)\cdot\textbf l+\left\|\textbf n\right\|^2-\left\|\textbf m\right\|^2.
\end{equation}
\textbf{Case 1.}
 \begin{equation*}
\left(\textbf n-\textbf m\right)\cdot\textbf l\neq 0.
\end{equation*}
In view of (\ref{052401}) and (\ref{052103}), one has
\begin{equation}\label{053150}
\sum_{\textbf j\in\mathbb{Z}^d}\left(k_{\textbf j}-k'_{\textbf j}\right)\left(\left\|\textbf j\right\|^2+\widetilde V_{\textbf j,s}\right)\geq |t|.
\end{equation}
Following the proof of (\ref{052102}), one gets
\begin{equation}\label{053160}
\left(\frac{12}{e\delta_s}\right)^2\cdot \exp\left\{\left(\frac{3600d}{\delta_s^2}\right)^d\cdot \exp\left\{d{\left(\frac{12d}{\delta_s} \right)}^{\frac{1}{\sigma-1}}\right\}\right\}\leq \epsilon_s^{-0.01}.
\end{equation}
Then we have
\begin{eqnarray*}
&&\left\|\frac{\partial^2}{{\partial q_{\textbf n+t\textbf l}\partial \bar q_{\textbf m-t\textbf l}}}\left(\sum_{a,k,k'\in\mathbb{N}^{\mathbb{Z}^d}}F_{akk'}{ \mathcal{M}_{akk'}}\right)\right\|_{\rho_s+2\delta_s}^*\\
&=&\left\|\sum_{a,k,k'\in\mathbb{N}^{\mathbb{Z}^d}}\frac{R_{akk'}}{\sum_{\textbf j\in\mathbb{Z}^d}\left(k_{\textbf j}-k'_{\textbf j}\right)\left(\left\|\textbf j\right\|^2+\widetilde V_{\textbf j}\right)}\cdot\frac{ \partial^2\mathcal{M}_{akk'}}{\partial q_{\textbf n+t\textbf l}\partial \bar q_{\textbf m+t\textbf l}}\right\|_{\rho_s+2\delta_s}^*\\
&\leq&\frac{\left\|R_0\right\|^*_{\rho_s+2\delta_s}}{|t|}\label{052202}\qquad (\mbox{by (\ref{053150})})\\
&\leq&\frac{\epsilon_s^{-0.01}\left\|R_0\right\|_{\rho_s+\delta_s}}{|t|}\qquad  (\mbox{by (\ref{003}) in Lemma \ref{052302} and (\ref{053160})})\\
&\leq&\frac{\epsilon_s^{0.98}}{|t|},
\end{eqnarray*}
where the last inequality uses (\ref{N7}) and (\ref{200}).

Hence one has
\begin{equation*}
\lim_{t\rightarrow\infty}\frac{\partial^2F_{0,s}}{{\partial q_{\textbf n+t\textbf l}\partial \bar q_{\textbf m-t\textbf l}}}=0
\end{equation*}
and
\begin{equation*}
\left\|\frac{\partial^2F_{0,s}}{{\partial q_{\textbf n+t\textbf l}\partial \bar q_{\textbf m-t\textbf l}}}-\lim_{t\rightarrow\infty}\frac{\partial^2F_{0,s}}{{\partial q_{\textbf n+t\textbf l}\partial \bar q_{\textbf m-t\textbf l}}}\right\|_{\rho_s+2\delta_s}^*\leq \frac{\epsilon_s^{0.98}}{|t|}.
\end{equation*}

\textbf{Case 2.}

\begin{equation*}
\left(\textbf n-\textbf m\right)\cdot\textbf l=0.
\end{equation*}
In view of (\ref{052103}), one has
\begin{eqnarray*}
&&\lim_{t\rightarrow\infty}\sum_{\textbf j\in\mathbb{Z}^d}\left(k_{\textbf j}-k'_{\textbf j}\right)\left(\left\|\textbf j\right\|^2+\widetilde V_{\textbf j,s}\right)\\
&=& \sum_{\textbf j\in\mathbb{Z}^d\atop \textbf j\neq \textbf n+t\textbf l,\textbf m+t\textbf l}\left(k_{\textbf j}-k'_{\textbf j}\right)\left(\left\|\textbf j\right\|^2+\widetilde V_{\textbf j,s}\right)+\left\|\textbf n\right\|^2-\left\|\textbf m\right\|^2\\
&:=&\mathcal{A}^{-1}
\end{eqnarray*}
and then
\begin{eqnarray*}
\lim_{t\rightarrow\infty}F_{akk'}
=\mathcal{A}\cdot{\lim_{t\rightarrow\infty}R_{akk'}}.
\end{eqnarray*}
Furthermore, we have
\begin{eqnarray*}
&&\left|F_{akk'}-\lim_{t\rightarrow\infty}F_{akk'}\right|\\
&\leq &\left|\left(\frac{1}{\sum_{\textbf j\in\mathbb{Z}^d}(k_{\textbf j}-k'_\textbf j)(\left\|\textbf j\right\|^2+\widetilde V_{\textbf j,s})}-\mathcal{A}\right)R_{akk'}\right|\\
&&+\left|\left(\frac{1}{\sum_{\textbf j\in\mathbb{Z}^d}(k_{\textbf j}-k'_\textbf j)(\left\|\textbf j\right\|^2+\widetilde V_{\textbf j,s})}\right)\left(R_{akk'}-\lim_{t\rightarrow\infty}R_{akk'}\right)\right|\\
&=&\left|\frac{\left(\widehat V_{\textbf n+t\textbf l,s}-\widehat V_{\textbf m+t\textbf l,s}\right)R_{akk'}}{\left(\sum_{\textbf j\in\mathbb{Z}^d}(k_{\textbf j}-k'_\textbf j)(\left\|\textbf j\right\|^2+\widetilde V_{\textbf j,s})\right)^2}\right|\\
&&+\left|\left(\frac{1}{\sum_{\textbf j\in\mathbb{Z}^d}(k_{\textbf j}-k'_\textbf j)(\left\|\textbf j\right\|^2+\widetilde V_{\textbf j,s})}\right)\left(R_{akk'}-\lim_{t\rightarrow\infty}R_{akk'}\right)\right|.
\end{eqnarray*}
In view of (\ref{053190}), one has
\begin{equation}\label{061101}
\left|\widehat V_{\textbf n+t\textbf l,s}-\widehat V_{\textbf m+t\textbf l,s}\right|\leq \frac4{|t|}.
\end{equation}
Using (\ref{060403}) and (\ref{061101}) one gets
\begin{equation}\label{060405}
\left|\frac{\left(\widehat V_{\textbf n+t\textbf l,s}-\widehat V_{\textbf m+t\textbf l,s}\right)R_{akk'}}{\left(\sum_{\textbf j\in\mathbb{Z}^d}(k_{\textbf j}-k'_\textbf j)(\left\|\textbf j\right\|^2+\widetilde V_{\textbf j,s})\right)^2}\right|\leq \frac{4\epsilon_s^{-0.02}\left|R_{akk'}\right|}{|t|}.
\end{equation}
By (\ref{060403}) again,  we have
\begin{equation}\label{060406}
\left|\left(\frac{1}{\sum_{\textbf j\in\mathbb{Z}^d}(k_{\textbf j}-k'_\textbf j)(\left\|\textbf j\right\|^2+\widetilde V_{\textbf j,s})}\right)\left(R_{akk'}-\lim_{t\rightarrow\infty}R_{akk'}\right)\right|\leq \epsilon_s^{-0.01}\left|R_{akk'}-\lim_{t\rightarrow\infty}R_{akk'}\right|.
\end{equation}
Using (\ref{T1}), (\ref{060405}) and (\ref{060406}), we obtain
\begin{eqnarray}
\nonumber\left\|\frac{\partial^2F_{0,s}}{\partial q_{\textbf n+t\textbf l}\partial{\bar q_{\textbf m-t\textbf l}}}-\lim_{t\rightarrow\infty}\frac{\partial^2F_{0,s}}{\partial q_{\textbf n+t\textbf l}\partial{\bar q_{\textbf m-t\textbf l}}}\right\|_{\rho_s+2\delta_s}^*
\leq\frac{\epsilon_{s}^{0.95}}{|t|}.
\end{eqnarray}

Finally, we point out that
in view of momentum conservation (\ref{050901}), one has
\begin{equation}\label{060103}
\frac{\partial^2 \mathcal{M}_{akk'}}{\partial q_{\textbf n+t\textbf l}\partial{\bar q_{\textbf m+t\textbf l}}}=0,
\end{equation}if
\begin{equation}\left\|\textbf n-\textbf m\right\|\geq \epsilon_s^{-0.01}+1.
\end{equation}

\textbf{Step 6. $\rho_{s+1}$--T\"{o}plitz-Lipschitz property of the remainder term}

It suffices to prove that $\left\{F_{0,s},R_{0,s}\right\}$ satisfies $\rho_{s+1}$--T\"{o}plitz-Lipschitz property. For simplicity, we write $F_{0,s}:=F$ and $R_{0,s}:=R$.

Given any $\textbf n,\textbf m.\textbf l\in\mathbb{Z}^d$ and $t\in\mathbb{Z}$, one has
\begin{eqnarray}\nonumber
\frac{\partial^{2}\{F,R\}}{\partial {q}_{{\textbf n}+t\textbf l}\partial \bar{q}_{{\textbf m}+t\textbf l}}
 &=&\frac{\partial^{2}\underset{\textbf j\in\mathbb{Z}^{d}}{\sum}\left(\frac{\partial F}{\partial q_{\textbf j}}\frac{\partial R}{\partial \bar{q}_{\textbf j}}-\frac{\partial F }{\partial\bar{q}_{\textbf j}}\frac{\partial R}{\partial{q}_{\textbf j}}\right)}{\partial {q}_{{\textbf n}+t\textbf l}\partial \bar{q}_{{\textbf m}+t\textbf l}}\\
 &= &\label{060101.}
  \underset{\textbf j\in\mathbb{Z}^{d}}{\sum}\left(
\frac{\partial^{2}F}{\partial \bar{q}_{\textbf m+t\textbf l}\partial{q}_{\textbf j+t\textbf l}} \cdot\frac{\partial^{2} R}{\partial \bar{q}_{\textbf j+t\textbf l}\partial{q}_{\textbf n+t\textbf l}}\right)\\
&&+\underset{\textbf j\in\mathbb{Z}^{d}}{\sum}\left(\frac{\partial^{2}F}{\partial q_{\textbf n+t\textbf l}\partial{q}_{\textbf j-t\textbf l}} \cdot\frac{\partial^{2} R }{\partial \bar{q}_{\textbf j-t\textbf l}\partial\bar{q}_{\textbf m+t\textbf l}}\right)\\
 &&-\underset{\textbf j\in\mathbb{Z}^{d}}{\sum}\left(
\frac{\partial^{2}F}{\partial \bar{q}_{\textbf m+t\textbf l}\partial\bar{q}_{\textbf j-t\textbf l}} \cdot\frac{\partial^{2} R }{\partial {q}_{\textbf j-t\textbf l}\partial{q}_{\textbf n+t\textbf l}}\right)\\
 &&- \underset{\textbf j\in\mathbb{Z}^{d}}{\sum}\left(
\frac{\partial^{2}F}{\partial {q}_{\textbf n+t\textbf l}\partial\bar{q}_{\textbf j+t\textbf l}} \cdot\frac{\partial^{2} R }{\partial {q}_{\textbf j+t\textbf l}\partial\bar{q}_{\textbf m+t\textbf l}}\right)\\
 &&\label{060102}+\underset{\textbf j\in\mathbb{Z}^{d}}{\sum}\left(
\frac{\partial^{3}F}{\partial q_{\textbf n+t\textbf l}\partial\bar{q}_{\textbf m+t\textbf l}\partial{q}_{\textbf j}} \cdot\frac{\partial R }{\partial \bar{q}_{\textbf j}}\right)\\&&
+\sum_{\textbf j\in\mathbb{Z}^d}\left(
\frac{\partial F}{\partial{q}_{\textbf j}} \cdot\frac{\partial^{3} R}{\partial \bar{q}_{\textbf j}\partial{q}_{\textbf n+t\textbf l}\partial\bar{q}_{\textbf m+t\textbf l}}\right)\\
 &&-\underset{\textbf j\in\mathbb{Z}^{d}}{\sum}\left(
\frac{\partial F}{\partial\bar{q}_{\textbf j}} \cdot\frac{\partial^{3} R }{\partial {q}_{\textbf j}\partial{q}_{\textbf n+t\textbf l}\partial\bar{q}_{\textbf m+t\textbf l}}\right)\\&&-\sum_{\textbf j\in\mathbb{Z}^d}\left(\frac{\partial^{3}F}{\partial {q}_{\textbf n+t\textbf l}\partial\bar{q}_{\textbf m+t\textbf l}\partial\bar{q}_{\textbf j}} \cdot\frac{\partial R }
{\partial {q}_{\textbf j}}\right).
\end{eqnarray}
Without loss of generality we only need to show the terms (\ref{060101.}) and (\ref{060102}) satisfies $\rho_{s+1}$--T\"{o}plitz-Lipschitz property.

 Now we will consider (\ref{060101.}) firstly. In view of (\ref{060103}), one has
\begin{equation}\label{060105}
\left\|\textbf m-\textbf j\right\|\leq \epsilon_{s}^{-0.01}\qquad \mbox{and }\qquad \left\|\textbf n-\textbf j\right\|\leq \epsilon_{s}^{-0.01},
\end{equation}
otherwise
\begin{equation}\label{060201}
\frac{\partial^{2}F}{\partial \bar{q}_{\textbf m+t\textbf l}\partial{q}_{\textbf j+t\textbf l}} \cdot\frac{\partial^{2} R}{\partial \bar{q}_{\textbf j+t\textbf l}\partial{q}_{\textbf n+t\textbf l}}=0.
\end{equation}
Denote that
\[
\underset{t\rightarrow\infty}{\lim}\frac{\partial^{2}F}{\partial \bar{q}_{{\textbf m}+t\textbf l}\partial {q}_{{\textbf j}+t\textbf l}} = F^{11}_{\textbf j \textbf m},\qquad \underset{t\rightarrow\infty}{\lim}\frac{\partial^{2}R}{\partial {q}_{{\textbf n}+t\textbf l}\partial \bar{q}_{{\textbf j}+t\textbf l}} = R^{11}_{\textbf n \textbf j}.\]
Then one has
\begin{eqnarray}\label{060202}
\lim_{t\rightarrow\infty}\underset{\textbf j\in\mathbb{Z}^{d}}{\sum}\left(
\frac{\partial^{2}F}{\partial \bar{q}_{\textbf m+t\textbf l}\partial{q}_{\textbf j+t\textbf l}} \cdot\frac{\partial^{2} R}{\partial \bar{q}_{\textbf j+t\textbf l}\partial{q}_{\textbf n+t\textbf l}}\right)=\sum_{\textbf j\in\mathbb{Z}^d\atop (\ref{060105})\ satisfies} F^{11}_{\textbf j\textbf m}R^{11}_{\textbf n\textbf j}.
\end{eqnarray}
Next we will estimate
\begin{equation}\label{060203}
\underset{\textbf j\in\mathbb{Z}^{d}}{\sum}\left(
\frac{\partial^{2}F}{\partial \bar{q}_{\textbf m+t\textbf l}\partial{q}_{\textbf j+t\textbf l}} \cdot\frac{\partial^{2} R}{\partial \bar{q}_{\textbf j+t\textbf l}\partial{q}_{\textbf n+t\textbf l}}\right)-\lim_{t\rightarrow\infty}\underset{\textbf j\in\mathbb{Z}^{d}}{\sum}\left(
\frac{\partial^{2}F}{\partial \bar{q}_{\textbf m+t\textbf l}\partial{q}_{\textbf j+t\textbf l}} \cdot\frac{\partial^{2} R}{\partial \bar{q}_{\textbf j+t\textbf l}\partial{q}_{\textbf n+t\textbf l}}\right).
\end{equation}
In view of (\ref{060201}) and (\ref{060202}), one has
\begin{eqnarray*}
\left\|(\ref{060203})\right\|_{\rho_{s+1}}^*&=&\left\|\sum_{\textbf j\in\mathbb{Z}^d\atop (\ref{060105})\ satisfies}\left(\frac{\partial^{2}F}{\partial \bar{q}_{\textbf m+t\textbf l}\partial{q}_{\textbf j+t\textbf l}} \cdot\frac{\partial^{2} R}{\partial \bar{q}_{\textbf j+t\textbf l}\partial{q}_{\textbf n+t\textbf l}}- F^{11}_{\textbf j\textbf m}R^{11}_{\textbf n\textbf j}\right)\right\|_{\rho_{s+1}}^*\\
&\leq&\left\|\sum_{\textbf j\in\mathbb{Z}^d\atop (\ref{060105})\ satisfies}\left(\frac{\partial^{2}F}{\partial \bar{q}_{\textbf m+t\textbf l}\partial{q}_{\textbf j+t\textbf l}}- F^{11}_{\textbf j\textbf m}\right)R^{11}_{\textbf n\textbf j}\right\|_{\rho_{s+1}}^*\\
&&+\left\|\sum_{\textbf j\in\mathbb{Z}^d\atop (\ref{060105})\ satisfies}\frac{\partial^{2}F}{\partial \bar{q}_{\textbf m+t\textbf l}\partial{q}_{\textbf j+t\textbf l}} \cdot\left(\frac{\partial^{2} R}{\partial \bar{q}_{\textbf j+t\textbf l}\partial{q}_{\textbf n+t\textbf l}}- R^{11}_{\textbf n\textbf j}\right)\right\|_{\rho_{s+1}}^*\\
&\leq&\sum_{\textbf j\in\mathbb{Z}^d\atop (\ref{060105})\ satisfies}\left\|\frac{\partial^{2}F}{\partial \bar{q}_{\textbf m+t\textbf l}\partial{q}_{\textbf j+t\textbf l}}- F^{11}_{\textbf j\textbf m}\right\|_{\rho_{s+1}}^*\left\|R^{11}_{\textbf n\textbf j}\right\|_{\rho_{s+1}}^*\\
&&+\sum_{\textbf j\in\mathbb{Z}^d\atop (\ref{060105})\ satisfies}\left\|\frac{\partial^{2}F}{\partial \bar{q}_{\textbf m+t\textbf l}\partial{q}_{\textbf j+t\textbf l}}\right\|_{\rho_{s+1}}^* \cdot\left\|\frac{\partial^{2} R}{\partial \bar{q}_{\textbf j+t\textbf l}\partial{q}_{\textbf n+t\textbf l}}- R^{11}_{\textbf n\textbf j}\right\|_{\rho_{s+1}}^*,
\end{eqnarray*}
where the last inequality is based on (\ref{002}) in Lemma \ref{l2}.

On one hand, using (\ref{060206}) one has
\begin{equation}
\left\|\frac{\partial^{2}F}{\partial \bar{q}_{\textbf m+t\textbf l}\partial{q}_{\textbf j+t\textbf l}}- F^{11}_{\textbf j\textbf m}\right\|_{\rho_{s+1}}^*\leq \frac{\epsilon_s^{0.95}}{|t|}
\end{equation}
and
\begin{equation}
\left\|\frac{\partial^{2} R}{\partial \bar{q}_{\textbf j+t\textbf l}\partial{q}_{\textbf n+t\textbf l}}- R^{11}_{\textbf n\textbf j}\right\|_{\rho_{s+1}}^*\leq \frac{\epsilon_s}{|t|}.
\end{equation}
which follows (\ref{T1}) and (\ref{060208}) in Lemma \ref{060207}.

On the other hand, using (\ref{003}) in Lemma \ref{052302} one gets
\begin{equation*}
\left\|\frac{\partial^{2}F}{\partial \bar{q}_{\textbf m+t\textbf l}\partial{q}_{\textbf j+t\textbf l}}\right\|_{\rho_{s+1}}^* \leq \epsilon_{s}^{0.95}
\end{equation*}
and
\begin{equation*}
\left\|R^{11}_{\textbf n\textbf j}\right\|_{\rho_{s+1}}^* \leq \epsilon_{s}^{0.95},
\end{equation*}
which follows from
\begin{equation*}
\left\|\frac{\partial^2R}{\partial q_{\textbf n+t\textbf l}\partial \bar q_{\textbf j+t\textbf l}}\right\|_{\rho_{s+1}}^*\leq \epsilon_s^{-0.01}\left\|R\right\|_{\rho_s}\leq \epsilon_s^{0.95}.
\end{equation*}
Hence, we have
\begin{equation*}
\left\|(\ref{060203})\right\|_{\rho_{s+1}}^*\leq \frac{\epsilon_s^{1.90}}{|t|}.
\end{equation*}

Next we will estimate
\begin{equation}\label{060302}
\underset{\textbf j\in\mathbb{Z}^{d}}{\sum}\left(\frac{\partial^{3}F}{\partial q_{\textbf n+t\textbf l}\partial\bar{q}_{\textbf m+t\textbf l}\partial{q}_{\textbf j}} \cdot\frac{\partial R }{\partial \bar{q}_{\textbf j}}
\right)-\lim_{t\rightarrow\infty}\underset{\textbf j\in\mathbb{Z}^{d}}{\sum}\left(\frac{\partial^{3}F}{\partial q_{\textbf n+t\textbf l}\partial\bar{q}_{\textbf m+t\textbf l}\partial{q}_{\textbf j}} \cdot\frac{\partial R }{\partial \bar{q}_{\textbf j}}
\right).
\end{equation}
Firstly one has
\begin{eqnarray}\nonumber
&&\|(\ref{060302})\|^*_{\rho_{s+1}}\nonumber\\
&=&
\left\|\sum_{\textbf j\in\mathbb{Z}^d}\left(\frac{\partial\left(\frac{\partial^{2}F}{\partial q_{\textbf n+t\textbf l}\partial\bar{q}_{\textbf m+t\textbf l}}-F_{\textbf n\textbf m}^{11}\right)}{\partial{q}_{\textbf j}}\right)\cdot\frac{\partial R }{\partial \bar{q}_{\textbf j}}\right\|^*_{\rho_{s+1}}\nonumber\\
\label{060304} &\leq&\underset{\textbf j\in\mathbb{Z}^d}{\sup}\left(\left\|\frac{\partial R }{\partial \bar{q}_{\textbf j}}\cdot e^{\rho_{s+1}\ln^{\sigma}\lfloor \textbf j\rfloor}\right\|^*_{\rho_{s+1}}\right)\\
\label{060303}&&\times\left( \sum_{\textbf j\in\mathbb{Z}^{d}}e^{-\rho_{s+1}\ln^{\sigma}\lfloor \textbf j\rfloor}\cdot \left\|\frac{\partial\left(\frac{\partial^{2}F}{\partial q_{\textbf n+t\textbf l}\partial\bar{q}_{\textbf m+t\textbf l}}-F_{\textbf n\textbf m}^{11}\right)}{\partial{q}_{ \textbf j}}\right\|^*_{\rho_{s+1}}\right) ,
\end{eqnarray}
where the last inequality is based on (\ref{002}).\\

On one hand, for any $\textbf j\in\mathbb{Z}^d$ one gets
\begin{eqnarray}
&&\nonumber \left\|\frac{\partial R }{\partial \bar{q}_{\textbf j}}\cdot e^{\rho_{s+1}\ln^{\sigma}\lfloor \textbf j\rfloor}\right\|^*_{\rho_{s+1}}\\
\nonumber  &=&\underset{\left\|q\right\|_{\sigma,\rho_{s+1}}\leq 1}{\sup}\left|{\frac{\partial R}{\partial{\bar q}_{\textbf j}}\cdot e^{\rho_{s+1}\ln^{\sigma}\lfloor \textbf j\rfloor}}\right|\\
\nonumber &\leq& \sup_{\left\|q\right\|_{\sigma,\rho_{s+1}}\leq 1}
\left\|X_R\right\|_{\sigma,\rho_{s+1}}\\
\nonumber&\leq& \exp\left\{10d\left(\frac{10000d}{\delta_{s}^2}\right)^d\cdot \exp\left\{d{\left(\frac{20d}{\delta_{s}} \right)}^{\frac{1}{\sigma-1}}\right\}\right\}\left\|R\right\|_{\rho_s+2\delta_s}\\
\nonumber&&(\mbox{in view of (\ref{050907.})})\\
\nonumber &\leq& \epsilon^{0.95}_{s},
\end{eqnarray}
where the last inequality is based on (\ref{N7}), (\ref{200}) and (\ref{053160}). Therefore one has
\begin{equation}
(\ref{060304})\leq \epsilon^{0.95}_{s}.
\end{equation}

On the other hand,
denote
\[
\frac{\partial^{2}F}{\partial q_{\textbf n+t\textbf l}\partial\bar{q}_{\textbf m+t\textbf l}}-F_{\textbf n\textbf m}^{11}= \sum_{a,k,k'\in\mathbb{N}^{\mathbb{Z}^d}}B_{akk'}\prod_{{\textbf w}\in\mathbb{Z}^d}I_{\textbf w}(0)^{a_{\textbf w}}q_{\textbf w}^{k_{\textbf w}}\bar q_{\textbf w}^{k_{\textbf w}'}.
\]
It follows easily that
\[\frac{\partial\left(\frac{\partial^{2}F}{\partial q_{\textbf n+t\textbf l}\partial\bar{q}_{\textbf m+t\textbf l}}-F_{\textbf n\textbf m}^{11}\right)}{\partial{q}_{\textbf j}}=\sum_{a,k,k'\in\mathbb{N}^{\mathbb{Z}^d}}k_{\textbf j}B_{akk'}\prod_{{\textbf w}\in\mathbb{Z}^d}I_{\textbf w}(0)^{a_{\textbf w}}q_{\textbf w}^{k_{\textbf w}-e_{\textbf j}}\bar q_{\textbf w}^{k_{\textbf w}'}.
\]
Here we assume that $ |k_{\textbf j}|\geq 1$. Otherwise,
\[
\frac{\partial\left(\frac{\partial^{2}F}{\partial q_{\textbf n+t\textbf l}\partial\bar{q}_{\textbf m+t\textbf l}}-F_{\textbf n\textbf m}^{11}\right)}{\partial{q}_{\textbf j}}=0.
\]
Then one has
\begin{eqnarray}\nonumber
&&e^{-(\rho_{s}+\frac52\delta_s)\ln^{\sigma}\lfloor \textbf j\rfloor}\left\|\frac{\partial\left(\frac{\partial^{2}F}{\partial q_{\textbf n+t\textbf l}\partial\bar{q}_{\textbf m+t\textbf l}}-F_{\textbf n\textbf m}^{11}\right)}{\partial{q}_{ \textbf j}}\right\|^*_{\rho_{s+1}}\\
\nonumber&=& \sum_{a,k,k'\in\mathbb{N}^{\mathbb{Z}^d}}\left|k_{\textbf j}B_{akk'}\right|\exp\left\{-\underset{\textbf w\in\mathbb{Z}^{d} }{\sum}\left(2ra_{\textbf w}+\rho_{s+1}\left(k_{\textbf w}+k_{\textbf w}'\right)\right)\ln^{\sigma}\lfloor \textbf n\rfloor\right\}\cdot e^{\frac12\delta_{s}\ln^{\sigma}\lfloor \textbf j\rfloor}\\
&\leq &\sum_{a,k,k'\in\mathbb{N}^{\mathbb{Z}^d}}\left|k_{\textbf j}B_{akk'}\right|\exp\left\{-\underset{\textbf w\in\mathbb{Z}^{d} }{\sum}\left(2ra_{\textbf w}+\left(\rho_{s}+\frac52\delta_s\right)\left(k_{\textbf w}+k_{\textbf w}'\right)\right)\ln^{\sigma}\lfloor \textbf n\rfloor\right\}\label{060430}\nonumber
\\
\nonumber&\leq&\sup_{a,k,k'\in\mathbb{N}^{\mathbb{Z}^d}}\left(\left( \sum_{\textbf w\in\mathbb{Z}^d}\left(k_{\textbf w}+k_{\textbf w}'\right)\ln^{\sigma}\lfloor \textbf w\rfloor\right)
\exp\left\{-\frac12\delta_{s}\underset{\textbf w\in\mathbb{Z}^{d}}{\sum}(k_{\textbf w}+k_{\textbf w}')\ln^{\sigma}\lfloor \textbf w\rfloor\right\}\right)\\&&
\times\left(\sum_{a,k,k'\in\mathbb{N}^{\mathbb{Z}^d}}\left|B_{akk'}\right|\exp\left\{-\underset{\textbf w\in\mathbb{Z}^{d} }{\sum}\left(2ra_{\textbf w}+\left(\rho_{s}+2\delta_s\right)\left(k_{\textbf w}+k_{\textbf w}'\right)\right)\ln^{\sigma}\lfloor \textbf w\rfloor\right\}\right)\nonumber\\
\nonumber\\
\nonumber&\leq& \left(\frac 2{e\delta_s}\right)\cdot\left(\left\|\frac{\partial^{2}F}{\partial q_{\textbf n+t\textbf l}\partial\bar{q}_{\textbf m+t\textbf l}}-F_{\textbf n\textbf m}^{11}\right\|^*_{\rho_{s}+2\delta_{s}}\right)
\end{eqnarray}
where the last inequality uses (\ref{042805*}) and (\ref{d2}) in Definition \ref{083103}.

Thus we have
\begin{eqnarray}
\nonumber (\ref{060303}) &\leq& \left(\frac 2{e\delta_s}\right)\cdot\left(\left\|\frac{\partial^{2}F}{\partial q_{\textbf n+t\textbf l}\partial\bar{q}_{\textbf m+t\textbf l}}-F_{\textbf n\textbf m}^{11}\right\|^*_{\rho_{s}+2\delta_{s}}\right)\sum_{\textbf j\in\mathbb{Z}^{d}}e^{-\frac12\delta_s\ln^{\sigma}\lfloor \textbf j\rfloor} \\
\nonumber
\nonumber &\leq&\left(\frac{2}{e\delta_{s}}\right)
\cdot\frac{ \epsilon^{0.95}_{s}}{|t|}\cdot\left( \left(\frac{12d}{\delta_s}\right)^d\cdot \exp\left\{d{\left(\frac{4d}{\delta_s} \right)}^{\frac{1}{\sigma-1}}\right\}\right)\\
 \nonumber &&(\mbox{in view of (\ref{060206}) and (\ref{0418011})})\\
\nonumber &\leq& \frac{ \epsilon^{0.93}_{s}}{|t|},
\end{eqnarray}
where the last inequality is based on (\ref{052102}).

Therefore, we have
\[
\|(\ref{060302})\|^*_{\rho_{s+1}}\leq \frac{\epsilon^{1.88}_{s}}{|t|}\leq \frac{\epsilon_{s+1}}{|t|}.
\]

\textbf{Step 7. The frequency shift.}

In view of (\ref{051402}), the new normal form $N_{s+1}$ is given by
\begin{equation}
N_{s+1}=N_s+[R_{0,s}]+[R_{1,s}].
\end{equation}Note that $[R_{0,s}]$ (by (\ref{051501.})) is a constant which does not affect the Hamiltonian vector field. Moreover, in view of (\ref{051502}), we denote by
\begin{equation}
\omega_{\textbf j,s}=\left\|\textbf j\right\|^2+\widetilde V_{\textbf j,s}+\sum_{a\in\mathbb{N}^{\mathbb{Z}^d}}B_{a00}^{(\textbf j)}\mathcal{M}_{a00},
\end{equation}
where the term $$\sum_{a\in\mathbb{N}^{\mathbb{Z}^d}}B_{a00}^{(\textbf j)}\mathcal{M}_{a00}$$ is the so-called frequency shift which will be estimated below.
For any $\textbf j\in\mathbb{Z}^d$, one has
\begin{eqnarray}
\nonumber\left|B^{(\textbf j)}_{a00}\right|
\nonumber&\leq& \left\|R_{1,s+1}\right\|_{\rho_{s+1}}^+e^{2\rho_{s+1}\left(\sum_{\textbf n\in\mathbb{Z}^d}a_{\textbf n}\ln^{\sigma}\lfloor \textbf n\rfloor+\ln^{\sigma}\lfloor \textbf j\rfloor-\ln^{\sigma}\lfloor \textbf n_1^{*}(a,0,0;\textbf j)\rfloor\right)}\\
\label{M12}&<&\epsilon_{s+1}^{0.6}\cdot e^{2\rho_{s+1}\left(\sum_{\textbf n\in\mathbb{Z}^d}a_{\textbf n}\ln^{\sigma}\lfloor \textbf n\rfloor+\ln^{\sigma}\lfloor \textbf j\rfloor-\ln^{\sigma}\lfloor \textbf n_1^{*}(a,0,0;\textbf j)\rfloor\right)}.
\end{eqnarray}
In view of (\ref{M13}) and (\ref{M12}) ,
we obtain
\begin{eqnarray}
&&\nonumber\left|\sum_{a\in\mathbb{N}^{\mathbb{Z}^d}}B^{(\textbf j)}_{a00}\mathcal{M}_{a00}\right|\\
\nonumber&\leq & \epsilon_{s+1}^{0.6}\sum_{a\in\mathbb{N}^{\mathbb{Z}^d}}e^{2\rho_{s+1}\left(\sum_{\textbf n\in\mathbb{Z}^d}a_{\textbf n}\ln^{\sigma}\lfloor \textbf n\rfloor+\ln^{\sigma}\lfloor \textbf j\rfloor-\ln^{\sigma}\lfloor \textbf n_1^{*}(a,0,0;\textbf j)\rfloor\right)}\cdot e^{-r\sum_{\textbf n\in\mathbb{Z}^d}2 a_{\textbf n}\ln^{\sigma}\lfloor \textbf n\rfloor}\\
\nonumber&\leq& \epsilon_{s+1}^{0.6}\sum_{a\in\mathbb{N}^{\mathbb{Z}^d}}e^{-r\sum_{\textbf n\in\mathbb{Z}^d}a_{\textbf n}\ln^{\sigma}\lfloor \textbf n\rfloor}\qquad\mbox{(in view of $\lfloor \textbf j\rfloor\leq \lfloor \textbf n_1^*(a,0,0;\textbf j)\rfloor$ and $2\rho_{s+1}<r$)}\\
\nonumber&\leq& \epsilon_{s+1}^{0.6}\prod_{\textbf n\in\mathbb{Z}^d}\left(1-e^{-r \ln^{\sigma}\lfloor \textbf n\rfloor}\right)^{-1} \qquad\qquad\qquad \mbox{(by Lemma \ref{a3})}\\
\label{M15}&\leq&\epsilon_{s+1}^{0.6}\exp\left\{\left(\frac{100d}{r^2}\right)^d\cdot \exp\left\{d\cdot\left(\frac{2d}{r}\right)^{\frac1{\sigma-1}}\right\}\right\}\qquad \mbox{(by Lemma \ref{a5})}\nonumber\\
&\leq &\epsilon_{s+1}^{0.55}.
\end{eqnarray}

Recall that
\begin{equation*}
\left[R_{1}\right]=\sum_{\textbf j\in\mathbb{Z}^d}J^{(\textbf j)}\sum_{a\in\mathbb{N}^{\mathbb{Z}^d}}B^{(\textbf j)}_{a00}\mathcal{M}_{a00},
\end{equation*}
and then
\begin{equation*}
\frac{\partial^2\left[R_{1}\right]}{\partial q_{\textbf n}\partial\bar q_{\textbf n}}=\sum_{a\in\mathbb{N}^{\mathbb{Z}^d}}B^{(\textbf n)}_{a00}\mathcal{M}_{a00}.
\end{equation*}
Note that
\begin{equation*}
\frac{\partial^2\left[R_{1}\right]}{\partial q_{\textbf n}\partial\bar q_{\textbf n}}=\left.\frac{\partial^2R_{s+1}}{\partial q_{\textbf n}\partial\bar q_{\textbf n}}\right|_{q=\bar q=0}
\end{equation*}
and then one has
\begin{equation*}
\lim_{\left\|\textbf n\right\|\rightarrow\infty}\sum_{a\in\mathbb{N}^{\mathbb{Z}^d}}B^{(\textbf n)}_{a00}\mathcal{M}_{a00}\ \mbox{exists}
\end{equation*}
and
\begin{equation}\label{061103}
\left\|\sum_{a\in\mathbb{N}^{\mathbb{Z}^d}}B^{(\textbf n)}_{a00}\mathcal{M}_{a00}-\lim_{\left\|\textbf n\right\|\rightarrow\infty}\sum_{a\in\mathbb{N}^{\mathbb{Z}^d}}B^{(\textbf n)}_{a00}\mathcal{M}_{a00}\right\|_{\rho_{s+1}}\leq \frac{\epsilon_{s+1}^{0.5}}{\left\|\textbf n\right\|}.
\end{equation}
Hence let
\begin{equation*}
\breve V_{s+1}:=\breve V_{s}+\lim_{\left\|\textbf n\right\|\rightarrow\infty}\sum_{a\in\mathbb{N}^{\mathbb{Z}^d}}B^{(\textbf n)}_{a00}\mathcal{M}_{a00},
\end{equation*}
which does not depend on $\textbf{n}$
and satisfies (\ref{061102}) for $s+1$.

Let
\begin{equation*}
\widehat V_{\textbf j,s+1}=\widehat V_{\textbf j,s}+\left(\sum_{a\in\mathbb{N}^{\mathbb{Z}^d}}B^{(\textbf j)}_{a00}\mathcal{M}_{a00}-\lim_{\left\|\textbf j\right\|\rightarrow\infty}\sum_{a\in\mathbb{N}^{\mathbb{Z}^d}}B^{(\textbf j)}_{a00}\mathcal{M}_{a00}\right),
\end{equation*}
and (\ref{053190}) with $s+1$ satisfies by (\ref{061103}).

Next,
if $V\in \mathcal{C}_{\frac{\eta_s}{2}}\left(V_s^*\right)$, by using Cauchy's estimate implies
\begin{eqnarray}
\nonumber\sum_{\textbf n\in\mathbb{Z}^d}\left|\frac{\partial \widehat{V}_{\textbf j,s}}{\partial V_{\textbf n}}(V)\right|
\nonumber&\leq& \frac{2}{\eta_s}\left\|\widehat{V}_s\right\|_\infty\\
\label{M11}&<&10 \eta_s^{-1}\ \ \mbox{(by (\ref{053190}))},
\end{eqnarray}
and let $X\in \mathcal{C}_{\frac{1}{10}\lambda_s\eta_s}\left(V_s^*\right)$, then
\begin{eqnarray*}
\left\|\widehat{V}_s(X)-\omega\right\|_{\infty}
&=&\left\|\widehat{V}_s(X)-\widehat{V}_s\left(V_s^*\right)\right\|_{\infty}\\
& \leq&\sup_{\mathcal{C}_{\frac{1}{10}\lambda_s\eta_s}\left(V_s^*\right)}\left\|\frac{\partial \widetilde{V}_s}{\partial V}\right\|_{l^{\infty}\rightarrow l^{\infty}}\cdot\left\|X-V_s^*\right\|_{\infty}\\
&<&10 \eta_s^{-1}\cdot\frac{1}{10}\lambda_s\eta_s\qquad \qquad  \ \mbox{(in view of (\ref{M11}))}\\
&=&\lambda_s,
\end{eqnarray*}
that is
\begin{equation*}
\widehat{V}_s\left(\mathcal{C}_{\frac{1}{10}\lambda_s\eta_s}\left(V_s^*\right)\right)\subseteq \mathcal{C}_{\lambda_s}\left(\omega\right).
\end{equation*}
By (\ref{M15}), we have
\begin{eqnarray*}
\left|\widehat{V}_{\textbf j,s+1}-\widehat{V}_{\textbf j,s}\right|
<\epsilon_{s+1}^{0.6}\cdot\exp\left\{{18}\cdot \exp\left\{4^{\frac1{\sigma-1}}\right\}\right\}
\label{M16}<\epsilon_{s+1}^{0.59},
\end{eqnarray*}
which verifies (\ref{206}). Further applying Cauchy's estimate on $\mathcal{C}_{\lambda_s\eta_s}\left(V_s^*\right)$, one gets
\begin{eqnarray}\label{633}
\sum_{\textbf n\in\mathbb{Z}^d}\left|\frac{\partial \widehat{V}_{\textbf j,s+1}}{\partial V_{\textbf n}}-\frac{\partial \widehat{V}_{\textbf j,s}}{\partial V_{\textbf n}}\right|
\leq\frac{\left\|\widehat{V}_{s+1}-\widehat{V}_{s}\right\|_\infty}{\lambda_s\eta_s}
\label{M17}\leq\frac{\epsilon_{s+1}^{0.59}}{\lambda_s\eta_s}.
\end{eqnarray}
Since
\begin{equation*}
\eta_{s+1}=\frac{1}{20}\lambda_s\eta_s,
\end{equation*}
hence one has
\begin{eqnarray}
\nonumber\lambda_s\eta_{s}&=&20\prod_{i=0}^{s}\left(\frac1{20}\lambda_i\right)\\
\nonumber&=&20\prod_{i=0}^{s}\left(\frac1{20}\epsilon_i^{0.01}\right)\\
\nonumber&\geq&20\prod_{i=0}^{s}\epsilon_i^{0.02}\\
\label{M19}&\geq&20\epsilon_{s+1}^{0.04}.
\end{eqnarray}
On $ \mathcal{C}_{\frac{1}{10}\lambda_s\eta_s}\left(V_s^*\right)$ and for any $\textbf j\in\mathbb{Z}^d$, we deduce from (\ref{M17}), (\ref{M19}) and the assumption (\ref{199}) that
\begin{eqnarray*}
\sum_{\textbf n\in\mathbb{Z}^d}\left|\frac{\partial \widehat{V}_{\textbf j,s+1}}{\partial V_{\textbf n}}-\delta_{\textbf j\textbf n}\right|
&\leq&\sum_{\textbf n\in\mathbb{Z}^d}\left|\frac{\partial \widehat{V}_{\textbf j,s+1}}{\partial V_{\textbf n}}-\frac{\partial \widehat{V}_{\textbf j,s}}{\partial V_{\textbf n}}\right|+\sum_{\textbf n\in\mathbb{Z}^d}\left|\frac{\partial \widehat{V}_{\textbf j,s}}{\partial V_{\textbf n}}-\delta_{\textbf j\textbf n}\right|\\
&\leq&\epsilon_{s+1}^{0.5}+d_s\epsilon_{0}^{\frac{1}{10}}\\
&<&d_{s+1}\epsilon_{0}^{\frac{1}{10}},
\end{eqnarray*}
and consequently
\begin{equation}\label{M20}
\left|\left|\frac{\partial \widehat{V}_{s+1}}{{\partial V}}-Id\right|\right|_{l^{\infty}\rightarrow l^{\infty}}<d_{s+1}\epsilon_{0}^{\frac{1}{10}},
\end{equation}
which verifies (\ref{199}) for $s+1$.

Finally, we will freeze $\omega$ by invoking an inverse function theorem. Consider the following functional equation
\begin{equation}\label{M21}
\widehat{V}_{s+1}\left(X\right)=\omega, \qquad  X\in \mathcal{C}_{\frac{1}{10}\lambda_s\eta_s}\left(V_s^*\right),
\end{equation}
from (\ref{M20}) and the standard inverse function theorem implies (\ref{M21}) having a solution $V_{s+1}^*$, which verifies (\ref{198}) for $s+1$. Rewriting (\ref{M21}) as
\begin{equation}\label{M22}
V_{s+1}^*-V_s^*=\left(I-\widehat{V}_{s+1}\right)\left(V_{s+1}^*\right)-\left(I-\widehat{V}_{s+1}\right)\left({V_s}^*\right)+\left(\widehat{V}_s-\widehat{V}_{s+1}\right)\left(V_s^*\right),
\end{equation}
and by using (\ref{M16}) and (\ref{M20}) one has
\begin{equation*}\label{M23}
\left\|V_{s+1}^*-V_s^*\right\|_{\infty}\leq \left(1+d_{s+1}\right)\epsilon_{0}^{\frac{1}{10}}\left\|V_{s+1}^*-V_s^*\right\|_{\infty}+\epsilon_{s+1}^{0.59}<\epsilon_{s+1}^{0.58}\leq  \lambda_s\eta_s,
\end{equation*}where the last inequality is based on (\ref{M19}),
which verifies (\ref{205}) and completes the proof of the iterative lemma.
\end{proof}

\subsection{Convergence}

We are now in a position to prove the convergence. To apply iterative lemma with $s=0$, set
\begin{equation}\label{031940}
V_0^*=\omega,\hspace{12pt}\widehat{V}_0=id,\hspace{12pt}\epsilon_0=C\epsilon,
\end{equation}
and consequently (\ref{198})-(\ref{202}) with $s=0$ are satisfied. Hence applying the iterative lemma, we obtain a decreasing
sequence of domains $D_{s}\times\mathcal{C}_{\eta_{s}}\left(V_{s}^*\right)$ and a sequence of
transformations
\begin{equation*}
\Phi^s=\Phi_1\circ\cdots\circ\Phi_s:\hspace{6pt}D_{s}\times\mathcal{C}_{\eta_{s}}\left(V_{s}^*\right)\rightarrow D_{0}\times\mathcal{C}_{\eta_{0}}(V_{0}),
\end{equation*}
such that $H\circ\Phi^s=N_s+R_s$ for $s\geq1$. Moreover, the
estimates (\ref{203})-(\ref{205}) hold. Thus we can show $V_s^*$ converge to a limit $V^*$ with the estimate
\begin{equation*}
||V^*-\omega||_{\infty}\leq\sum_{s=0}^{\infty}2\epsilon_{s}^{0.5}<\epsilon_{0}^{0.4},
\end{equation*}
and $\Phi^s$ converge uniformly on $D_*\times\{V_*\}$, where $$D_*=\left\{(q_{\textbf n})_{\textbf n\in\mathbb{Z}^d}:\frac{2}{3}\leq|q_{\textbf n}|e^{r\ln^{\sigma}\lfloor \textbf n\rfloor}\leq\frac{5}{6}\right\},$$ to $\Phi:D_*\times\{V^*\}\rightarrow D_0$ with the estimates
\begin{eqnarray}
\nonumber&&||\Phi-id||_{\sigma,r}\leq \epsilon_{0}^{0.4},\\
\nonumber&&||D\Phi-Id||_{(\sigma,r)\rightarrow(\sigma,r)}\leq \epsilon_{0}^{0.4}.
\end{eqnarray}
Hence
\begin{equation}\label{060101}
H_*=H\circ\Phi=N_*+R_{2,*},
\end{equation}
where
\begin{equation}
N_*=\sum_{\textbf n\in\mathbb{Z}^d}\left(\left\|\textbf n\right\|^2+\xi+\omega_{\textbf n}\right)|q_{\textbf n}|^2,
\end{equation}$$\xi=\sum_{s\geq 0}\breve V_s,$$
and
\begin{equation}\label{062811}
||R_{2,*}||_{0.2}^{+}\leq\frac{7}{6}\epsilon_0.
\end{equation}
By (\ref{050907}), the Hamiltonian vector field $X_{R_{2,*}}$ is a bounded map from $\mathfrak{H}_{\sigma,r}$ into $\mathfrak{H}_{\sigma,r}$, and
we get an invariant torus $\mathcal{T}$ with frequency $\left(\left\|\textbf n\right\|^2+\xi+\omega_{\textbf n}\right)_{\textbf n\in\mathbb{Z}^d}$ for ${X}_{H_*}$.

\subsection{Proof of Theorem \ref{031930}}
\begin{proof}
Expand $u$ into Fourier series by
\begin{equation*}
u=\sum_{\textbf n\in\mathbb{Z}^d}q_{\textbf n}\phi_{\textbf n }(x),
\end{equation*}
where \begin{equation*}
\phi_{\textbf n}(x)=\frac{1}{(2\pi)^{d/2}}e^{\sqrt{-1}\sum_{i=1}^dn_ix_i}.
\end{equation*}
Then the  Hamiltonian of equation (\ref{061510}) is given by
\begin{equation}\label{H.}
	H(q,\bar q)=N(q,\bar q)+ R(q,\bar q),
\end{equation}
where
\begin{equation*}
N(q,\bar q)=\sum_{\textbf n\in\mathbb{Z}^d}\left(\left\|\textbf n\right\|^2+V_{\textbf n}\right)\left|q_{\textbf n}\right|^2
\end{equation*}
and
\begin{equation*}
R(q,\bar q)=\epsilon\sum_{a,k,k'\in\mathbb{N}^{\mathbb{Z}^d}\atop |a|=0,|k|+|k'|=4}B_{akk'}\mathcal{M}_{akk'}.
\end{equation*}
Here \begin{equation*}
B_{akk'}=\frac{1}{(2\pi)^d},
\end{equation*}if mass conservation (\ref{051702}) and momentum conservation (\ref{050901}) satisfy; otherwise
\begin{equation*}
B_{akk'}=0.
\end{equation*}
Then one has
\begin{equation*}
\left\|R\right\|_{\rho_0}\leq  \frac{\epsilon}{(2\pi)^d}:=\epsilon_0.
\end{equation*}

Then the assumptions (\ref{031940}) in iterative lemma for $s=0$ hold.  Applying iterative lemma,  $\Phi(\mathcal{T})$ is the desired invariant torus for the Hamiltonian (\ref{H.}). Moreover, we deduce the torus $\Phi(\mathcal{T})$ is linearly stable from the fact that (\ref{060101}) is a normal form of order 2 around the invariant torus.
\end{proof}
\section{Appendix}
\subsection{Technical Lemma}
\begin{lem}\label{122203}
Given any $\sigma>2$, there exists a constant $c(\sigma)>e^3$ depending on $\sigma$ only such that
\begin{equation}\label{122201}
\ln^{\sigma}(x+y)-\ln^{\sigma}x-\frac12\ln^{\sigma}y\leq 0, \quad \mbox{for}\ c(\sigma)\leq y\leq x.
\end{equation}
\end{lem}
\begin{proof}
Write $x=ty$ with $t\geq 1$, and then
\begin{eqnarray*}
&&\ln^{\sigma}(x+y)-\ln^{\sigma}x-\frac12\ln^{\sigma}y\\
&=&\ln^{\sigma}(ty+y)-\ln^{\sigma}(ty)-\frac12\ln^{\sigma}y\\
&=&\ln^{\sigma}y\left(\left(1+\frac{\ln(1+t)}{\ln y}\right)^{\sigma}-\left(1+\frac{\ln t}{\ln y}\right)^{\sigma}-\frac12\right).
\end{eqnarray*}
To prove (\ref{122201}), it suffices to show that
\begin{equation}\label{122202}
\sup_{t\geq 1,y\geq c(\sigma)}\left(\left(1+\frac{\ln(1+t)}{\ln y}\right)^{\sigma}-\left(1+\frac{\ln t}{\ln y}\right)^{\sigma}\right)\leq \frac12.
\end{equation}
Using differential mean value theorem, one has
\begin{eqnarray*}
\left(1+\frac{\ln(1+t)}{\ln y}\right)^{\sigma}-\left(1+\frac{\ln t}{\ln y}\right)^{\sigma}
\leq \frac{\sigma}{\ln y}\left(1+\frac{\ln(1+t)}{\ln y}\right)^{\sigma-1}\ln \left(\frac{1+t}{t}\right).
\end{eqnarray*}

\textbf{Case 1.} $$\frac{\ln(1+t)}{\ln y}\leq 1.$$

Then one has
\begin{equation*}
\frac{\sigma}{\ln y}\left(1+\frac{\ln(1+t)}{\ln y}\right)^{\sigma-1}\ln \left(\frac{1+t}{t}\right)\leq \frac{\sigma\cdot 2^{\sigma-1}}{\ln y}\cdot \ln 2,
\end{equation*}
where using
\begin{equation*}
\ln \left(\frac{1+t}{t}\right)\leq \ln 2.
\end{equation*}
Taking
\begin{equation}\label{51501}
c_1(\sigma)=\exp\left\{\sigma\cdot 2^{\sigma}\cdot\ln 2 \right\},
\end{equation}
we finish the proof of (\ref{122202}) in view of $y\geq c_1(\sigma)$.

\textbf{Case 2.} $$\frac{\ln(1+t)}{\ln y}>1.$$

Then one has
\begin{eqnarray*}
&&\frac{\sigma}{\ln y}\left(1+\frac{\ln(1+t)}{\ln y}\right)^{\sigma-1}\ln \left(\frac{1+t}{t}\right)\\
&\leq& \frac{\sigma\cdot 2^{\sigma-1}}{\ln^{\sigma}y}\cdot \ln^{\sigma-1}(1+t)\cdot \ln\left(\frac{1+t}{t}\right)\\
&\leq&\frac{\sigma\cdot 2^{\sigma-1}}{\ln^{\sigma}y}\cdot 2\left(\frac{\sigma-1}e\right)^{\sigma-1},
\end{eqnarray*}
where noting that
\begin{equation*}
\max_{t\geq 1}\left(\ln^{\sigma-1}(1+t)\cdot \ln\left(\frac{1+t}{t}\right)\right)\leq 2\left(\frac{\sigma-1}e\right)^{\sigma-1}:=c^*(\sigma).
\end{equation*}
Taking
\begin{equation}\label{51502}
c_2(\sigma)=\exp\left\{{\left(\sigma\cdot 2^{\sigma}\cdot c^*(\sigma)\right)^{\frac1{\sigma}}}\right\},
\end{equation}
we finish the proof of (\ref{122202}) in view of $y\geq c_2(\sigma)$.

Finally letting
\begin{equation*}\label{030402}
c(\sigma)=\max \{c_1(\sigma),c_2(\sigma)\},
\end{equation*}
we finish the proof of (\ref{122201}).
\end{proof}
\begin{rem}
In view of (\ref{51501}) and (\ref{51502}), one has
\begin{equation*}
c_1(\sigma)\rightarrow 2^{8}
\end{equation*}
and
\begin{equation*}
c_2(\sigma)\rightarrow\exp\left\{\frac4{\sqrt e}\right\}
\end{equation*}
as $\sigma\rightarrow 2$.
Hence we can take
\begin{equation*}
c(\sigma):=2^{10}.
\end{equation*}
\end{rem}
\begin{rem}\label{51505}
In view of (\ref{122201}), one has
\begin{equation}\label{51504}
\ln^{\sigma}x+\ln^{\sigma} y-\ln^{\sigma}(x+y)\geq \frac12\ln^{\sigma}y\geq 0.
\end{equation}
\end{rem}
\begin{lem}\label{8.6}For $\sigma>2$ and $\delta\in(0,1)$, let $f_{\sigma,\delta}(x)=e^{-\delta x^{\sigma}+x}$, then we have
\begin{equation}\label{042805}
\max_{x\geq0}f_{\sigma,\delta}(x)\leq \exp\left\{\left(\frac1{\delta}\right)^{\frac1{\sigma-1}}\right\}.
\end{equation}
\end{lem}
\begin{proof}
Since
\begin{equation*}
f_{\sigma,\delta}'(x)=e^{-\delta x^{\sigma}+x}\left(-\delta\sigma x^{\sigma-1}+1\right),
\end{equation*}
we have
\begin{equation*}
f_{\sigma,\delta}'(x)=0\Leftrightarrow x=\left(\frac{1}{\delta\sigma}\right)^{\frac1{\sigma-1}}.
\end{equation*}
Then one has
\begin{equation*}
\max_{x\geq0}f_{\sigma,\delta}(x)=f_{\sigma,\delta}\left(\left(\frac{1}{\delta\sigma}\right)^{\frac1{\sigma-1}}\right)\leq \exp\left\{\left(\frac1{\delta}\right)^{\frac1{\sigma-1}}\right\},
\end{equation*}
where the last inequality uses $\sigma>2$.
\end{proof}
\begin{lem}\label{8.6}For $p\geq 1$ and $\delta\in (0,1)$, let $$g_{p,\delta}(x)=x^{p}e^{-\delta x},$$ then the following inequality holds
\begin{equation}\label{042805*}
\max_{x\geq0}g_{p,\delta}(x)\leq \left(\frac{p}{e\delta}\right)^p.
\end{equation}
\end{lem}
\begin{proof}
Since
\begin{equation*}
g_{p,\delta}'(x)=px^{p-1}e^{-\delta x}-\delta x^{p}e^{-\delta x},
\end{equation*}
then we have
\begin{equation*}
g_{p,\delta}'(x)=0\Leftrightarrow x=\frac p{\delta}.
\end{equation*}
Then one has
\begin{equation*}\label{042608*}
\max_{x\geq0}g_{p,\delta}(x)=g_{p,\delta}\left(\frac{p}\delta\right)= \left(\frac{p}{e\delta}\right)^p.
\end{equation*}
\end{proof}
\begin{lem}\label{lem2}
For $\sigma>2$ and $\delta\in(0,1)$, we have
\begin{equation}\label{0418011}
\sum_{j\geq 1}e^{-\delta \ln^{\sigma} j}\leq \frac{6}\delta\cdot \exp\left\{\left(\frac1{\delta}\right)^{\frac1{\sigma-1}}\right\}.
\end{equation}
\end{lem}
\begin{proof}
Obviously, one has
\begin{eqnarray*}
\label{e}&&\sum_{j\geq 1}e^{-\delta \ln^{\sigma} j}\\
 &\leq& 2+\int_{1}^{+\infty}e^{-\delta \ln^{\sigma} x}\mathrm{d}x\\
&=&2+\int_{0}^{+\infty}e^{-\delta y^{\sigma}+y}\mathrm{d}y\\
&\leq&2\exp\left\{\left(\frac2{\delta}\right)^{\frac1{\sigma-1}}\right\}\cdot\int_{0}^{+\infty}e^{-\frac12\delta y^{\sigma}}\mathrm{d}y\qquad \mbox{(by (\ref{042805}) and $\sigma>2$)}\\
&\leq&2\exp\left\{\left(\frac2{\delta}\right)^{\frac1{\sigma-1}}\right\}\cdot\left(1+\int_{1}^{+\infty}e^{-\frac12\delta y}\mathrm{d}y\right)\\
&\leq &\frac{6}\delta\cdot \exp\left\{\left(\frac2{\delta}\right)^{\frac1{\sigma-1}}\right\}.
\end{eqnarray*}

\end{proof}
\begin{lem}\label{a3}
For $\sigma>2$ and $\delta\in(0,1)$,  we have the following inequality
\begin{equation}\label{041809}
\sum_{a\in\mathbb{N}^{\mathbb{Z}^d}}e^{-\delta\sum_{\textbf n\in\mathbb{Z}^d}a_{\textbf n}\ln^{\sigma}\lfloor \textbf n\rfloor}\leq\prod_{\textbf n\in\mathbb{Z}^d}\frac{1}{1-e^{-\delta \ln^{\sigma}\lfloor \textbf n\rfloor}}.
\end{equation}
\end{lem}
\begin{proof}
By a direct calculation, one has
$$\sum\limits_{a\in\mathbb{N}^{\mathbb{Z}^d}}e^{-\delta\sum_{\textbf n\in\mathbb{Z}^d}a_{\textbf n}\ln^{\sigma}\lfloor \textbf n\rfloor}\leq \prod\limits_{\textbf n\in\mathbb{Z}^d}\left(\sum\limits_{a_{\textbf n}\in\mathbb{N}}e^{-\delta a_{\textbf n}\ln^{\sigma}\lfloor \textbf n\rfloor}\right)
=\prod\limits_{\textbf n\in\mathbb{Z}^d}\frac{1}{1-e^{-\delta \ln^{\sigma}\lfloor \textbf n\rfloor}}.$$
\end{proof}

\begin{lem}\label{a5}
For $\sigma>2$ and $0<\delta\ll 1$, then we have
\begin{equation}\label{122401}
\prod_{\textbf n\in\mathbb{Z}^d}\frac{1}{1-e^{-\delta \ln^{\sigma}\lfloor \textbf n\rfloor}}\leq\exp\left\{\left(\frac{100d}{\delta^2}\right)^d\cdot \exp\left\{d{\left(\frac{2d}\delta \right)}^{\frac{1}{\sigma-1}}\right\}\right\}.
\end{equation}
\end{lem}
\begin{proof}
If $x\in(0,3-2\sqrt{2})$, one has
\begin{equation}\label{042506}\frac{1}{1-e^{-x}}\leq \frac{1}{x^{2}}
\end{equation}
and
\begin{equation}\label{022208}
\ln\left(\frac{1}{1-x}\right)\leq\sqrt{x}.
\end{equation}
Let \begin{equation*}\label{ntheta}
N=\exp\left\{\left(\ln\left(3+2\sqrt{2}\right)\delta^{-1}\right)^{\frac1\sigma}\right\}.
\end{equation*}If $\lfloor \textbf n\rfloor >N$, then one has
\begin{equation*}
-\delta\ln^{\sigma}\lfloor \textbf n\rfloor\leq \ln\left(3-2\sqrt{2}\right),
\end{equation*}
which implies
\begin{equation}\label{022207}
e^{-\delta\ln^{\sigma}\lfloor \textbf n\rfloor}\leq 3-2\sqrt{2}.
\end{equation}
Hence we have
\begin{eqnarray*}
\nonumber&&\prod\limits_{\textbf n\in\mathbb{Z}^d}\frac{1}{1-e^{-\delta \ln^{\sigma}\lfloor \textbf n\rfloor}}
\nonumber\\
&=&\nonumber\left(\prod\limits_{\lfloor \textbf n\rfloor\leq N}\frac{1}{1-e^{-\delta \ln^{\sigma}\lfloor \textbf n\rfloor }}\right)\left(\prod\limits_{\lfloor \textbf n\rfloor> N}\frac{1}{1-e^{-\delta \ln^{\sigma} \lfloor \textbf n\rfloor }}\right)\\
&\leq& \left(\prod\limits_{\lfloor \textbf n\rfloor\leq N}\frac{1}{1-e^{-\delta }}\right)\left(\prod\limits_{\lfloor \textbf n\rfloor> N}\frac{1}{1-e^{-\delta \ln^{\sigma} \lfloor \textbf n\rfloor }}\right),\end{eqnarray*}
On one hand, using (\ref{042506}) we have
\begin{eqnarray}\label{022210}
\prod\limits_{\lfloor \textbf n\rfloor\leq N}\frac{1}{1-e^{-\delta }}\leq \left(\frac{1}{\delta^2}\right)^{(2N+1)^d}\leq \left(\frac{1}{\delta}\right)^{6^dN^d}.
\end{eqnarray}
On the other hand, by (\ref{022208}) and (\ref{022207}) one has
\begin{eqnarray}
\nonumber&&\prod\limits_{\lfloor \textbf n\rfloor> N}\frac{1}{1-e^{-\delta \ln^{\sigma} \lfloor \textbf n\rfloor }}\\
&=&\nonumber \exp\left\{\sum_{\lfloor \textbf n\rfloor >N}\ln\left(\frac{1}{1-e^{-\delta \ln^{\sigma} \lfloor \textbf n\rfloor }}\right)\right\}\\
&\leq&\exp\left\{\sum_{\lfloor \textbf n\rfloor >N}e^{-\frac\delta2 \ln^{\sigma} \lfloor \textbf n\rfloor }\right\}\nonumber\\
&\leq&\exp\left\{\left(1+2\sum_{j\geq 1}e^{-\frac\delta{2d} \ln^{\sigma}j}\right)^d\right\}\nonumber\\
&\leq&\exp\left\{\left(\frac{20d}\delta\right)^d\cdot \exp\left\{d\left(\frac{2d}{\delta}\right)^{\frac1{\sigma-1}}\right\}\right\}\label{022211},
\end{eqnarray}
where the last inequality is based on (\ref{0418011}) in Lemma \ref{lem2}.

Combing (\ref{022210}) and (\ref{022211}), we finish the proof of (\ref{122401}) using $0<\delta\ll 1$.

\end{proof}
\begin{lem}
\label{b1}For $\sigma>2$, $\delta\in(0,1)$, $p=1,2$ and $a=(a_{\textbf n})_{\textbf n\in\mathbb{Z}^d}\in\mathbb{N}^{\mathbb{Z}^d}$, then we have
\begin{equation}\label{042807}
\prod_{\textbf n\in\mathbb{Z}^d}\left(1+a_{\textbf n}^p\right)e^{-2\delta a_{\textbf n}\ln^{\sigma} \lfloor \textbf n\rfloor}\leq \exp\left\{3dp\left(\frac{p}{\delta}\right)^{\frac 1{\sigma-1}}\cdot\exp\left\{\left(\frac1\delta\right)^{\frac1\sigma}\right\}
\right\}.
\end{equation}
\end{lem}
\begin{proof}
Let
\begin{equation*}
f_{p,\sigma,\delta}(x)=x^{p}e^{-\delta\ln^{\sigma}x},
\end{equation*}
and one has
\begin{equation}\label{030102}
\max_{x\geq 1}f_{p,\sigma,\delta}(x)\leq \exp\left\{p\left(\frac{p}{\delta}\right)^{\frac1{\sigma-1}}\right\}.
\end{equation}
In fact
\begin{equation*}
f'_{p,\sigma,\delta}(x)=px^{p-1}e^{-\delta\ln^{\sigma}x}
+x^{p-1}e^{-\delta\ln^{\sigma}x}\left(-\delta\sigma\ln^{\sigma-1}x\right)
\end{equation*}
and then
\begin{equation*}
f'_{p,\sigma,\delta}(x)=0\Leftrightarrow x=\exp\left\{\left(\frac{p}{\delta\sigma}\right)^{\frac1{\sigma-1}}\right\}.
\end{equation*}
Hence one has
\begin{equation*}
\max_{x\geq 1}f_{p,\sigma,\delta}(x)\leq f\left(\exp\left\{\left(\frac{p}{\delta\sigma}\right)^{\frac1{\sigma-1}}\right\}\right)
\leq\exp\left\{p\left(\frac{p}{\delta}\right)^{\frac1{\sigma-1}}\right\}.
\end{equation*}
Note that
$$\prod_{\textbf n\in\mathbb{Z}^d}\left(1+a_{\textbf n}^2\right)e^{-2\delta a_{\textbf n}\ln^{\sigma}\lfloor \textbf n\rfloor}=\prod_{\textbf n\in\mathbb{Z}^d\atop a_{\textbf n}\geq1}\left(1+a_{\textbf n}^2\right)e^{-2\delta a_{\textbf n}\ln^{\sigma}\lfloor \textbf n\rfloor}.$$
Then we can assume $a_{\textbf n}\geq 1\ \mbox{for}\ \forall\  \textbf n\in\mathbb{Z}^d$ in what follows. Thus one has
\begin{eqnarray}
\nonumber&&\left(1+a_{\textbf n}^p\right)e^{-2\delta a_{\textbf n}\ln^{\sigma}\lfloor \textbf n\rfloor}\\
\nonumber&\leq&\left( 2a_{\textbf n}\right)^pe^{-2\delta a_{\textbf n}\ln^{\sigma}\lfloor \textbf n\rfloor}\\
\nonumber&=&\frac{1}{\ln^{\sigma p}\lfloor \textbf n\rfloor}\cdot\left(2a_{\textbf n}\ln^{\sigma}\lfloor \textbf n\rfloor\right)^pe^{-2\delta a_{\textbf n}\ln^{\sigma}\lfloor \textbf n\rfloor}\\
\label{160516}&\leq&\exp\left\{p\left(\frac{p}{\delta}\right)^{\frac1{\sigma-1}}\right\},
\end{eqnarray}
where the last inequality is based on (\ref{030102}) and using the fact that
\begin{equation*}
\frac{1}{\ln^{\sigma p}\lfloor \textbf n\rfloor}\leq 1.
\end{equation*}
On the other hand, for $p=1,2$ one has
\begin{equation}\label{042806}
(1+a_{\textbf n}^p)e^{-2\delta a_{\textbf n}\ln^{\sigma} \lfloor \textbf n\rfloor}\leq 1
\end{equation}
when $\ln^{\sigma} \lfloor \textbf n\rfloor\geq \delta^{-1}$.

Therefore, we have
\begin{eqnarray*}
&&\prod_{\textbf n\in\mathbb{Z}^d}\left(1+a_{\textbf n}^p\right)e^{-2\delta a_{\textbf n}\ln^{\sigma}\lfloor \textbf n\rfloor}\\
&=&\left(\prod_{\ln^{\sigma} \lfloor \textbf n\rfloor< \delta^{-1}}\left(1+a_{\textbf n}^p\right)e^{-2\delta a_{\textbf n}\ln^{\sigma}\lfloor \textbf n\rfloor}\right)\left(\prod_{\ln^{\sigma} \lfloor \textbf n\rfloor\geq \delta^{-1}}\left(1+a_{\textbf n}^p\right)e^{-2\delta a_{\textbf n}\ln^{\sigma}\lfloor \textbf n\rfloor}\right)\\
&\leq&\prod_{\ln^{\sigma} \lfloor \textbf n\rfloor<\delta^{-1}}\left(1+a_{\textbf n}^p\right)e^{-2\delta a_{\textbf n}\ln^{\sigma}\lfloor \textbf n\rfloor}\qquad \mbox{(in view of (\ref{042806}))}\\
&\leq&\prod_{\ln^{\sigma} \lfloor \textbf n\rfloor<\delta^{-1}}\exp\left\{p\left(\frac{p}{\delta}\right)^{\frac1{\sigma-1}}\right\}\qquad \ \mbox{(in view of (\ref{160516}))}\\
&\leq&\left(\exp\left\{p\left(\frac{p}{\delta}\right)^{\frac1{\sigma-1}}\right\}\right)^{3d\exp\left\{\left(\frac1\delta\right)^{\frac1\sigma}\right\}}\ \ \ \\
&=&\exp\left\{3dp\left(\frac{p}{\delta}\right)^{\frac 1{\sigma-1}}\cdot\exp\left\{\left(\frac1\delta\right)^{\frac1\sigma}\right\}
\right\},
\end{eqnarray*}
which finishes the proof of (\ref{042807}).
\end{proof}

\subsection{Measure Estimate}
\begin{lem}\label{050601}
Let the set $\Pi $, which is defined by (\ref{050701}),
with probability measure. Then there exists a subset $\Pi_{\gamma}\subset\Pi$ satisfying
\begin{equation}
\mbox{meas}\ \Pi_{\gamma}\leq C\gamma,
\end{equation}
where $C$ is a positive constant, such that for any $\omega\in\Pi\setminus\Pi_{\gamma}$, the inequalities (\ref{040601}) and (\ref{040602}) holds.
\end{lem}
\begin{proof}
Define the resonant set $\mathcal{R}_{1}$ by
\begin{equation}\label{020701}
\mathcal{R}_1=\bigcup_{j\in\mathbb{Z}\atop 0\neq\ell\in\mathbb{Z}^{\mathbb{Z}^d}}\mathcal{R}_{\ell,j},
\end{equation}
where
\begin{eqnarray*}
 \mathcal{R}_{\ell,j}=\left\{\omega:\left| \sum_{\textbf n\in\mathbb{Z}^d}{\ell}_{\textbf n}\omega_{\textbf n}+j\right|< \gamma \prod_{\textbf n\in \mathbb{Z}^d}\frac{1}{1+|\ell_{\textbf n}|^3\langle \textbf n\rangle^{d+4}}\right\},
\end{eqnarray*}
 for any $0\neq \ell\in\mathbb{Z}^{\mathbb{Z}^d}$ with $|\ell|<\infty$. Let
\begin{equation*}
m(\ell)=\min\left\{\langle\textbf n\rangle:{\ell}_{\textbf{n}}\neq 0\right\},
\end{equation*}
and one has
\begin{equation}\label{033110}
\mbox{meas}\ \mathcal{R}_{\ell,j}\leq \gamma\cdot m(\ell) \prod_{\textbf n\in\mathbb{Z}^d}\frac{1}{1+|\ell_{\textbf n}|^3\langle\textbf n\rangle^{d+4}}.
\end{equation}
In view of (\ref{050701}), one has
\begin{equation}
\left| \sum_{\textbf n\in\mathbb{Z}^d}{\ell}_{\textbf n}\omega_{\textbf n}\right|\leq |\ell|\leq \prod_{\textbf n\in \mathbb{Z}^d\atop \ell_{\textbf n}\neq 0}{|\ell_{\textbf n}|\langle \textbf n\rangle}.
\end{equation}
Hence if
\begin{equation}
|j|>2\prod_{\textbf n\in \mathbb{Z}^d\atop \ell_{\textbf n}\neq 0}{|\ell_{\textbf n}|\langle \textbf n\rangle},
\end{equation}
then one has
\begin{equation}
\left| \sum_{\textbf n\in\mathbb{Z}^d}{\ell}_{\textbf n}\omega_{\textbf n}+j\right|\geq |j|-\left| \sum_{\textbf n\in\mathbb{Z}^d}{\ell}_{\textbf n}\omega_{\textbf n}\right|\geq 1,
\end{equation}
which implies the set $\mathcal{R}_{\ell,j}$ is empty.
Therefore we always assume that
\begin{equation}\label{050801.}
|j|\leq2\prod_{\textbf n\in \mathbb{Z}^d\atop \ell_{\textbf n}\neq 0}{|\ell_{\textbf n}|\langle \textbf n\rangle}.
\end{equation}
Furthermore, we have
\begin{eqnarray}\label{032901}
\mbox{meas}\ \mathcal{R}_1\leq \gamma \sum_{\ell\in\mathbb{Z}^{\mathbb{Z}^d}}\sum_{j\in\mathbb{Z}\atop j\ satisfies \ (\ref{050801.})}m(\ell)\left(\prod_{\textbf n\in\mathbb{Z}^d}\frac{1}{1+|\ell_{\textbf n}|^3\langle\textbf n\rangle^{d+4}}\right),
\end{eqnarray}
and then
\begin{eqnarray}
&&\nonumber\sum_{\ell\in\mathbb{Z}^{\mathbb{Z}^d}}\sum_{j\in\mathbb{Z}\atop j\ satisfies \ (\ref{050801.})}m(\ell)\left(\prod_{\textbf n\in\mathbb{Z}^d}\frac{1}{1+|\ell_{\textbf n}|^3\langle\textbf n\rangle^{d+4}}\right)\\
&=&\nonumber\sum_{m\geq 1}\sum_{\ell\in\mathbb{Z}^{\mathbb{Z}^d}\atop m(\ell)=m,\ell_{\textbf n}\neq 0}\sum_{j\in\mathbb{Z}\atop j\ satisfies \ (\ref{050801.})}m\left(\prod_{\textbf n\in\mathbb{Z}^d}\frac{1}{1+|\ell_{\textbf n}|^3\langle\textbf n\rangle^{d+4}}\right)\\
&\leq&4\sum_{m\geq 1}m\prod_{\textbf n\in\mathbb{Z}^d\atop \langle\textbf n\rangle\geq m}\left(
\sum_{0\neq\ell_{\textbf n}\in\mathbb{Z}}\frac{1}{|\ell_{\textbf n}|^2\langle\textbf n\rangle^{d+3}}\right)\nonumber\\
&\leq&4\left(\sum_{0\neq i\in\mathbb{Z}}i^{-2}\right)\nonumber\left(\sum_{m\geq 1}m\sum_{\textbf n\in{\mathbb{Z}^d}\atop \langle \textbf n\rangle\geq m}\frac{1}{\langle \textbf n\rangle^{d+3}}\right)\\
&\leq\nonumber&4\left(\sum_{0\neq i\in\mathbb{Z}}i^{-2}\right)\left(\sum_{\textbf{n}\in\mathbb{Z}^d}\frac{1}{\langle \textbf n\rangle^{d+2}}\right)\\
&\leq&C_1\label{032902},
\end{eqnarray}
where $C_1$ is an absolutely positive constant. Combining (\ref{032901}) and (\ref{032902}), one has
\begin{equation}\label{040201}
\mbox{meas}\ \mathcal{R}_1\leq C_1\gamma.
\end{equation}

Define the resonant set
\begin{equation}
\mathcal{R}_2=\bigcup_{ j\in\mathbb{Z}\atop0\neq\ell\in\mathbb{Z}^{{\mathbb{Z}^d}}}\widetilde {\mathcal{R}}_{\ell,j},
\end{equation}
where
the resonant set $\widetilde{\mathcal{R}}_{\ell,j}$ is given by
\begin{eqnarray}
 \widetilde{\mathcal{R}}_{\ell,j}=\left\{\omega:\left| \sum_{\textbf{n}\in\mathbb{Z}^d}{\ell}_{\textbf n}\omega_{\textbf n}+j\right|< \frac{\gamma^3}{16} \prod_{\textbf{n}\in \mathbb{Z}^d\atop \left\|\textbf n\right\|\leq \left\|{\textbf n}_3^*(\ell)\right\|}\left(\frac{1}{1+|\ell_{\textbf n}|^3\langle \textbf n\rangle ^{d+7}}\right)^4\right\},
\end{eqnarray}
for any $0\neq \ell\in\mathbb{Z}^{\mathbb{Z}^d}$ satisfying  $0<|\ell|<\infty$ and
\begin{equation*}
\left\|{\textbf n}_2^{*}(\ell)\right\|>\left\|{\textbf n}_3^{*}(\ell)\right\|.
\end{equation*}
As (\ref{033110}), one has
\begin{equation}\label{012802}
\mbox{meas}\ \widetilde{\mathcal{R}}_{\ell,j}\leq \frac{\gamma^3}{16}\cdot m(\ell)\prod_{\textbf n\in\mathbb{Z}^d\atop \left\|\textbf n\right\|\leq  \left\|{\textbf n}_3^{*}(\ell)\right\|}\left(\frac{1}{1+|\ell_{\textbf n}|^3\langle \textbf n\rangle^{d+7}}\right)^4.
\end{equation}
Without loss of generality, we assume that
\begin{equation*}
\textbf n_1={\textbf n}_1^{*}(\ell),\qquad \textbf n_2={\textbf n}_2^{*}(\ell).
\end{equation*}
Then  one has
\begin{equation}\label{060901}
 \sum_{\textbf n\in\mathbb{Z}^d}\ell_{\textbf n}\omega_{\textbf n}=\left( \sum_{\textbf n\in\mathbb{Z}^d,\atop \textbf n\neq \textbf n_1,\textbf n_2}\ell_{\textbf n}\omega_{\textbf n}\right)+\sigma_{\textbf n_1}\omega_{\textbf n_1}+\sigma_{\textbf n_2}\omega_{\textbf n_2},
\end{equation}
where $\sigma_{\textbf n_1},\sigma_{\textbf n_2}\in\left\{-1,1\right\}$. Define $\widetilde {\ell}=(\widetilde \ell_{\textbf n})_{\textbf n\in\mathbb{Z}^d}$ by
\begin{equation*}
\widetilde \ell_{\textbf n}=\ell_{\textbf n},\qquad \mbox{if}\ \textbf n\neq \textbf n_1,\textbf n_2;
\end{equation*}
and
\begin{equation*}
\widetilde \ell_{\textbf n}=0,\qquad \mbox{if}\ \textbf n= \textbf n_1,\textbf n_2.
\end{equation*}

If $\widetilde \ell=0$, using \textbf{Momentum Conversation} (\ref{050901}) one has
\begin{equation}
\sigma_{\textbf n_1}\cdot \sigma_{\textbf n_2}=-1
\end{equation}
and
\begin{equation}
\textbf n_1=\textbf n_2,
\end{equation}
which implies $\ell=0$. Therefore we always assume that $\widetilde \ell \neq 0$.

Hence, if $\omega\in\Pi\setminus\mathcal{R}_1$ (where $\mathcal{R}_1$ is defined in (\ref{020701})) and
\begin{equation}\label{012801}
\left\|\textbf n_2^{*}(\ell)\right\|\geq \frac4{\gamma}\prod_{\textbf n\in\mathbb{Z}^d\atop \left\|\textbf n\right\|\leq  \left\|{\textbf n}_3^{*}(\ell)\right\|}\left(1+|\ell_{\textbf n}|^3\langle \textbf n\rangle^{d+4}\right),
\end{equation}
then one has
\begin{eqnarray*}
&&\left|\sum_{\textbf n\in\mathbb{Z}^d}\ell_{\textbf n}\omega_{\textbf n}+j\right|\\
&\geq&\left|\left|\left|\sum_{\textbf n\in\mathbb{Z}^d}\ell_{\textbf n}\omega_{\textbf n}\right|\right|\right|\\
&\geq&\left|\left|\left|\sum_{\textbf n\in\mathbb{Z}^d,\atop \textbf n\neq \textbf n_1,\textbf n_2}\ell_{\textbf n}\omega_{\textbf n}\right|\right|\right|-\left|\left|\left|\sigma_{\textbf n_1}\omega_{\textbf n_1}+\sigma_{\textbf n_2}\omega_{\textbf n_2}\right|\right|\right|\qquad (\mbox{by (\ref{060901})})\\
&\geq&  \gamma \prod_{\textbf n\in \mathbb{Z}^d\atop \textbf n\neq \textbf n_1,\textbf n_2}\frac{1}{1+|\ell_{\textbf n}|^3\langle \textbf n\rangle^{d+4}}-\frac2{\left\|\textbf n_2^{*}(\ell)\right\|}\\
&\geq &\frac{\gamma}{2} \prod_{\textbf n\in \mathbb{Z}^d\atop \textbf n\neq \textbf n_1,\textbf n_2}\frac{1}{1+|\ell_{\textbf n}|^3\langle \textbf n\rangle^{d+4}},
\end{eqnarray*}
where the last inequality is based on (\ref{012801}). Therefore we always assume
\begin{equation}\label{012803}
\left\|{\textbf n}_2^{*}(\ell)\right\|< \frac4{\gamma}\prod_{\textbf n\in\mathbb{Z}^d\atop \textbf n\neq \textbf n_1,\textbf n_2}\left(1+|\ell_{\textbf n}|^3\langle \textbf n\rangle^{d+4}\right):=A(\ell).
\end{equation}
In view of the following two inequalities
\begin{equation*}
\left\|{\textbf n}_1^{*}(\ell)\right\|\leq \sum_{i\geq 2}\left\|{\textbf n}_i^{*}(\ell)\right\|,
\end{equation*}
and
\begin{equation*}
\sum_{i\geq 2}\left\|{\textbf n}_i^{*}(\ell)\right\|=\sum_{\textbf n\in\mathbb{Z}^d\atop \textbf n\neq \textbf n_1}\left|\ell_{\textbf n}\right|\cdot\left\|\textbf n\right\|\leq \prod_{\textbf n\in\mathbb{Z}^d\atop \textbf n\neq \textbf n_1}\left(1+|\ell_{\textbf n}|^3\langle \textbf n\rangle^{d+4}\right),
\end{equation*}
one has
\begin{equation}\label{012804}
\left\|{\textbf n}_1^{*}(\ell)\right\|< A(\ell)\left(\prod_{\textbf n\in\mathbb{Z}^d\atop \textbf n\neq \textbf n_1,\textbf n_2}\left(1+|\ell_{\textbf n}|^3\langle \textbf n\rangle^{d+4}\right)\right):=B(\ell).
\end{equation}
Hence if
\begin{equation*}
|j|>2\left(\prod_{\textbf n\in \mathbb{Z}^d,\textbf{n}_1,\textbf n_2\atop \ell_{\textbf n}\neq 0}{|\ell_{\textbf n}|\langle \textbf n\rangle}\right)A(\ell)B(\ell),
\end{equation*}
then one has
\begin{equation}
\left| \sum_{\textbf n\in\mathbb{Z}^d}{\ell}_{\textbf n}\omega_{\textbf n}+j\right|\geq |j|-\left| \sum_{\textbf n\in\mathbb{Z}^d}{\ell}_{\textbf n}\omega_{\textbf n}\right|\geq 1,
\end{equation}
which implies the set $\widetilde{\mathcal{R}}_{\ell,j}$ is empty.
Therefore we always assume that
\begin{equation}\label{050801}
|j|\leq2\left(\prod_{\textbf n\in \mathbb{Z}^d,\textbf{n}_1,\textbf n_2\atop \ell_{\textbf n}\neq 0}{|\ell_{\textbf n}|\langle \textbf n\rangle}\right)A(\ell)B(\ell).
\end{equation}
In view of (\ref{012802}), (\ref{012803}), (\ref{012804}) and following the proof of (\ref{040201}), one has
\begin{equation}\label{012807}
\mbox{meas}\ \mathcal{R}_2\leq C_2\gamma,
\end{equation}
where $C_2$ is a positive constant.

Letting
$$\Pi_{\gamma}=\mathcal{R}_1\bigcup\mathcal{R}_2,$$
then one has
\begin{equation}\label{012806}
\mbox{meas}\ \Pi_{\gamma}\leq C\gamma,
\end{equation}
and for any $\omega\in\Pi\setminus \Pi_{\gamma}$, the inequalities (\ref{040601}) and (\ref{040602}) holds.

\end{proof}

\subsection{Proof of Lemma \ref{052302}}
\begin{proof}
Without loss of generality, it suffices to prove
 \begin{equation}\label{052501}
\left\|\frac{\partial^{2}R}{\partial q_{\textbf m}\partial {q}_{\textbf l}}\right\|^{*}_{\rho+\delta} \leq \left(\frac{12}{e\delta}\right)^2\cdot \exp\left\{\left(\frac{3600d}{\delta^2}\right)^d\cdot \exp\left\{d{\left(\frac{12d}\delta \right)}^{\frac{1}{\sigma-1}}\right\}\right\}\left\|R\right\|_{\rho}.
\end{equation}
For any $ \textbf m, \textbf l\in \mathbb{Z}^d$, one has
\begin{eqnarray}\nonumber
\frac{\partial^{2}R}{\partial q_{\textbf m}\partial {q}_{\textbf l}}&=&  {\sum_{a,k,k'\in\mathbb{N}^{\mathbb{Z}^d}}}k_{\textbf m}k_{\textbf l}R_{akk'}\prod_{\textbf n\in \mathbb{Z}^d}
I(0)^{a}q^{k-e_{\textbf m}-e_{\textbf l}}\bar q^{k'}.
\end{eqnarray}
In view of (\ref{d2}), one obtains
\begin{eqnarray}\nonumber
&&\left\|\frac{\partial^{2}R}{\partial q_{\textbf m}\partial {q}_{\textbf l}}\right\|^{*}_{\rho+\delta} \\
\nonumber &=& \underset{a,k,k'\in\mathbb{N}^{\mathbb{Z}^d}}{\sum}\left|k_{\textbf m}k_{\textbf l}R_{akk'}\right|\exp\left\{-\sum_{\textbf n\in\mathbb{Z}^d}2ra_{\textbf n} \ln^{\sigma}\lfloor \textbf n\rfloor\right\}\\
\nonumber &&\times
\exp\left\{-\left(\rho+\delta\right)\sum_{\textbf n\in\mathbb{Z}^d}\left(k_{\textbf n}+k'_{\textbf n}\right) \ln^{\sigma}\lfloor \textbf n\rfloor+\left(\rho+\delta\right)\left(\ln^{\sigma}\lfloor \textbf m\rfloor+\ln^{\sigma}\lfloor \textbf l\rfloor\right)\right\}\\
\nonumber
&\leq& \underset{a,k,k'\in\mathbb{N}^{\mathbb{Z}^d}}{\sum}|k_{\textbf m}k_{\textbf l}R_{akk'}|\exp\left\{-(\rho+\delta)\left(\sum_{\textbf n\in\mathbb{Z}^d}\left(2a_{\textbf n}+k_{\textbf n}+k'_{\textbf n}\right) \ln^{\sigma}\lfloor \textbf n\rfloor-\ln^{\sigma}\lfloor\textbf m\rfloor-\ln^{\sigma}\lfloor \textbf l\rfloor\right)\right\},
\end{eqnarray}
where the last inequality uses $r>\rho+\delta$.

In view of (\ref{042602}), one has
\begin{equation*}
\left|R_{akk'}\right|\leq \left\|R\right\|_{\rho}\exp\left\{\rho\left(\sum_{\textbf n\in\mathbb{Z}^d}\left(2a_{\textbf n}+k_{\textbf n}+k'_{\textbf n}\right)\ln^{\sigma}\lfloor \textbf n\rfloor-2\ln^{\sigma}\lfloor \textbf n_1^*(a,k,k')\rfloor\right)\right\}
\end{equation*}
and then
\begin{eqnarray}
\nonumber \left\|\frac{\partial^{2}R}{\partial q_{\textbf m}\partial {q}_{\textbf l}}\right\|^{*}_{\rho+\delta}
&\leq& \mathcal{C}\left\|R\right\|_{\rho},
\end{eqnarray}
where
\begin{equation*}
\mathcal{C}={\sum_{a,k,k'\in\mathbb{N}^{\mathbb{Z}^d}}}k_{\textbf m}k_{\textbf l}\exp\left\{-\delta\left(\sum_{i\geq1}
\ln^{\sigma}\lfloor \textbf n_{i}(a,k,k')\rfloor\right)+\delta\left(\ln^{\sigma}\lfloor \textbf m\rfloor+\ln^{\sigma}\lfloor \textbf l\rfloor\right)\right\}.
\end{equation*}
To prove the inequality (\ref{052501}) holds, it suffices to show that
\begin{equation}\label{032301}
\mathcal{C}\leq \left(\frac{12}{e\delta}\right)^2\cdot \exp\left\{\left(\frac{3600d}{\delta^2}\right)^d\cdot \exp\left\{d{\left(\frac{12d}\delta \right)}^{\frac{1}{\sigma-1}}\right\}\right\},
\end{equation}which will be discussed in the following three cases:

$\textbf{Case 1.}$ $$\max\{\lfloor \textbf m\rfloor,\lfloor \textbf l\rfloor\}\leq \lfloor\textbf n_{3}^*(a,k,k')\rfloor.$$

Then one has
\begin{equation}\label{052104}
-\delta\left(\sum_{i\geq1}
\ln^{\sigma}\lfloor \textbf n_{i}(a,k,k')\rfloor\right)+\delta\left(\ln^{\sigma}\lfloor \textbf m\rfloor+\ln^{\sigma}\lfloor \textbf l\rfloor\right)\leq -\frac{\delta}3\left(\sum_{i\geq1}
\ln^{\sigma}\lfloor \textbf n_{i}(a,k,k')\rfloor\right).
\end{equation}
Using (\ref{060320}) and (\ref{052104}), we have
\begin{eqnarray}\nonumber
\mathcal{C}
\nonumber &\leq&\underset{a,k,k'\in\mathbb{N}^{\mathbb{Z}^d}}{\sum}\left(\underset{i\geq 1}{\sum}\ln^{\sigma}\lfloor \textbf n_{i}^*(a,k,k')\rfloor\right)^{2}\exp\left\{-\frac{\delta}{3}\left(\underset{i\geq 1}{\sum}\ln^{\sigma}\lfloor \textbf n_{i}^*(a,k,k')\rfloor\right)\right\}\\
\nonumber &\leq&\sup\left\{\left(\underset{i\geq 1}{\sum}\ln^{\sigma}\lfloor \textbf n_{i}^*(a,k,k')\rfloor\right)^{2}\exp\left\{-\frac{\delta}{6}\left(\underset{i\geq 1}{\sum}\ln^{\sigma}\lfloor \textbf n_{i}^*(a,k,k')\rfloor\right)\right\}\right\}\\
&&\times\underset{a,k,k'\in\mathbb{N}^{\mathbb{Z}^d}}{\sum}\exp\left\{-\frac{\delta}{6}\left(\underset{i\geq 1}{\sum}\ln^{\sigma}\lfloor \textbf n_{i}^*(a,k,k')\rfloor\right)\right\}\nonumber\\
\label{032303} &\leq& \left(\frac{12}{e\delta}\right)^2\cdot \exp\left\{\left(\frac{3600d}{\delta^2}\right)^d\cdot \exp\left\{d{\left(\frac{12d}\delta \right)}^{\frac{1}{\sigma-1}}\right\}\right\},
\end{eqnarray}
where the last inequality follows from (\ref{042805*}) in Lemma \ref{8.6}, (\ref{041809}) in Lemma \ref{a3}
 and (\ref{122401}) in Lemma \ref{a5}.

$\textbf{Case 2.}$ $$\min\{\lfloor \textbf m\rfloor,\lfloor \textbf l\rfloor\}> \lfloor \textbf n_{3}^*(a,k,k')\rfloor.$$

In this case, one has
\begin{equation}\label{052310}
|k_{\textbf m}|+|k_{\textbf l}|\leq 2.
\end{equation}
Without loss of generality, we assume that
\begin{equation*}\label{033101}
\textbf n_1^*(a,k,k')=\textbf l,\qquad \textbf n_2^*(a,k,k')=\textbf m.
\end{equation*}
In view of (\ref{052310}), we thus obtain
\begin{eqnarray}
\nonumber \mathcal{C} &\leq& 5\sum_{a,k,k'\in\mathbb{N}^{\mathbb{Z}^d}\atop
\textbf n_1^*(a,k,k')=\textbf l,\textbf n_2^*(a,k,k')=\textbf m}\exp\left\{-
\delta\left(\underset{i\geq 3}{\sum}\ln^{\sigma}\lfloor \textbf n_{i}^*(a,k,k')\rfloor\right)\right\}\\
&\leq& 5 \exp\left\{\left(\frac{100d}{\delta^2}\right)^d\cdot \exp\left\{d{\left(\frac{2d}\delta \right)}^{\frac{1}{\sigma-1}}\right\}\right\}\label{052502},
\end{eqnarray}
where the last inequality follows from (\ref{041809}) in Lemma \ref{a3}
 and (\ref{122401}) in Lemma \ref{a5}.

$\textbf{Case 3.}$ $$\lfloor\textbf m\rfloor> \lfloor\textbf n_{3}^*(a,k,k')\rfloor,\quad \lfloor\textbf l\rfloor \leq \lfloor \textbf n_{3}^*(a,k,k')\rfloor$$ or $$\lfloor\textbf l\rfloor> \lfloor\textbf n_{3}^*(a,k,k')\rfloor,\quad \lfloor\textbf m\rfloor \leq \lfloor \textbf n_{3}^*(a,k,k')\rfloor.$$
Without loss of generality, we assume
$$\lfloor\textbf m\rfloor> \textbf n_{3}^*(a,k,k'),\quad \lfloor\textbf l\rfloor \leq \lfloor \textbf n_{3}^*(a,k,k')\rfloor$$
and $\textbf m=\textbf n_1^*(a,k,k').$
In this case one has
\begin{equation*}\label{052320}
|k_{\textbf m}|\leq 2,
\end{equation*}
and then
\begin{eqnarray}\nonumber
 \mathcal{C}&\leq& 2\sum_{a,k,k'\in\mathbb{N}^{\mathbb{Z}^d}\atop \textbf n_1^*(a,k,k')=\textbf m}k_{\textbf l}\exp\left\{-\frac
\delta2\left(\underset{i\geq 2}{\sum}\ln^{\sigma}\lfloor \textbf n_{i}^*(a,k,k')\rfloor\right)\right\}.
\end{eqnarray}
Following the proof of \textbf{Case 1.}, we have
\begin{equation}\label{052503}
\mathcal{C}\leq \left(\frac{8}{e\delta}\right)\cdot \exp\left\{\left(\frac{1600d}{\delta^2}\right)^d\cdot \exp\left\{d{\left(\frac{8d}\delta \right)}^{\frac{1}{\sigma-1}}\right\}\right\}.
\end{equation}

In view of (\ref{032303}), (\ref{052502}) and (\ref{052503}), we finish the proof of (\ref{032301}).

\end{proof}
\subsection{Proof of Lemma \ref{010}}
\begin{proof}
Let
\begin{equation*}
R_1(q,\bar q)=\sum_{a,k,k'\in\mathbb{N}^{\mathbb{Z}^d}} b_{akk'}\mathcal{M}_{akk'}
\end{equation*}
and
\begin{equation*}
R_2(q,\bar q)=\sum_{A,K,K'\in\mathbb{N}^{\mathbb{Z}^d}} B_{AKK'}\mathcal{M}_{AKK'}.
\end{equation*}
Then one has
\begin{equation*}
\{R_1,R_2\}=\sum_{a,k,k',A,K,K'\in\mathbb{N}^{\mathbb{Z}^d}}b_{akk'}B_{AKK'}\{\mathcal{M}_{akk'},\mathcal{M}_{AKK'}\},
\end{equation*}
where
\begin{eqnarray*}
\{\mathcal{M}_{akk'},\mathcal{M}_{AKK'}\}
&=&{\textbf{i}}\sum_{\textbf j\in\mathbb{Z}^d}\left(\prod_{\textbf n\neq \textbf j}I_{\textbf n}(0)^{a_{\textbf n}+A_{\textbf n}}q_{\textbf n}^{k_{\textbf n}+K_{\textbf n}}\bar{q}_{\textbf n}^{k_{\textbf n}'+K_{\textbf n}'}\right)\\
&&\times\left(\left(k_{\textbf j}K_{\textbf j}'-k_{\textbf j}'K_{\textbf j}\right)I_{\textbf j}(0)^{a_{\textbf j}+A_{\textbf j}}q_{\textbf j}^{k_{\textbf j}+K_{\textbf j}-1}\bar{q}_{\textbf j}^{k_{\textbf j}'+K_{\textbf j}'-1}\right).
\end{eqnarray*}
Given any $\alpha,\kappa,\kappa'\in\mathbb{N}^{\mathbb{Z}^d}$, the coefficient of the monomial $\mathcal{M}_{\alpha\kappa\kappa'}$ in $\left\{R_1,R_2\right\}$ is given by
\begin{equation}\label{006}
 B_{\alpha\kappa\kappa'}=\textbf{i}\sum_{{\textbf j}\in\mathbb{Z}^d}\sum_{*}\sum_{**}\left(k_{\textbf j}K_{\textbf j}'-k_{\textbf j}'K_{\textbf j}\right)b_{akk'}B_{AKK'},
\end{equation}
where
\begin{equation*}
\sum_{*}=\sum_{a,A\in\mathbb{N}^{\mathbb{Z}^d} \atop a+A=\alpha},
\end{equation*}
and
\begin{equation*}
\sum_{**}=\sum_{k,k',K,K'\in\mathbb{N}^{\mathbb{Z}^d}\atop \mbox{when}\ {\textbf n}\neq {\textbf j}, k_{\textbf n}+K_{\textbf n}=\kappa_{\textbf n},k_{\textbf n}'+K_{\textbf n}'=\kappa_{\textbf n}';\mbox{when}\ {\textbf n}={\textbf j}, k_{\textbf n}+K_{\textbf n}-1=\kappa_{\textbf n},k_{\textbf n}'+K_{\textbf n}'-1=\kappa_{\textbf n}'}.
\end{equation*}
To estimate (\ref{042704}), some simple facts are given firstly:

$\textbf{1}.$ If ${\textbf j}\notin\ \mbox{supp}\ (k+k') \bigcap\ \mbox{supp}\ (K+K')$, then
\begin{equation*}
k_{\textbf j}K_{\textbf j}'-k_{\textbf j}'K_{\textbf j}=0.
\end{equation*}
Hence we always assume ${\textbf j}\in\ \mbox{supp}\ (k+k') \bigcap\ \mbox{supp}\ (K+K')$, which implies
\begin{equation}\label{122801}
\lfloor{\textbf j}\rfloor\leq \min\left\{\lfloor{\textbf n}_1^{*}(a,k,k')\rfloor,\lfloor{\textbf n}_1^{*}(A,K,K')\rfloor\right\}.
\end{equation}
Since
\begin{equation*}
\mbox{supp}\ (\alpha,\kappa,\kappa')\subset \mbox{supp}\ (a,k,k')\bigcup \mbox{supp}\ (A,K,K'),
\end{equation*}
the following inequality always holds
\begin{equation}\label{041801}
\lfloor{\textbf n}_1^{*}(\alpha,\kappa,\kappa')\rfloor\leq \max\left\{\lfloor{\textbf n}_1^{*}(a,k,k')\rfloor,\lfloor{\textbf n}_1^{*}(A,K,K')\rfloor\right\}.
\end{equation}
Combining (\ref{122801}) and (\ref{041801}), one has
\begin{equation}\label{031901}
\ln^{\sigma}\lfloor{\textbf j}\rfloor+\ln^{\sigma}\lfloor{\textbf n}_1^{*}(\alpha,\kappa,\kappa')\rfloor-\ln^{\sigma}\lfloor{\textbf n}_1^{*}(a,k,k')\rfloor-\ln^{\sigma}\lfloor{\textbf n}_1^{*}(A,K,K')\rfloor\leq 0.
\end{equation}

$\textbf{2}.$ Note that
\begin{eqnarray}
\sum_{i\geq 1}\ln^{\sigma}\lfloor{\textbf n}_i^{*}(a,k,k')\rfloor
\geq\sum_{i\geq 3}\ln^{\sigma}\lfloor{\textbf n}_i^{*}(a,k,k')\rfloor
\geq \sum_{\textbf n\in\mathbb{Z}^d}\left(2a_{\textbf n}+k_{\textbf n}+k_{\textbf n}'\right)
\label{042604},
\end{eqnarray}
where using $\sum_{\textbf n\in\mathbb{Z}^d}\left(2a_{\textbf n}+k_{\textbf n}+k_{\textbf n}'\right)\geq 4$.

Based on (\ref{042604}), we obtain
\begin{equation}\label{042606}
\sum_{\textbf n\in\mathbb{Z}^d}(k_{\textbf n}+k_{\textbf n}')(K_{\textbf n}+K_{\textbf n}')
\leq \left(\sum_{i\geq 3}\ln^{\sigma}\lfloor{\textbf n}_i^{*}(a,k,k')\rfloor\right)\left(\sum_{i\geq 3}\ln^{\sigma}\lfloor{\textbf n}_i^{*}(A,K,K')\rfloor\right).
\end{equation}

$\textbf{3}.$ It is easy to see that
\begin{eqnarray}\label{052802}
\sum_{\textbf n\in\mathbb{Z}^d}\left(2\alpha_{\textbf n}+\kappa_{\textbf n}+\kappa'_{\textbf n}\right)=\sum_{\textbf n\in\mathbb{Z}^d}\left(2a_{\textbf n}+k_{\textbf n}+k'_{\textbf n}\right)+\sum_{\textbf n\in\mathbb{Z}^d}\left(2A_{\textbf n}+K_{\textbf n}+K'_{\textbf n}\right)-2
\end{eqnarray}
and
\begin{equation}
\sum_{i\geq 1}\ln^{\sigma}\lfloor{\textbf n}_i^{*}(\alpha,\kappa,\kappa')\rfloor=\left(\sum_{i\geq 1}\ln^{\sigma}\lfloor{\textbf n}_i^{*}(a,k,k')\rfloor\right)+\left(\sum_{i\geq 1}\ln^{\sigma}\lfloor{\textbf n}_i^{*}(A,K,K')\rfloor\right)-2\ln^{\sigma}\lfloor{\textbf j}\rfloor\label{052801}.
\end{equation}
In view of (\ref{042602}) and (\ref{001}) in Lemma \ref{005}, one has
\begin{equation}
\left|b_{akk'}\right|
\label{007}\leq\left\|R_1\right\|_{\rho-\delta_1}e^{\rho\left(\sum_{{\textbf n}\in\mathbb{Z}^d}\left(2a_{\textbf n}+k_{\textbf n}+k_{\textbf n}'\right)
\ln^{\sigma}\lfloor{\textbf n}\rfloor-2\ln^{\sigma}\lfloor {\textbf n}_1^{*}(a,k,k')\rfloor\right)}e^{-{\frac12}\delta_1
\left(\sum_{i\geq 3}\ln^{\sigma}\lfloor{\textbf n}_i^{*}(a,k,k')\rfloor\right)},
\end{equation}
and
\begin{equation}
\left|B_{AKK'}\right|\leq \label{008}\left\|R_2\right\|_{\rho-\delta_2}e^{\rho\left(\sum_{{\textbf n}\in\mathbb{Z}^d}\left(2A_{\textbf n}+K_{\textbf n}+K_{\textbf n}'\right)
\ln^{\sigma}\lfloor{\textbf n}\rfloor-2\ln^{\sigma} \lfloor{\textbf n}_1^{*}(A,K,K')\rfloor\right)}e^{-\frac12\delta_2
\left(\sum_{i\geq 3}\ln^{\sigma}\lfloor{\textbf n}_i^{*}(A,K,K\rfloor\right)}.
\end{equation}
 Using (\ref{052801}), substitution of (\ref{007}) and (\ref{008}) in (\ref{006}) gives
\begin{eqnarray*}
\label{009}\left|B_{\alpha\kappa\kappa'}\right|
\nonumber&\leq&||R_1||_{\rho-\delta_1}||R_2||_{\rho-\delta_2}\sum_{{\textbf j}\in\mathbb{Z}^d}\sum_{*}\sum_{**}\left|k_{\textbf j}K_{\textbf j}'-k_{\textbf j}'K_{\textbf j}\right|\\
&&\times e^{\rho\left(\sum_{{\textbf n}\in\mathbb{Z}^d}(2a_{\textbf n}+k_{\textbf n}+k_{\textbf n}')
\ln^{\sigma}\lfloor{\textbf n}\rfloor-2\ln^{\sigma}\lfloor {\textbf n}_1^{*}(a,k,k')\rfloor\right)}\\
&&\times e^{\rho\left(\sum_{{\textbf n}\in\mathbb{Z}^d}\left(2A_{\textbf n}+K_{\textbf n}+K_{\textbf n}'\right)
\ln^{\sigma}\lfloor{\textbf n}\rfloor-2\ln^{\sigma}\lfloor {\textbf n}_1^{*}(A,K,K')\rfloor\right)}\\
\nonumber&&\times e^{-{\frac12}\delta_1
\left(\sum_{i\geq 3}\ln^{\sigma}\lfloor{\textbf n}_i^{*}(a,k,k')\rfloor\right)}e^{-\frac12\delta_2
\left(\sum_{i\geq 3}\ln^{\sigma}\lfloor{\textbf n}_i^{*}(A,K,K)\rfloor\right)}\\
\nonumber&=&\mathcal{A}\cdot ||R_1||_{\rho-\delta_1}||R_2||_{\rho-\delta_2}e^{\rho\left(\sum_{{\textbf n}\in\mathbb{Z}^d}(2\alpha_{\textbf n}+\kappa_{\textbf n}+\kappa_{\textbf n}')
\ln^{\sigma}\lfloor{\textbf n}\rfloor-2\ln^{\sigma}\lfloor {\textbf n}_1^{*}(\alpha,\kappa,\kappa')\rfloor\right)},
\end{eqnarray*}
where
\begin{eqnarray*}
\mathcal{A}&=&\sum_{\textbf j\in\mathbb{Z}^d}\sum_{*}\sum_{**}\left|k_{\textbf j}K_{\textbf j}'-k_{\textbf j}'K_{\textbf j}\right|\\
&& \times e^{2\rho\left(\ln^{\sigma}\lfloor{\textbf j}\rfloor+\ln^{\sigma}\lfloor{\textbf n}_1^{*}(\alpha,\kappa,\kappa')\rfloor-\ln^{\sigma}\lfloor{\textbf n}_1^{*}(a,k,k')\rfloor-\ln^{\sigma}\lfloor{\textbf n}_1^{*}(A,K,K')\rfloor\right)}\\
&&\times e^{-{\frac12}\delta_1
\left(\sum_{i\geq 3}\ln^{\sigma}\lfloor{\textbf n}_i^{*}(a,k,k')\rfloor\right)}e^{-\frac12\delta_2
\left(\sum_{i\geq 3}\ln^{\sigma}\lfloor{\textbf n}_i^{*}(A,K,K)\rfloor\right)}.
\end{eqnarray*}

To show (\ref{042704}) holds, it suffices to prove
\begin{equation}\label{041802}
\mathcal{A}\leq \frac{1}{\delta_2}\cdot\exp\left\{3\left(\frac{14400d}{\delta_1^2}\right)^d\cdot \exp\left\{d{\left(\frac{24d}{\delta_1} \right)}^{\frac{1}{\sigma-1}}\right\}\right\}.
\end{equation}

Now we will prove the inequality (\ref{041802}) holds in the following two cases:

 $\textbf{Case. 1.} \ \lfloor\textbf{n}_1^{*}(\alpha,\kappa,\kappa')\rfloor\leq \lfloor\textbf n_1^{*}(A,K,K')\rfloor$.

\textbf{Case. 1.1.}\begin{equation}\label{051501} \lfloor\textbf{j}\rfloor\leq \lfloor\textbf{n}_3^{*}\left(a,k,k'\right)\rfloor.\end{equation}
 Using (\ref{051501}), one has
\begin{eqnarray}
&&\nonumber e^{2\rho\left(\ln^{\sigma}\lfloor\textbf{j}\rfloor-\ln^{\sigma}\lfloor\textbf{n}_1^{*}\left(a,k,k'\right)\rfloor\right)}
\\
&\leq&\nonumber e^{2\rho\left(\ln^{\sigma}\lfloor\textbf{n}_3^{*}\left(a,k,k'\right)\rfloor-\ln^{\sigma}\lfloor
\textbf{n}_1^{*}\left(a,k,k'\right)\rfloor\right)}
\\
&\leq& e^{\frac12\delta_1\left(\ln^{\sigma}\lfloor\textbf{n}_3^{*}\left(a,k,k'\right)\rfloor
-\ln^{\sigma}\lfloor\textbf{n}_1^{*}\left(a,k,k'\right)\rfloor\right)}\label{051001},
\end{eqnarray}
where the last inequality is based on $0<\delta_1\leq \frac14 \rho$,
and then
\begin{eqnarray}
&&\nonumber e^{2\rho\left(\ln^{\sigma}\lfloor{\textbf j}\rfloor+
\ln^{\sigma}\lfloor{\textbf n}_1^{*}(\alpha,\kappa,\kappa')\rfloor-\ln^{\sigma}\lfloor{\textbf n}_1^{*}(a,k,k')\rfloor-\ln^{\sigma}\lfloor{\textbf n}_1^{*}(A,K,K')\rfloor\right)}e^{-\frac12\delta_1\left(\sum_{i\geq 3}\ln^{\sigma}\lfloor \textbf n_i^*(a,k,k')\rfloor\right)}
\\&\leq& e^{-{\frac16\delta_1}\left(\sum_{i\geq 1}\ln^{\sigma}\lfloor\textbf{n}_i^{*}\left(a,k,k'\right)\rfloor\right)}\label{042603},
\end{eqnarray}
which follows from (\ref{041801}) and (\ref{051001}).

Note that if $\textbf j,a,k,k'$ are specified, and then $A,K,K'$ are uniquely determined.
In view of (\ref{042606}) and (\ref{042603}), we have
\begin{eqnarray*}
\mathcal{A}
&\leq&\sum_{a,k,k'\in\mathbb{N}^{\mathbb{Z}^d}}\left(\sum_{i\geq 1}\ln^{\sigma}\lfloor\textbf{n}_i^{*}\left(a,k,k'\right)\rfloor\right)\left(\sum_{i\geq 3}\ln^{\sigma}\lfloor\textbf{n}_i^{*}\left(A,K,K'\right)\rfloor\right)\\
&&\times e^{-\frac16\delta_1\left(\sum_{i\geq 1}\ln^{\sigma}\lfloor\textbf{n}_i^{*}\left(a,k,k'\right)\rfloor\right)}
e^{-\frac12\delta_2\left(\sum_{i\geq 3}\ln^{\sigma}
\lfloor\textbf{n}_i^{*}\left(A,K,K'\right)\rfloor\right)}\\
&\leq&\frac{48}{e^2\delta_1\delta_2}\sum_{a,k,k'\in\mathbb{N}^{\mathbb{Z}^d}}
e^{-\frac1{12}\delta_1\left(\sum_{i\geq 1}\ln^{\sigma}\lfloor\textbf{n}_i^{*}\left(a,k,k'\right)\rfloor\right)}\qquad \mbox{(in view of (\ref{042805*}) with $p=1$)}\\
&=&\frac{48}{e^2\delta_1\delta_2}\sum_{a,k,k'\in\mathbb{N}^{\mathbb{Z}^d}}
e^{-\frac1{12}\delta_1\left(\sum_{\textbf n\in\mathbb{Z}^d}\left(2a_{\textbf n}+k_{\textbf n}+k'_{\textbf n}\right)\ln^{\sigma}\lfloor\textbf n\rfloor\right)}\\
&\leq&\frac{48}{e^2\delta_1\delta_2}\left(\sum_{a\in\mathbb{N}^{\mathbb{Z}^d}}
e^{-\frac1{12}\delta_1\sum_{n\in\mathbb{Z}^d}2a_{\textbf n}\ln^{\sigma}\lfloor\textbf n\rfloor}\right)
\left(\sum_{k\in\mathbb{N}^{\mathbb{Z}^d}}e^{-\frac1{12}\delta_1\sum_{n\in\mathbb{Z}^d}k_{\textbf n}\ln^{\sigma}\lfloor\textbf n\rfloor}\right)^2
\\
&\leq&\frac{48}{e^2\delta_1\delta_2}\prod_{\textbf n\in\mathbb{Z}^d}\left({1-e^{-\frac16\delta_1\ln^{\sigma}\lfloor \textbf n\rfloor}}\right)^{-1}
\left({1-e^{-\frac1{12}\delta_1\ln^{\sigma}\lfloor\textbf n\rfloor}}\right)^{-2}      \\
&&\nonumber\mbox{(which is based on Lemma \ref{a3})}\\
&\leq&\frac{1}{\delta_2}\cdot\exp\left\{3\left(\frac{14400d}{\delta_1^2}\right)^d\cdot \exp\left\{d{\left(\frac{24d}{\delta_1} \right)}^{\frac{1}{\sigma-1}}\right\}\right\} ,
\end{eqnarray*}
where the last inequality uses (\ref{122401}) in Lemma \ref{a5}.

\textbf{Case. 1.2.}\begin{equation} \lfloor\textbf j\rfloor>\lfloor\textbf n_3^{*}(a,k,k')\rfloor.\end{equation}

In this case, we have
\begin{equation}\label{040801}
\textbf j\in\left\{\textbf n_1^{*}(a,k,k'),\textbf n_2^{*}(a,k,k')\right\}.
\end{equation}
Furthermore, if $2a_{\textbf j}+k_{\textbf j}+k'_{\textbf j}>2$, then $\lfloor\textbf j\rfloor\leq \lfloor\textbf n_3^{*}(a,k,k')\rfloor$, we are in $\textbf{Case. 1.1.}$. Hence in what follows, we always assume
\begin{equation*}
2a_{\textbf j}+k_{\textbf j}+k'_{\textbf j}\leq2,
\end{equation*}
which implies
\begin{equation}\label{2.24}
k_{\textbf j}+k_{\textbf j}'\leq 2.
\end{equation}
For simplicity, for $i\geq 1$ denote by $$\textbf n_i=\textbf n_i^{*}(a,k,k')$$
and
$$\textbf N_i=\textbf n_i^{*}(A,K,K').$$
From (\ref{031901}), (\ref{040801}) and (\ref{2.24}), it follows that
\begin{equation*}
\mathcal{A}\nonumber
\nonumber\leq2\sum_{a,k,k'\in\mathbb{N}^{\mathbb{Z}^d}}\left(K_{\textbf n_1}+K_{\textbf n_1}'+K_{\textbf n_2}+K_{\textbf n_2}'\right) \cdot e^{-\frac12\left(\delta_1\left(\sum_{i\geq 3}\ln^{\sigma}\lfloor\textbf  n_i\rfloor\right){+\delta_2\left(\sum_{i\geq 3}\ln^{\sigma}\lfloor\textbf N_i\rfloor\right)}\right)}.
\end{equation*}
In view of (\ref{042604}) and (\ref{052802}), we have
\begin{eqnarray}
\sum_{\textbf n\in\mathbb{Z}^d}\left(2\alpha_{\textbf n}+\kappa_{\textbf n}+\kappa_{\textbf n}'\right)
\leq\label{042701}\left(\sum_{i\geq3}\ln^{\sigma}\lfloor\textbf n_{i}\rfloor\right)+\left(\sum_{i\geq3}\ln^{\sigma}\lfloor\textbf N_{i}\rfloor\right).
\end{eqnarray}
Moreover, note that  for any $\textbf j\in\mathbb{Z}^d$,
\begin{equation}\label{030520}
 K_{\textbf j}+K'_{\textbf j}\leq \kappa_{\textbf j}+\kappa'_{\textbf j}-k_{\textbf j}-k'_{\textbf j}+2\leq\kappa_{\textbf j}+\kappa'_{\textbf j}+2.
\end{equation}
Using (\ref{042701}) and (\ref{030520}), one has
\begin{eqnarray}
\nonumber \mathcal{A}&\leq& 2\sum_{a,k,k'\in\mathbb{N}^{\mathbb{Z}^d}}\left(\kappa_{\textbf n_1}+\kappa'_{\textbf n_1}+\kappa_{\textbf n_2}+\kappa'_{\textbf n_2}+4\right)  \\
&&\label{2.26}\times e^{-\frac14\delta_1\left(\sum_{i\geq 3}\ln^{\sigma}\lfloor\textbf n_i\rfloor\right)}e^{-\frac18\delta\left(\sum_{\textbf n\in\mathbb{Z}^d}\left(2\alpha_{\textbf n}+\kappa_{\textbf n}+\kappa_{\textbf n}'\right)\right)},
\end{eqnarray}
where $\delta=\min\{\delta_1,\delta_2\}$.

\begin{rem}\label{042703} Firstly, note that $\{\textbf n_1,\textbf n_2\}\cap \mathrm{supp}\ \mathcal{M}_{\alpha\kappa\kappa'}\neq \emptyset$. Thus $\textbf n_1$  (or $\textbf n_2$) ranges in a set of cardinality no more than \begin{equation}\label{042702}\#\mathrm{supp} \ \mathcal{M}_{\alpha\kappa\kappa'}\leq \sum_{\textbf n\in\mathbb{Z}^d}\left(2\alpha_{\textbf n}+\kappa_{\textbf n}+\kappa_{\textbf n}'\right).
\end{equation}Secondly, if $(\textbf n_i)_{i\geq 3}$ and $\textbf n_1$ (resp. $\textbf n_2$) is specified, then $\textbf n_2$ (resp. $\textbf n_1$) is determined uniquely. Thirdly, if $(\textbf n_i)_{i\geq 1}$ is given, then $\textbf n(a,k,k')$ is specified, and hence $\textbf n(a,k,k')$ is specified up to a factor of
\begin{equation}\label{030203}
\prod_{\textbf n\in\mathbb{Z}^d}\left(1+l_{\textbf n}^2\right),
\end{equation}
where
$$l_{\textbf n}=\#\{j:\textbf n_{j}=\textbf n\}.$$
\end{rem}
Since $\lfloor\textbf n_1\rfloor\geq \lfloor\textbf n_2\rfloor>\lfloor\textbf n_3\rfloor$, one has
\begin{equation}\label{030202}
\prod_{\textbf n\in\mathbb{Z}^d\atop \textbf n=\textbf n_1,\textbf n_2}\left(1+l_{\textbf n}^2\right)\leq 5.
\end{equation}
Following (\ref{2.26})-(\ref{030202}), we thus obtain
\begin{eqnarray}
\nonumber \mathcal{A}
\nonumber&\leq& 60\sum_{(\textbf n_i)_{i\geq3}}\prod_{\textbf m\in\mathbb{Z}^d\atop \lfloor\textbf m\rfloor\leq \lfloor\textbf n_3\rfloor}\left(1+l_{\textbf m}^2\right)e^{-\frac14\delta_1\left(\sum_{i\geq 3}\ln^{\sigma}\lfloor\textbf n_i\rfloor\right)}\\
\nonumber&&\times\left(\sum_{\textbf n\in\mathbb{Z}^d}\left(2\alpha_{\textbf n}+\kappa_{\textbf n}+\kappa_{\textbf n}'\right)\right)e^{-\frac18\delta\left(\sum_{\textbf n\in\mathbb{Z}^d}\left(2\alpha_{\textbf n}+\kappa_{\textbf n}+\kappa_{\textbf n}'\right)\right)}\\
&\leq & \nonumber\frac{480}{e\delta}\sum_{(\textbf n_i)_{i\geq3}}
\prod_{\textbf m\in\mathbb{Z}^d\atop \lfloor\textbf m\rfloor\leq \lfloor\textbf n_3\rfloor}\left(1+l_{\textbf m}^2\right)e^{-\frac14\delta_1\left(\sum_{i\geq 3}\ln^{\sigma}\lfloor\textbf n_i\rfloor\right)}\quad \mbox{(by (\ref{042805*}))}\\
\nonumber&\leq&\frac{480}{e\delta}\sup_{(\textbf n_i)_{i\geq3}}\left(\prod_{\textbf m\in\mathbb{Z}^d\atop\left\|\textbf m\right\|\leq \left\|\textbf n_3\right\|}\left(1+l_{\textbf m}^2\right)
e^{-\frac18\delta_1\sum_{i\geq3}\ln^{\sigma}\lfloor\textbf n_i\rfloor}\right)
\\&&\times\sum_{(\textbf n_i)_{i\geq3}}
e^{-\frac18\delta_1\sum_{i\geq3}\ln^{\sigma}\lfloor\textbf n_i\rfloor}\label{031910}.
\end{eqnarray}
By (\ref{042807}), one has
\begin{equation}\label{031911}
\sup_{(\textbf n_i)_{i\geq3}}\left(\prod_{\textbf m\in\mathbb{Z}^d\atop\lfloor\textbf m\rfloor\leq \lfloor\textbf n_3\rfloor}\left(1+l_{\textbf m}^2\right)
e^{-\frac18\delta_1\sum_{i\geq3}\lfloor\textbf n_i\rfloor}\right)\leq\exp\left\{6d\left(\frac{32}{\delta_1}\right)^{\frac{1}{\sigma-1}}\cdot\exp\left\{\left(\frac1{\delta_1}\right)^{\frac1{\sigma}}\right\}\right\}
.
\end{equation}
In view of  (\ref{041809}) and (\ref{122401}), we have
\begin{equation}\label{031912}
\sum_{(\textbf n_i)_{i\geq3}}
e^{-\frac18\delta_1\sum_{i\geq3}\ln^{\sigma}\lfloor\textbf n_i\rfloor}\leq \exp\left\{\left(\frac{6400d}{\delta_1^2}\right)^d\cdot \exp\left\{d{\left(\frac{16d}{\delta_1} \right)}^{\frac{1}{\sigma-1}}\right\}\right\}.
\end{equation}
By (\ref{031910}), (\ref{031911}) and (\ref{031912}), we finish the proof of (\ref{041802}).

$\textbf{Case. 2.}\ \lfloor\textbf{n}_1^{*}(\alpha,\kappa,\kappa')\rfloor>\lfloor\textbf{n}_1^{*}(A,K,K')\rfloor.$

In view of (\ref{041801}), one has $\textbf n_1^{*}(a,k,k')=\textbf n_1^{*}(\alpha,\kappa.\kappa')$. Hence,  $\textbf n_2^{*}(a,k,k')$ is determined by $\textbf n_1^{*}(a,k,k')$ and $(\textbf n_i^{*}(a,k,k'))_{i\geq 3}$. Similar as $\textbf{Case 1.2}$, we have
\begin{eqnarray*}
\mathcal{A}
&\leq&   \frac{1}{\delta_2}\cdot\exp\left\{3\left(\frac{14400d}{\delta_1^2}\right)^d\cdot \exp\left\{d{\left(\frac{24d}{\delta_1} \right)}^{\frac{1}{\sigma-1}}\right\}\right\}.
\end{eqnarray*}
\end{proof}
\subsection{Proof of Lemma \ref{051301}}
\begin{proof}
Firstly, we will prove the inequality (\ref{N6}).
 Fixed $a,k,k'\in\mathbb{N}^{\mathbb{Z}^d}$, consider the monomial
$$\mathcal{M}_{akk'}=\prod_{\textbf n\in\mathbb{Z}^d}I_{\textbf n}(0)^{a_{\textbf n}}q_{\textbf n}^{k_{\textbf n}}\bar q_{\textbf n}^{k_{\textbf n}'}$$ satisfying $k_{\textbf n}k_{\textbf n}'=0$ for all ${\textbf n}\in\mathbb{Z}^d$. It is easy to see that $\mathcal{M}_{akk'}$ comes from some parts of the terms $\mathcal{M}_{\alpha\kappa\kappa'}$ with no assumption for $\kappa$ and $\kappa'$.
Write $\mathcal{M}_{\alpha\kappa\kappa'}$ in the form of
\begin{equation*}
\mathcal{M}_{\alpha\kappa\kappa'}=\mathcal{M}_{\alpha bll'}=\prod_{{\textbf n}\in\mathbb{Z}}I_{\textbf n}(0)^{\alpha_{\textbf n}}I_{\textbf n}^{b_{\textbf n}}q_{\textbf n}^{l_{\textbf n}}{\bar q_{\textbf n}}^{l_{\textbf n}'}
\end{equation*}
where
\begin{equation*}
b_{\textbf n}=\min\left\{\kappa_{\textbf n},\kappa_{\textbf n}'\right\},\quad l_{\textbf n}=\kappa_{\textbf n}-b_{\textbf n},\quad l_{\textbf n}'=\kappa_{\textbf n}'-b_{\textbf n}
\end{equation*}
and then
$l_{\textbf n}l_{\textbf n}'=0$ for all ${\textbf n}\in\mathbb{Z}^d$.

Express the term
\begin{equation*}
\prod_{{\textbf n}\in\mathbb{Z}^d}I_{\textbf n}^{b_{\textbf n}}=\prod_{{\textbf n}\in\mathbb{Z}^d}\left(I_{\textbf n}(0)+J_{\textbf n}\right)^{b_{\textbf n}}
\end{equation*}by the monomials of the following form
\begin{equation*}
\prod_{{\textbf n}\in\mathbb{Z}^d}I_{\textbf n}(0)^{b_{\textbf n}},
\end{equation*}
\begin{equation*}
b_{\textbf m}I_{\textbf m}(0)^{b_{\textbf m}-1}J_{\textbf m}\left(\prod_{{\textbf n}\in\mathbb{Z}^d\atop {\textbf n}\neq {\textbf m}}I_{\textbf n}(0)^{b_{\textbf n}}\right),\qquad {\textbf m}\in\mathbb{Z}^d,
\end{equation*}
\begin{equation*}
\left(\prod_{{\textbf n}\in\mathbb{Z}^d\atop \left\|{\textbf n}\right\|<\left\|{\textbf m}\right\|}I_{\textbf m}(0)^{b_{\textbf m}}\right)\left(b_{\textbf m}(b_{\textbf m}-1)I_{\textbf m}(0)^{r}J_{\textbf m}^2I_{\textbf m}^{b_{\textbf m}-r-2}\right)\left(\prod_{{\textbf n}\in\mathbb{Z}^d\atop \left\|{\textbf n}\right\|\geq {\textbf m},{\textbf n}\neq {\textbf m}}I_{\textbf n}^{b_{\textbf n}}\right),\qquad {\textbf m}\in\mathbb{Z},\ 0\leq r\leq b_{\textbf m}-2.
\end{equation*}
and
\begin{eqnarray*}
&&\left(\prod_{\left\|{\textbf n}\right\|< \left\|{\textbf m}_1\right\|}I_{\textbf n}(0)^{b_{\textbf n}}\right)\left(b_{{\textbf m}_1}I_{{\textbf m}_1}(0)^{b_{{\textbf m}_1}-1}J_{{\textbf m}_1}\right)\left(\prod_{\left\|{\textbf m}_1\right\|\leq\left\|{\textbf n}\right\|\leq \left\|{\textbf m}_2\right\|\atop {\textbf n}\neq {\textbf m}_1,{\textbf m}_2}I_{\textbf n}(0)^{b_{\textbf n}}\right)
\\
&&\nonumber\times\left(b_{{\textbf m}_2}I_{{\textbf m}_2}(0)^{r}J_{{\textbf m}_2}I_{{\textbf m}_2}^{b_{{\textbf m}_2}-r-1}\right)\left(\prod_{\left\|{\textbf n}\right\|>\left\|{\textbf m}_2\right\|}I_{\textbf n}^{b_{\textbf n}}\right)\qquad \left\|{\textbf m}_1\right\|\leq \left\|{\textbf m}_2\right\|,\ 0\leq r\leq b_{{\textbf m}_2}-1.
\end{eqnarray*}
Now we will estimate the bounds for the coefficients respectively. For any given $\textbf n\in\mathbb{Z}^d$,
\begin{equation}\label{030510}
I_{\textbf n}(0)^{a_{\textbf n}}q_{\textbf n}^{k_{\textbf n}}\bar q_{\textbf n}^{k_{\textbf n}'}=\sum_{b_{\textbf n}=\min\left\{ \kappa_{\textbf n}, \kappa_{\textbf n}'\right\}}I_{\textbf n}(0)^{\alpha_{\textbf n}+b_{\textbf n}}q_{\textbf n}^{\kappa_{\textbf n}-b_{\textbf n}}\bar q_{\textbf n}^{\kappa_{\textbf n}'-b_{\textbf n}}.
\end{equation}
Hence, one has
\begin{equation}\label{N2}
a_{\textbf n}=\alpha_{\textbf n}+b_{\textbf n},
\end{equation}
and
\begin{equation}\label{N3}
k_{\textbf n}=\kappa_{\textbf n}-b_{\textbf n},\qquad k_{\textbf n}'=\kappa_{\textbf n}'-b_{\textbf n}.
\end{equation}
Therefore, if $0\leq\alpha_{\textbf n}\leq a_{\textbf n}$ is chosen, so $b_{\textbf n},\kappa_{\textbf n},\kappa_{\textbf n}'$ are determined.

On the other hand, by (\ref{042602}) we have
\begin{eqnarray}
\nonumber\left|B_{\alpha\kappa\kappa'}\right|
&\leq&\nonumber \left\|R\right\|_{\rho}e^{\rho\left(\sum_{{\textbf n}\in\mathbb{Z}^d}\left(2\alpha_{\textbf n}+\kappa_{\textbf n}+\kappa_{\textbf n}'\right)\ln^{\sigma}\lfloor{\textbf n}\rfloor-2\ln^{\sigma}\lfloor{\textbf n}_1^{*}(\alpha,\kappa,\kappa')\rfloor\right)}\qquad\qquad\qquad\qquad\ \\
&=&\nonumber\left\|R\right\|_{\rho}e^{\rho\left(\sum_{{\textbf n}\in\mathbb{Z}^d}\left(2\alpha_{\textbf n}+\left(k_{\textbf n}+a_{\textbf n}-\alpha_{\textbf n}\right)+\left(k_{\textbf n}'+a_{\textbf n}-\alpha_{\textbf n}\right)\right)\ln^{\sigma}\lfloor{\textbf n}\rfloor-2\ln^{\sigma}\lfloor{\textbf n}_1^{*}(\alpha,\kappa,\kappa')\rfloor\right)}\quad\\
&&\nonumber{(\mbox{in view of (\ref{N2}) and (\ref{N3})})}\\
&=&\nonumber\left\|R\right\|_{\rho}e^{\rho\left(\sum_{{\textbf n}\in\mathbb{Z}^d}\left(2a_{\textbf n}+k_{\textbf n}+k_{\textbf n}'\right)\ln^{\sigma}\lfloor{\textbf n} \rfloor-2\ln^{\sigma}\lfloor{\textbf n}_1^{*}(a,k,k')\rfloor\right)}.
\end{eqnarray}
Using (\ref{030510}), one has
\begin{equation}\label{N4}
\left|B_{akk'}\right|\leq\left\|R\right\|_{\rho}\prod_{{\textbf n}\in\mathbb{Z}^d}\left(1+a_{\textbf n}\right)
e^{\rho\left(\sum_{{\textbf n}\in\mathbb{Z}^d}\left(2a_{\textbf n}+k_{\textbf n}+k_{\textbf n}'\right)\ln^{\sigma}\lfloor{\textbf n}\rfloor-2\ln^{\sigma}\lfloor {\textbf n}_1^{*}(a,k,k')\rfloor\right)}.
\end{equation}
In view of (\ref{051302}) and (\ref{N4}), we have
\begin{equation}\label{N5}
\left\|R_0\right\|_{\rho+\delta}^{+}
\leq\mathcal{C}\left\|R\right\|_{\rho},
\end{equation}
where
\begin{equation}
\mathcal{C}=\prod_{{\textbf n}\in\mathbb{Z}^d}\left(1+a_{\textbf n}\right)e^{-\delta\left(\sum_{{\textbf n}\in\mathbb{Z}^d}\left(
2a_{\textbf n}+k_{\textbf n}+k_{\textbf n}'\right)\lfloor{\textbf n}\rfloor-2\lfloor{\textbf n}_1^{*}(a,k,k')\rfloor\right)}.
\end{equation}
Now it suffices to prove that
\begin{equation}\label{052004}
\mathcal{C}\leq \exp\left\{10d\left(\frac{10}{\delta}\right)^{\frac 1{\sigma-1}}\cdot\exp\left\{\left(\frac{10}\delta\right)^{\frac1\sigma}\right\}
\right\}
\end{equation}in the following three cases.

\textbf{Case 1.} $\lfloor{\textbf n}_1^{*}(a,k,k')\rfloor=\lfloor{\textbf n}_3^{*}(a,k,k')\rfloor.$

Using (\ref{001}), one has
\begin{eqnarray}
\mathcal{C}
&\leq&\nonumber \prod_{\textbf n\in\mathbb{Z}^d}\left(1+a_{\textbf n}\right)e^{-\frac12{\delta}\sum_{i\geq3}\ln^{\sigma}\lfloor{\textbf n}_i^{*}(a,k,k')\rfloor}\\
&\leq&\nonumber\prod_{\textbf n\in\mathbb{Z}^d}\left(1+a_{\textbf n}\right)^{-\frac16{\delta}\sum_{i\geq1}\ln^{\sigma}\lfloor{\textbf n}_i^{*}(a,k,k')\rfloor}\\
&=&\nonumber\prod_{{\textbf n}\in\mathbb{Z}^d}\left(1+a_{\textbf n}\right)e^{-\frac16\delta\sum_{{\textbf n}\in\mathbb{Z}^d}
\left(2a_{\textbf n}+k_{\textbf n}+k_{\textbf n}'\right)\ln^{\sigma}\lfloor{\textbf n}\rfloor}\\
&\leq&\nonumber\prod_{{\textbf n}\in\mathbb{Z}^d}\left((1+a_{\textbf n})e^{-\frac13{\delta}a_{\textbf n}\ln^{\sigma}\lfloor\textbf n\rfloor}\right)\\
&\leq&
\exp\left\{3d\left(\frac{6}{\delta}\right)^{\frac 1{\sigma-1}}\cdot\exp\left\{\left(\frac6\delta\right)^{\frac1\sigma}\right\}
\right\}\label{030205},
\end{eqnarray}
where the last inequality is based on (\ref{042807}).

\textbf{Case 2.} $\lfloor{\textbf n}_1^{*}(a,k,k')\rfloor>\lfloor{\textbf n}_2^{*}(a,k,k')\rfloor=\lfloor{\textbf n}_3^{*}(a,k,k')\rfloor.$

In this case, we have
\begin{eqnarray}
\nonumber \mathcal{C}&=&\prod_{\lfloor\textbf n\rfloor\leq \lfloor{\textbf n}_2^{*}(a,k,k')\rfloor}\left(1+a_{\textbf n}\right)e^{-\delta\left(\sum_{{\textbf n}\in\mathbb{Z}^d}\left(
2a_{\textbf n}+k_{\textbf n}+k_{\textbf n}'\right)\ln^{\sigma}\lfloor{\textbf n}\rfloor-2\ln^{\sigma}\lfloor{\textbf n}_1^{*}(a,k,k')\rfloor\right)}\\
&\leq&\prod_{\lfloor\textbf n\rfloor\leq \lfloor{\textbf n}_2^{*}(a,k,k')\rfloor}\left(1+a_{\textbf n}\right)e^{-\frac12\delta\sum_{i\geq3}\ln^{\sigma}\lfloor\textbf n_i^{*}(a,k,k')\rfloor}\nonumber\quad \mbox{(by (\ref{001}))}
\\
&\leq&\nonumber\prod_{\lfloor\textbf n\rfloor\leq \lfloor{\textbf n}_2^{*}(a,k,k')\rfloor}\left(1+a_{\textbf n}\right)e^{-\frac14\delta\sum_{i\geq2}\ln^{\sigma}\lfloor\textbf n_i^{*}(a,k,k')\rfloor}\\
&\leq&\label{030206} \exp\left\{3d\left(\frac{8}{\delta}\right)^{\frac 1{\sigma-1}}\cdot\exp\left\{\left(\frac8\delta\right)^{\frac1\sigma}\right\}
\right\},
\end{eqnarray}
where the last inequality follows from the proof of (\ref{030205}).

\textbf{Case 3.} $\lfloor{\textbf n}_2^{*}(a,k,k')\rfloor>\lfloor{\textbf n}_3^{*}(a,k,k')\rfloor.$

In this case, we have
\begin{eqnarray}
\mathcal{C}&\leq&2\prod_{\lfloor\textbf n\rfloor\leq \lfloor{\textbf n}_3^{*}(a,k,k')\rfloor}\left(1+a_{\textbf n}\right)e^{-\delta\left(\sum_{{\textbf n}\in\mathbb{Z}^d}\left(
2a_{\textbf n}+k_{\textbf n}+k_{\textbf n}'\right)\ln^{\sigma}\lfloor{\textbf n}\rfloor-2\ln^{\sigma}\lfloor{\textbf n}_1^{*}(a,k,k')\rfloor\right)}\nonumber\\
&\leq&2\prod_{\lfloor\textbf n\rfloor\leq \lfloor{\textbf n}_3^{*}(a,k,k')\rfloor}\left(1+a_{\textbf n}\right)e^{-\frac12\delta\sum_{i\geq3}\ln^{\sigma}\lfloor\textbf n_i^{*}(a,k,k')\rfloor}\nonumber
\\
&\leq&2\exp\left\{3d\left(\frac{4}{\delta}\right)^{\frac 1{\sigma-1}}\cdot\exp\left\{\left(\frac4\delta\right)^{\frac1\sigma}\right\}
\right\},\label{041301}
\end{eqnarray}
where the last inequality follows from the proof of (\ref{030205}).

In view of (\ref{030205}), (\ref{030206}) and (\ref{041301}), we finish the proof of (\ref{052004}).

Similarly, we get
\begin{eqnarray*}
&&\left\|R_1\right\|_{\rho+\delta}^{+},\left\|R_2\right\|_{\rho+\delta}^{+}\leq\exp\left\{10d\left(\frac{10}{\delta}\right)^{\frac 1{\sigma-1}}\cdot\exp\left\{\left(\frac{10}\delta\right)^{\frac1\sigma}\right\}
\right\}\left\|R\right\|_{\rho},
\end{eqnarray*}
Then we have
\begin{equation*}
\left\|R\right\|_{\rho+\delta}^{+}\leq \exp\left\{10d\left(\frac{10}{\delta}\right)^{\frac 1{\sigma-1}}\cdot\exp\left\{\left(\frac{10}\delta\right)^{\frac1\sigma}\right\}
\right\}\left\|R\right\|_{\rho}.
\end{equation*}

On the other hand, the coefficient of $\mathcal{M}_{\alpha bll'}$ increases by at most a factor $$\left(\sum_{\textbf n\in\mathbb{Z}^d}(\alpha_{\textbf n}+b_{\textbf n})\right)^2.$$ Then one has
\begin{eqnarray}
\nonumber\left\|R\right\|_{\rho+\delta}
&\leq&\nonumber\left\|R\right\|_{\rho}^{+}\left(\sum_{\textbf n\in\mathbb{Z}^d}(\alpha_{\textbf n}+b_{\textbf n})\right)^2e^{-\delta\left(\sum_{\textbf n\in\mathbb{Z}^d}\left(2a_{\textbf n}+k_{\textbf n}+k_{\textbf n}'\right)\ln^{\sigma}\lfloor\textbf n\rfloor-2\ln^{\sigma}\lfloor\textbf n_1^{*}(a,k,k')\rfloor\right)}\\
&\leq&\nonumber\left\|R\right\|_{\rho}^{+}\left(2\sum_{i\geq 3}\ln^{\sigma}\lfloor\textbf n_i^{*}(a,k,k')\rfloor\right)^2 e^{-\frac12\delta\sum_{i\geq3}\ln^{\sigma}\lfloor\textbf n_1^{*}(a,k,k')\rfloor}\quad (\mbox{in view of (\ref{042604})})\\
&\leq&\nonumber\frac{64}{e^2\delta^2}||R||_{\rho}^{+},
\end{eqnarray}
where the last inequality is based on Lemma \ref{8.6} with $p=2$.

\end{proof}

\section*{Acknowledgments}
   The author is supported by NNSFC No. 12071053.

\end{document}